\documentclass[11pt,oneside,reqno]{amsart}

\usepackage[margin=1in]{geometry}  

\usepackage{graphicx} 
\usepackage{comment}
\usepackage{amssymb,amsbsy,amsmath,amsfonts,amssymb,amscd}
\usepackage{mathrsfs}
\usepackage{epsfig, graphicx}
\usepackage{graphics,epsfig}
\usepackage{mathrsfs}
\usepackage{amsthm}
\usepackage{mathtools}
\usepackage[dvipsnames]{xcolor}
\usepackage{enumitem}
\usepackage{graphicx}
\usepackage{subcaption}
\newtheorem{remark}{Remark}

\newtheorem{proposition}{Proposition}

\usepackage{algorithm}
\usepackage{algpseudocode}
\usepackage{caption}
\usepackage{subcaption}
\usepackage{cleveref}

\usepackage{algorithm}
\usepackage{algpseudocode}

\newcommand{\half}{\frac{1}{2}}
\newcommand{\energy}{\mathcal{E}}

\newcommand{\Amat}{\mathsf{A}}

\newcommand{\Dmat}{\mathsf{D}}
\newcommand{\Imat}{\mathsf{I}}

\newcommand{\Jmat}{\mathsf{J}}

\newcommand{\rd}{\mathrm{d}}

\newcommand{\calS}{\mathcal{S}}

\newcommand{\Wass}{\mathcal{W}}

\newcommand{\bx}{{\boldsymbol{x}}}

\newcommand{\by}{{\boldsymbol{y}}}
\newcommand{\bv}{{\boldsymbol{v}}}
\newcommand{\bu}{\boldsymbol{u}}
\newcommand{\bS}{\boldsymbol{S}}

\newcommand{\bT}{\boldsymbol{T}}
\newcommand{\bE}{\boldsymbol{E}}
\newcommand{\bJ}{\boldsymbol{J}}

\newcommand{\RR}{{\mathbb{R}}}

\newcommand{\np}{{n+1}}
\newcommand{\norm}[1]{ \| #1 \|}

\newcommand{\R}{\mathbb{R}}

\renewcommand{\paragraph}[1]{\noindent{\bf #1}}

\newcommand{\lw}[1]{\textcolor{cyan}{{Wang: #1}}}

\title[Deep Kinetic JKO schemes]{Deep Kinetic JKO schemes for Vlasov-Fokker-Planck Equations}
\author[Lee]{Wonjun Lee}
\email{lee.8222@osu.edu}
\address{Mathematics Department, The Ohio State University, OH 43210.}
\author[Wang]{Li Wang}
\email{liwang@umn.edu}
\address{School of Mathematics, University of Minnesota, MN 55455.}
\author[Li]{Wuchen Li}
\email{wuchen@mailbox.sc.edu}
\address{Department of Mathematics, University of South Carolina, Columbia, SC 29208.}

\begin{document}

\maketitle
\begin{abstract}
We introduce a deep neural network–based numerical method for solving kinetic Fokker Planck equations, including both linear and nonlinear cases. Building upon the conservative dissipative structure of Vlasov-type equations, we formulate a class of generalized minimizing movement schemes as iterative constrained minimization problems: the conservative part determines the constraint set, while the dissipative part defines the objective functional. This leads to an analog of the classical Jordan–Kinderlehrer–Otto (JKO) scheme for Wasserstein gradient flows, and we refer to it as the kinetic JKO scheme. To compute each step of the kinetic JKO iteration, we introduce a particle-based approximation in which the velocity field is parameterized by deep neural networks. The resulting algorithm can be interpreted as a kinetic-oriented neural differential equation that enables the representation of high-dimensional kinetic dynamics while preserving the essential variational and structural properties of the underlying PDE. We validate the method with extensive numerical experiments and demonstrate that the proposed kinetic JKO–neural ODE framework is effective for high-dimensional numerical simulations.       
\end{abstract}

\section{Introduction}
Many dynamical models in complex systems from physics, chemistry, and engineering exhibit a conservative–dissipative structure \cite{RevModPhys.87.593,DuongPeletierZimmer2014ConservativeDissipativeKramers,ottinger2005beyond}: part of the evolution is reversible and preserves an energy-like quantity, while the remaining part is irreversible and drives the system toward equilibrium through entropy dissipation. Examples range from kinetic equations with collisions to viscous and diffusive continuum models to thermodynamic systems. 

In this paper, we focus on designing numerical schemes for a particular class of conservative–dissipative systems, namely the kinetic Fokker–Planck equation, which plays a central role in plasma physics and dynamical density functional theory \cite{DDFT}. A representative form is:
\begin{align} \label{VFP00}
     \partial_t f + \bv \cdot \nabla_\bx f - \nabla_\bx \phi \cdot \nabla_\bv f = \nabla_\bv \cdot (\bv f + \nabla_\bv f)\,,
\end{align}
where $f$ is a probability density function defined on the phase space $(\bx,\bv) \in \Omega_\bx \times \R^d$. $\phi(\bx)\in\mathbb{R}^1$ represents an external potential function, which may either be prescribed a priori or determined self-consistently. The left-hand side of \eqref{VFP00} describes the reversible Hamiltonian dynamics, under which the Hamiltonian functional
\begin{equation*}
    \mathcal H[f] = \int \half|\bv|^2 f + \phi(\bx) f \rd \bx \rd \bv\,,
\end{equation*}
is conserved. In contrast, the right-hand side represents the dissipative dynamics, which drives the system toward local equilibrium by dissipating a relative entropy functional
 \begin{equation*}
    \mathcal S[f] = \int \half|\bv|^2 f + f \log f \rd \bx \rd \bv\,.
    \end{equation*}
There is also a free energy $\energy[f]$, defined as the sum of $\mathcal{H}$ and the negative Shannon–Boltzmann entropy $\int f \log f ,\mathrm{d}\bx ,\mathrm{d}\bv$, which dissipates along the dynamics \eqref{VFP00}.

Simulating system \eqref{VFP00} remains a challenging task. 
One central difficulty, as emphasized in \cite{ottinger2005beyond}, is the preservation of the underlying conservative-dissipative structure under suitable time- and spatial-discretization schemes, which is essential for ensuring reliability and physical fidelity in numerical simulations. Another obstacle arises from the high dimensionality of the models, which are posed in a seven-dimensional phase space: 3 in $\bx$, 3 in $\bv$, and 1 in $t$. As a result, grid-based methods rapidly become computationally infeasible due to the \emph{curse of dimensionality}. Instead, machine learning based approaches, such as neural ODEs \cite{chen2018neural,haber2018stable,PNAS}, have emerged as promising alternatives for approximating high-dimensional probability densities.

In particular, for equations that contain only the dissipative part and does not depend on the spatial variable, 
that is, $\partial_t f(t,\bv) = \nabla_v \cdot(\bv f(t,\bv) + \nabla_\bv f(t,\bv))$ or more generally 
\begin{align} \label{cont}
    \partial_t f(t,\bv) = \nabla_\bv \cdot (f \bu[f])\,,
\end{align}
where $\bu[f]\in\mathbb{R}^d$ is a velocity field that may depend linearly or nonlinearly on $f$, several learning-enhanced methods have been developed to address the challenge of high dimensionality \cite{boffi2023probability, mokrov2021large, li2023self, HuLiuWangXu2024EnergeticVariationalNNGF, lu2024score,ilin2025transport,huang2025score, huang2024jko, XuChengXie2023JKOiFlow}. Among these, transport-based approaches have attracted noticeable attention \cite{boffi2023probability, li2023self}.
The key idea starts from the particle representation of \eqref{cont}. Specifically, let $\{\bv_i(t)\}_{i=1}^N$ be a set of i.i.d. samples drawn from $f(t,\cdot)$. Their evolution is then governed by
\begin{align} \label{particle}
    \frac{\rd}{\rd t} \bv_i(t) = \bu(f(\bv_1,\cdots , \bv_N))\,.
\end{align}
The task is to infer the velocity field $\bu$ directly from the particle ensemble, without explicitly accessing the density $f$. Approaches such as score matching \cite{boffi2023probability} and velocity matching \cite{li2023self} have been developed to accomplish this goal. 

Another approach also relies on the particle representation \eqref{particle}. However, instead of updating the particles explicitly, it learns the velocity field implicitly through the minimizing movement scheme, also known as the JKO scheme \cite{JKO}. More precisely, one views 
\eqref{cont} as the Wasserstein gradient flow 
\begin{align*}
    \partial_t f(t,\bv) = \nabla_\bv \cdot (f \nabla_\bv \frac{\delta \mathcal S}{\delta f})\,,
\end{align*}
where $\frac{\delta}{\delta f}$ is the $L^2$ first variation operator w.r.t. function $f$, and with a slight abuse of notation, $\mathcal{S}$ is the relative entropy for $f$ only in  variable $\bv$. Given a time approximation function  
 $f^{n}(\cdot) \approx f(t^{n}, \cdot)$, the JKO scheme determines $f^{n+1} (\cdot) \approx f(t^{n+1}, \cdot)$ with $t^{n+1} = (n+1)\Delta t$ through:
\begin{align} \label{JKO00}
    f^{n+1} \in \arg \min_f \Big\{   {\Wass_2^2(f, f^n)} +  2 \Delta t \calS(f)\Big\}\,.
\end{align}
Here, $ \Wass_2(f, f^n)$ denotes the Wasserstein-2 distance between $f$ and $f^n$. The scheme in \eqref{JKO00} can be interpreted as a time-implicit scheme of gradient descent, where the functional on the right-hand side acts as the Wasserstein proximal operator associated with $\calS$. Using the dynamic formulation of the Wasserstein distance \cite{BB00}, \eqref{JKO00} can be reformulated as \cite{LiWang22}
\begin{equation} \label{classicalJKO}
\begin{dcases}
& (f,\bu)=\arg\inf_{f,\bu}~ \int_0^1\int_{\RR^d} f |\bu|^2 \rd \bv \rd \tau  + 2\Delta t\calS(f(1,\cdot))\,,
\\ & \textrm{s.t.} \quad  \partial_\tau f + \nabla_\bv \cdot (f \bu) = 0\,, \quad f(0,\bv) = f^n (\bv)\,.
\end{dcases}
\end{equation}
Let 
\[
\frac{\rd }{\rd \tau} \bT(\tau, \bv) = \bv (\tau, \bT(\tau,\bv))\,, \qquad \bT(0,\bv) = \bv\,.
\]
denote the flow map associated with the velocity field $\bu$. Using a particle approximation of the density, \eqref{classicalJKO} can be translated into the following optimization problem over the velocity field:
\begin{equation} \label{JKO-NN1} \hspace{-1cm} 
    \begin{dcases}
    \min_{\bu}  \frac{1}{N}\sum_{j=1}^N \bigg[  \int_0^1 |\bu (\tau, \bT(\tau,\bv_j ))|^2 \rd \tau  + 2 \Delta t V(\bT(1, \bv_j)) \bigg]  + 2\Delta t \calS(\bT_\# f^n)\,,
    \\ \text{s.t.} ~~\frac{\rd}{\rd \tau} \bT(\tau, \bv_j) = \bu  (\tau, \bT(\tau,\bv_j)), \qquad \bT(0,\bv_j) = \bv_j^n \,.
    \end{dcases}
\end{equation}
Further details can be found in \cite{lee2024deep}.

When considering the full system \eqref{VFP00}, an additional requirement is the conservation of the Hamiltonian. While general methods such as velocity matching \cite{li2023self, ZhouOsherLi2025FokkerPlanckMFCSNF} can be applied directly to \eqref{VFP00} in the extended phase space $(\bx,\bv)$, the resulting models do not necessarily preserve the desired conservative–dissipative structure. 
In this work, we instead exploit the intrinsic decomposition of kinetic equations into conservative Hamiltonian and dissipative gradient flow components. Our main idea is to introduce a variant of the formulation \eqref{classicalJKO}, in which the conservative dynamics are encoded in the constraint set, while the dissipative dynamics determine the objective functional. We refer to this variational formulation as the {\it kinetic JKO scheme}. In this way, both the conservative and dissipative structures of the system are naturally preserved. Moreover, the JKO framework guarantees the monotonic decay of the associated Lyapunov functional, such as the free energy.



We note that generalized JKO schemes for kinetic equations have been studied in \cite{huang2000variational,DuongPeletierZimmer2014ConservativeDissipativeKramers}, particularly within the framework of general Equation for Non-equilibrium reversible–irreversible coupling (GENERIC) \cite{ottinger2005beyond} and macroscopic fluctuation theory (MFT) \cite{MFT}. However, that work primarily focuses on the theoretical properties, such as Kantorovich formulations, of variational problems and does not aim to provide efficient and scalable numerical algorithms. In contrast, our approach makes this feasible by combining a particle representation with a density evolution parameterized through a neural ODE. Another related line of work is based on the idea of score-matching. For instance, \cite{lu2024score,boffi2023probability,ilin2025transport,huang2025score} apply score matching methods to approximate mean-field Fokker-Planck equations. Some aspects of convex analysis for the score-based matching optimization method for two-layer neural network functions have been studied \cite{WangChenPilanciLi2024OptimalNNWassersteinGradient}. Additionally, \cite{HuLiuWangXu2024EnergeticVariationalNNGF,JinLiuWuYeZhou2025PWGF,Lindsey2025MNE,LiuLiZhaZhou2022NeuralParametricFokkerPlanck,ZUO2026114501} have studied neural projected schemes, which can be viewed as semi-time-discretizations of the neural network-based computational method. 
Compared with the above-mentioned existing works, 
our approach is variational: time discretization is performed at the level of a constrained minimization problem. This structure provides enhanced numerical stability and ensures compatibility with energy dissipation.

The rest of the paper is organized as follows. Section~\ref{sec2} reviews the kinetic Fokker–Planck equation and its conservative–dissipative decomposition from a free energy dissipation perspective. We illustrate this decomposition for both linear and nonlinear cases. Section~\ref{sec3} introduces the generalized kinetic JKO scheme in Eulerian coordinates and then studies its implementation in Lagrangian coordinates. We employ kinetic neural ODEs to approximate the logarithm of the density and discuss the numerical properties of the resulting method, including the free energy dissipation. Section~\ref{sec4} presents numerical experiments that verify the scheme's accuracy, stability, and scalability. The paper concludes in Section~\ref{sec5}.

\section{Kinetic Fokker-Planck equation}\label{sec2}
In this section, we review the kinetic Fokker–Planck equation, in both its linear and nonlinear forms, and highlight its conservative–dissipative structure.

\subsection{Linear case}
Consider the Vlasov-Fokker-Planck equation
\begin{align} \label{VFP0}
     \partial_t f + \bv \cdot \nabla_\bx f - \nabla_\bx \phi \cdot \nabla_\bv f = \epsilon \nabla_\bv \cdot (\bv f + \nabla_\bv f)\,,
\end{align}
where $f:=f(t,\bx,\bv)$ denotes the phase-space number density of charged particles at time $t$, location $\bx$ and with velocity $\bv$. The function $\phi(\bx)$ is a given potential, with $\bx \in \R^d$ or a compact subset $\Omega_\bx \subset \mathbb{R}^d$,  and $\bv\in\mathbb{R}^d$. 
Unless stated otherwise, when $\Omega_\bx$ is compact,  we impose periodic boundary conditions in the spatial variable $\bx$. And $\epsilon>0$ models the strength of the collision. 

We begin by rewriting \eqref{VFP0} in a form that clearly reveals its conservative–dissipative structure. Specifically, it can be written as:
\begin{align} \label{VFP1}
    \partial_t f =  \underbrace{\nabla_{\bx,\bv} \cdot \left[  f \Jmat\nabla_{\bx,\bv} \frac{\delta \energy}{\delta f}\right]}_{\text{conservative}}+ \underbrace{\nabla_{\bx,\bv} \cdot \left[  f \Dmat
\nabla_{\bx,\bv} \frac{\delta \energy}{\delta f}\right]}_{\text{dissipative}}\,, 
\end{align}
where the free energy is defined as 
\begin{align} \label{energy0}
    \energy[f] := \int_{\RR^d}\!\int_{\Omega_{\bx}} \frac{|\bv|^2}{2} f + \phi(\bx) f + f \log f \rd \bx \rd \bv \,.
\end{align}
Here $\Jmat \in\mathbb{R}^{2d\times 2d}$ is an anti-symmetric matrix and $\Dmat \in \mathbb{R}^{2d\times 2d}$ is a symmetric positive semi-definite matrix 
\begin{equation*} 
  \Dmat =\begin{pmatrix}
0 & 0\\
0 & \epsilon\Imat
\end{pmatrix}\,, \quad \Jmat =\begin{pmatrix}
0 & -\Imat\\
\Imat & 0
\end{pmatrix}\,.  
\end{equation*}
And $\Imat$ is an identity matrix in $\mathbb{R}^{d\times d}$, and the operator $\frac{\delta}{\delta f}$ is the $L^2$ first variation w.r.t. function $f$.

Indeed, the form \eqref{VFP1} follows from a direct computation. Note that 
\begin{align*}
    \frac{\delta \energy[f]}{\delta f} ={\frac{|\bv|^2}{2}  + \phi(\bx)  + \log f+1} \, ,
\end{align*}
and 
\begin{equation*}
\nabla_{\bx,\bv}   \frac{\delta \energy[f]}{\delta f}=\begin{pmatrix}
\nabla_\bx\phi(\bx)+\nabla_\bx\log f
\\
\bv+\nabla_\bv\log f 
\end{pmatrix}\,.
\end{equation*}
Clearly, 
\begin{equation*}
\begin{split}
\nabla_{\bx, \bv}\cdot \left[  f \Jmat \nabla_{\bx,\bv} \frac{\delta \energy}{\delta f}\right]
=& \nabla_{\bx, \bv}\cdot \left[ f\begin{pmatrix}
-\bv-\nabla_{\bv}\log f\\
\nabla_\bx\phi(\bx)+\nabla_{\bx}\log f
\end{pmatrix}
\right]\\
=& \nabla_{\bx, \bv}\cdot \left[ f\begin{pmatrix}
-\bv\\
\nabla_\bx\phi(\bx)
\end{pmatrix}
\right]\,,
\end{split}
\end{equation*}
where in the last equality, we use the fact that $f\nabla_{\bx,\bv}\log f= \nabla_{\bx,\bv} f$, and   
\begin{equation*}
 \nabla_{\bx,\bv}\cdot \left[f \begin{pmatrix}-\nabla_{\bv}\log f\\
\nabla_{\bx}\log f
\end{pmatrix}\right]=  \nabla_{\bx,\bv}\cdot \left[\begin{pmatrix}-\nabla_{\bv} f\\
\nabla_{\bx} f
\end{pmatrix}\right]=0\,.
\end{equation*}
Similarly, we have
\begin{equation*}
\begin{split}
   \nabla_{\bx,\bv} \cdot \left[  f \Dmat
\nabla_{\bx,\bv} \frac{\delta \energy}{\delta f}\right]=& \nabla_{\bx,\bv} \cdot \left[  f \begin{pmatrix}
0\\
\bv+\nabla_{\bv}\log f
\end{pmatrix}
\right]\\
=&\nabla_{\bv}\cdot\big(f\bv+\nabla_{\bv}f\big)\,,
\end{split}
\end{equation*}
where we use the fact that $f\nabla_{\bv}\log f=\nabla_{\bv}f$.

Since $\Dmat + \Jmat$ is non-degenerate, it follows from \eqref{VFP1} that the equilibrium occurs when $\nabla_{x,v}\frac{\delta}{\delta f}\energy[f]=0$, which implies that $\frac{\delta}{\delta f}\energy[f]=\textrm{Constant}$. Therefore, the invariant distribution of equation \eqref{VFP0} satisfies 
\begin{equation*}
    f^\infty(\bx,\bv)=\frac{1}{Z}e^{-(\frac{|\bv|^2}{2}+\phi(\bx))}\,,
\end{equation*}
where $Z$ is a normalization constant satisfying 
\begin{equation*}
   Z=\int_{\RR^d}\!\int_{\Omega_{\bx}} e^{-(\frac{|\bv|^2}{2}+\phi(\bx))}\rd\bx \rd\bv<+\infty\,. 
\end{equation*}

The following free energy dissipation result holds. 
\begin{proposition} \label{thm-diss}
Suppose $f(t,\cdot)$ is the solution of equation \eqref{VFP0}. Then 
\begin{equation} \label{0414}
    \frac{d}{dt}\energy[f](t,\cdot)=-\int_{\RR^d}\int_{\Omega_\bx}(\nabla_{\bx,\bv}\frac{\delta}{\delta f}\energy[f], \Dmat \nabla_{\bx,\bv}\frac{\delta}{\delta f}\energy[f]) f \rd \bx \rd \bv \leq 0\,,
\end{equation}
where $\energy[f]$ is defined in \eqref{energy0}. 
\end{proposition}
\begin{proof}
The proof is by a direct computation, with an emphasis on the semi-positive definite matrix $\Dmat$ and the symplectic matrix $\Jmat$. 
\begin{equation*}
\begin{split}
   \frac{d}{dt}\energy[f]=& \int_{\RR^d}\int_{\Omega_\bx} \partial_t f\cdot \frac{\delta}{\delta f}\energy \rd\bx \rd\bv\\
=&\int_{\RR^d}\int_{\Omega_\bx}\Big(\nabla_{\bx,\bv} \cdot \left[  f \Jmat \nabla_{\bx,\bv} \frac{\delta \energy}{\delta f}\right]+\nabla_{\bx,\bv} \cdot \left[  f \Dmat
\nabla_{\bx,\bv} \frac{\delta \energy}{\delta f}\right]\Big)\cdot \frac{\delta}{\delta f}\energy \rd\bx \rd\bv\\
=&-\int_{\RR^d}\int_{\Omega_\bx}
\Big((\nabla_{\bx,\bv}\frac{\delta}{\delta f}\energy, \Jmat \nabla_{\bx,\bv}\frac{\delta}{\delta f}\energy)+(\nabla_{\bx,\bv}\frac{\delta}{\delta f}\energy, \Dmat \nabla_{\bx,\bv}\frac{\delta}{\delta f}\energy)\Big) f \rd\bx \rd \bv\\
=&-\int_{\RR^d}\int_{\Omega_\bx}
(\nabla_{\bx,\bv}\frac{\delta}{\delta f}\energy, \Dmat \nabla_{\bx,\bv}\frac{\delta}{\delta f}\energy) f \rd\bx \rd \bv\leq 0\,, 
\end{split}
\end{equation*}
where we use the fact that 
$\Jmat$ is a symplectic matrix such that $u^{T}\Jmat u=0$, for any vector $u\in\mathbb{R}^{2d}$. This finishes the proof. 
\end{proof}

We note that equation \eqref{VFP1} is {\it not} a gradient flow due to the presence of the anti-symmetric matrix $J$. Nevertheless, in the following sections, we will introduce a variational formulation of \eqref{VFP1} that maintains the associated free energy dissipation property \eqref{0414}.

\subsection{Nonlinear case}
Consider the Vlasov-Poisson-Fokker-Planck system: 
\begin{equation} \label{VPFP}
\begin{dcases}
    & \partial_t f + \bv \cdot \nabla_\bx f - \nabla_\bx \phi \cdot \nabla_\bv f = \epsilon \nabla_\bv \cdot (\bv f + T_0\nabla_\bv f)\,, \quad \bx \in \Omega \subset \RR^{d_x}\,,
     \\ & -\Delta_\bx \phi = \rho(t, \bx) - h\,, \qquad \rho(t,\bx) = \int_{\RR^d} f\rd \bv\,.
\end{dcases}
\end{equation}
This equation is commonly used to model a single-species plasma, describing the motion of electrons in a static ion background, where $h$ is a constant representing the background charge. It satisfies the global neutrality condition
\[
\int_{\Omega_\bx} h  \rd\bx = \int_{\Omega_\bx} \rho(t,\bx) \rd\bx\,.
\]
Here, $T_0$ is a given background temperature (consider the ion as a steady thermal bath). 

Compared to \eqref{VFP0}, the main difference here is that $\phi$ is not given apriori, but obtained self-consistently through the Poisson equation. In theory, one can obtain it using the Green's function $G$, such that
\begin{align}
    \phi(\bx) = \int_{\Omega_\bx} G(\bx-\by) (\rho(t,\by) - h) \rd \by\,. 
\end{align}
Similar to Proposition~\ref{thm-diss}, this nonlinear system \eqref{VPFP} also satisfies free energy dissipation. 

\begin{proposition}
    For equation \eqref{VPFP}, define the Lyapunov functional: 
    \begin{align*}
        \energy[f] := \int_{\RR^d}\int_{\Omega_\bx} \left( \half|\bv|^2 f + T_0f \log f \right) \rd \bx \rd \bv + \half \int_{\Omega_\bx}|\nabla_\bx \phi|^2  \rd \bx \,. 
    \end{align*}
Then $\frac{\rd}{\rd t} \energy[f](t) \leq 0$.
\end{proposition}
\begin{proof}
Note first that
\begin{align*}
    \half \int_{\Omega_\bx} |\nabla_\bx \phi|^2 \rd \bx &= -\half \int_{\Omega_\bx} \phi \Delta_\bx \phi  \rd \bx 
    = \half \int_{\Omega_\bx} \phi(\rho- h) \rd \bx 
    \\ & = \half \int_{\Omega_\bx} \int_{\Omega_\bx} (\rho(t,\bx)-h) G(\bx-\by)(\rho(t,\by)-h) \rd \bx \rd \by \,.
\end{align*}
Then \eqref{VPFP} can also be written in the form of \eqref{VFP1}, with 
\[
\frac{\delta }{\delta f} \left[ \half \int_{\Omega_\bx} \norm{\nabla_\bx \phi}^2 \rd \bx \right] = G \ast (\rho - h) = \phi(\bx)\,.
\]
The rest proof follow the same argument as in Proposition~\ref{thm-diss}. 
\end{proof}
  

We note a closely related variant of \eqref{VPFP}, the Vlasov–Ampere–Fokker–Planck system, which is given by
\begin{equation} \label{VAFP}
\begin{dcases}
    & \partial_t f + \bv \cdot \nabla_\bx f + \bE \cdot \nabla_\bv f = \epsilon \nabla_\bv \cdot (\bv f + T_0\nabla_\bv f)\,,
     \\ & \partial_t \bE(t,x) = -\bJ(t,\bx)\,, \qquad \bJ(t,\bx) = \int_{\RR^d}  \bv f\rd \bv\,.
\end{dcases}
\end{equation}
One can readily show that if the initial electric field $\bE(0,\bx)$ satisfies 
\begin{align*} 
    \nabla_\bx \cdot \bE(0,\bx) = \rho(0,\bx) - h \,,
\end{align*}
then, since $\bE$ evolves according to \eqref{VAFP}, this relation is preserved for all time:
\[
\nabla_\bx \cdot \bE(t,\bx) = \rho(t,\bx) - h\,.
\]
Indeed, integrating the $f$ equation in \eqref{VAFP} with respect to 
$\bv$ yields the continuity equation $\partial_t \rho = - \nabla_{\bx}\cdot\bJ$. Taking the divergence of the evolution equation for $\bE$ then gives  $\partial_t \nabla_\bx \cdot\bE = \partial_t \rho$. Since $\nabla_\bx \cdot\bE$ and $\rho- h$ satisfy the same evolution equation and agree initially, they coincide for all $t\geq 0$.
If we write $\bE = - \nabla_\bx \phi$, then we recover the Poisson equation.  Consequently, it also satisfies the same free energy dissipation property. 

In fact, the Vlasov–Ampère–Fokker–Planck system also enjoys the following dissipation property for a general initial condition $\bE(0,x)$.
\begin{proposition}
    For equation \eqref{VAFP}, define the energy functional
    \begin{align*}
        \energy[f] := \int_{\RR^d}\int_{\Omega_\bx} \left( \half|\bv|^2 f + T_0f \log f  \right) \rd \bx \rd \bv + \half \int_{\Omega_\bx}  |\bE(t,x)|^2  \rd \bx \,, 
    \end{align*}
    then    
    \begin{equation*}            
    \frac{\rd}{\rd t} \energy[f] =-\epsilon\int_{\RR^d} \!\int_{\Omega_\bx} \|T_0\nabla_{\bv}\log \frac{f}{e^{-\frac{\|\bv\|^2}{2T_0}}}\|^2 f \rd\bx \rd\bv\leq 0\,. 
    \end{equation*}
\end{proposition}
\begin{proof}
Let $(f, \bE)$ be the solution to \eqref{VAFP}. Then we can rewrite it as:
\begin{equation*}
 \partial_t f 
 = \epsilon T_0\nabla_\bv \cdot \left( f  \nabla_\bv \log\frac{f}{e^{-\frac{|\bv|^2}{2T_0}}}\right)-\nabla_{\bx,\bv}\cdot\left(f\begin{pmatrix}
  \bv\\ \bE(t,x)
 \end{pmatrix}\right)\,, 
\end{equation*}
where we used $f\nabla_{\bv}\log f=\nabla_{\bv} f$. Then we have 
\begin{equation*}
   \begin{split}
\frac{d}{dt}  \energy[f] =&\int_{\RR^d}\!\int_{\Omega_{\bx}} (T_0 \log f+T_0+\frac{1}{2}|\bv|^2)\cdot \partial_tf \rd\bx\rd\bv +\int_{\Omega_\bx} \bE \cdot \partial_t \bE \rd\bx  \\
=&- \epsilon \int_{\RR^d}\!\int_{\Omega_{\bx}} \| T_0 \nabla_\bv \log\frac{f}{e^{-\frac{|\bv|^2}{2T_0}}}\|^2 f \rd\bx\rd\bv \\
&+\int_{\RR^d}\!\int_{\Omega_{\bx}} \Big(T_0 \nabla_\bx\log f\cdot \bv+ T_0 \nabla_\bv\log f\cdot \bE+\bv \cdot \bE \Big) f \rd\bx\rd\bv-\int \int  \bE\cdot \bv f\rd\bx\rd\bv\\
=&- \epsilon\int_{\RR^d}\!\int_{\Omega_{\bx}} \| T_0 \nabla_\bv \log\frac{f}{e^{-\frac{|\bv|^2}{2T_0}}}\|^2 f \rd\bx\rd\bv\,. 
   \end{split}
\end{equation*}
In the last equality of the above formula, we use the fact that
\begin{equation*}
\begin{split}
&   \int_{\RR^d}\!\int_{\Omega_{\bx}} \Big(T_0 \nabla_\bx\log f\cdot \bv+ T_0 \nabla_\bv\log f\cdot \bE\Big) f \rd\bx\rd\bv\\
= & \int_{\RR^d}\!\int_{\Omega_{\bx}} \Big(T_0 \nabla_\bx f\cdot \bv+ T_0 \nabla_\bv f\cdot \bE\Big) \rd\bx\rd\bv= -\int_{\RR^d}\!\int_{\Omega_{\bx}} \Big(T_0  f \nabla_\bx\cdot\bv+ T_0  f\nabla_\bv\cdot \bE\Big) \rd\bx\rd\bv
=0\,.
\end{split}
\end{equation*}
This finishes the proof. 
\end{proof}

\section{Generalized JKO formulation and particle method}\label{sec3}
In this section, we propose the generalized JKO formulation for kinetic Fokker-Planck equations. We then apply a variational particle method in which the velocity is constructed using the neural ODE method.  
\subsection{Generalized JKO formulation} 
We propose the following generalized dynamical JKO formulation of \eqref{VFP0} or \eqref{VFP1}. Given $f^n(\bx, \bv)$ at time $t = t^n:= n \Delta t$, we obtain $f^{n+1}(\bx,\bv)$ as follows:
\begin{equation} \label{JKO0}
    \begin{dcases}
    \min_{f, \bu} \frac{1}{2\Delta t} \int_0^1 \int_{\RR^d}\!\int_{\Omega_{\bx}} (|\bu|^2 f)(\tau, \bx, \bv) \rd \bx \rd \bv \rd \tau  + \epsilon \mathcal{E}[f(1, \cdot, \cdot)]\,,
    \\ \hspace{3cm} \mathcal{E}[f] :=\int_{\RR^d}\!\int_{\Omega_{\bx}}
    [\frac{|\bv|^2}{2} + \phi(\bx)] f + f \log f \rd \bx \rd \bv\,,
    \\ \text{s.t.} \quad  \partial_\tau f + \Delta t \bv \cdot \nabla_\bx f - \Delta t\nabla_\bx \phi \cdot \nabla_\bv f +  \nabla_\bv \cdot (\bu f) = 0, \quad f(0, \bx, \bv) = f^n(\bx, \bv)\,.
    \end{dcases}
\end{equation}
Compared to the vanilla dynamical JKO for Wasserstein gradient flow \eqref{classicalJKO}, the main difference lies in the additional term $ \Delta t (\bv \cdot \nabla_\bx f - \nabla_\bx \phi \cdot \nabla_\bv f) $ in the constraint PDE. This term arises from the conservative part of the PDE, specifically the term involving $\Jmat$ in \eqref{VFP1}. This is intuitive because the conservative part is what prevents the original equation from being a gradient flow.  

We also note that if we only consider the dissipative term in \eqref{VFP1}, the associated energy should be $\int_{\RR^d}\!\int_{\Omega_{\bx}}
    \frac{|\bv|^2}{2}  f + f \log f \rd \bx \rd \bv$. However, it is important to retain the full energy in the formulation. Including the full energy does not alter the optimality condition, but it is crucial for ensuring that the algorithm preserves the energy dissipation property stated in the following proposition.

\begin{proposition}[Free energy dissipation for \eqref{JKO0}] \label{prop: energy dissipation}
Let $f^{n+1}(\bx,\bv)$ be the solution to \eqref{JKO0}, then we have 
\[
\mathcal{E}[f^{n+1}] \leq \mathcal{E}[f^n]\,,
\]
where $\energy$ is defined in \eqref{energy0}.
\end{proposition}
\begin{proof}
    Note first that 
    \[
    \mathcal{E}[f^{n+1}] \leq \energy [\tilde f]\,,
    \]
    where $\tilde f$ is obtained by solving
    \[
    \partial_\tau \tilde f + \Delta t \bv \cdot \nabla_\bx \tilde f - \Delta t\nabla_\bx \phi \cdot \nabla_\bv \tilde f  = 0\,, \quad \tilde f(0, \bx, \bv) = f^n(\bx, \bv)\,.
    \]
    Now we claim that $\energy[\tilde f] = \mathcal{E}[f^n]$. To establish this, we need to show that the energy $\mathcal{E}[f]$ is conserved along the flow:
    \begin{align} \label{feqn}
        \partial_t f + \bv \cdot \nabla_\bx f - \nabla_\bx \phi \cdot \nabla_\bv f =0\,. 
    \end{align}
    Let
    \begin{align*}
    H[f]:=\int_{\RR^d}\!\int_{\Omega_{\bx}} [\frac{|\bv|^2}{2} + \phi(\bx)] f \rd \bx \rd \bv  \,, \quad H_0[f] = \int_{\RR^d}\!\int_{\Omega_{\bx}} f \log f \rd \bx\rd \bv \,,
    \end{align*}
    then $\mathcal{E}[f] = H[f] + H_0[f] $. First, it is direct to check that $\frac{\rd}{\rd t}H[f] = 0$ for $f$ satisfying \eqref{feqn}. To check $H_0[f]$, we see that
    \begin{align*}
        \frac{\rd H_0}{\rd t} &= \int \frac{\delta H_0}{\delta f} \nabla_{\bx, \bv} \cdot \left( f \Jmat \nabla_{\bx, \bv} \frac{\delta E}{\delta f}\right)\rd\bx\rd\bv
        = \int \frac{\delta E}{\delta f} \nabla_{\bx, \bv} \cdot \left( f \Jmat \nabla_{\bx, \bv} \frac{\delta H_0}{\delta f}\right)\rd\bx\rd\bv
        \\ & = \int \frac{\delta E}{\delta f} \nabla_{\bx, \bv} \cdot \left(  \Jmat \nabla_{\bx, \bv} f \right)\rd\bx\rd\bv = 0\,,
    \end{align*}
    where the last equality is due to $f\nabla_{\bx,\bv}\log f=\nabla_{\bx,\bv}f$ and the anti-symmetry property of matrix $\Jmat$. 
\end{proof}

Now we translate \eqref{JKO0} using the flow mapping function. Denote 
\[
\bT(\tau, \bx^n,\bv^n)=(\bx(\tau; \bx^n), \bv(\tau; \bv^n))^\top\,,
\]
where $\bx(\tau, \bx^n )$, $\bv(\tau; \bx^n)$ are solutions to the following ODE system 
\begin{equation*}
\begin{aligned}
    ~~\frac{\rd}{\rd \tau} \bx(\tau) &= \Delta t \bv(\tau)\,, & \bx(0) = \bx^0\,,
    \\ \qquad \frac{\rd}{\rd \tau} \bv(\tau) &= - \Delta t \nabla_\bx \phi(\bx(\tau)) +  \bu (\tau,\bx(\tau),\bv(\tau))\,,  & \bv(0) = \bv^0\,.\\ 
    \end{aligned}
\end{equation*}

We state the following proposition, which we term kinetic neural ODEs. It may be regarded as a generalization of the neural ODE framework developed in \cite{lee2024deep}.

\begin{proposition}[Kinetic neural ODEs]
The following equation holds:
\begin{align} \label{push2}
\frac{\rd}{\rd \tau} \log |\det \nabla_{\bx,\bv} \bT(\tau, \bx, \bv)|=   \nabla_{\bv(\tau)}\cdot \bu(\tau,  \bT(\tau,\bx, \bv))\,.
\end{align}
\end{proposition}
\begin{proof}
Equation \eqref{push2} is derived from the Monge-Ampere equation. 
Assume that $\bT(\tau, \bx, \bv)$ is invertible, i.e. $|\mathrm{det}\nabla_{\bx,\bv} \bT(\tau, \bx, \bv)|\neq 0$.
Denote 
\[
\xi(\bx(\tau), \bv(\tau)) :=\frac{\rd}{\rd \tau}\bT(\tau, \bx,\bv)= (\bv(\tau) \Delta t, -\nabla_{\bx} \phi(\bx(\tau)) \Delta t +  \bu(\tau, \bx(\tau), \bv(\tau)))\,.
\]
Then
\begin{equation*}
\begin{split}
\frac{\rd}{\rd \tau} \log \det |\nabla_{\bx,\bv} \bT(\tau, \bx,\bv)| =&\mathrm{tr}\Big((\nabla_{\bx,\bv} \bT(\tau, \bx,\bv))^{-1}\frac{\rd}{\rd \tau}\nabla_{\bx,\bv} \bT(\tau,\bx,\bv)\Big)\\
 =&  \mathrm{tr}\Big((\nabla_{\bx,\bv} \bT(\tau, \bx,\bv))^{-1}\nabla_{\bx, \bv}\frac{\rd }{\rd \tau} \bT(\tau,\bx,\bv)\Big)\\
  =&  \mathrm{tr}\Big((\nabla_{\bx,\bv} \bT(\tau, \bx,\bv))^{-1} \nabla_{\bx,\bv} \xi(\bx(\tau), \bv(\tau))
  \Big)\\
  =& \nabla_{{\bx(\tau),\bv(\tau)}} \cdot \xi(\bT(\tau, \bx, \bv))\\
  =&\nabla_{{\bv(\tau)}}\cdot \bu(\tau, \bT(\tau, \bx, \bv))\,,
\end{split}
\end{equation*}
where the first equality uses the Jacobi identity, and the second last equality uses the chain rule. 
\end{proof}
We approximate $\bu$ by $\bu_\theta$ with $\theta \in \Theta$ being the parameters. We assume that the family $\{\bu_\theta\}_{\theta \in \Theta}$ is sufficiently expressive to approximate a broad class of vector fields on $\mathbb{R}^{d} \times \mathbb{R}^d$.
Since the inner time is discretized using a single time step, we employ a fully connected neural network that takes $(\bx,\bv) \in \mathbb{R}^{d} \times \mathbb{R}^d$ as input and outputs a vector in $\mathbb{R}^d$. A typical architecture has the form
\[
\bu_\theta(t,\bx,\bv)
=
W_L \sigma\!\left(
W_{L-1} \sigma\!\left(
\cdots
\sigma\!\left(W_1 (\bx,\bv) + b_1\right)
\cdots
\right) + b_{L-1}
\right) + b_L\,,
\]
where $L$ denotes the number of layers, $W_k$ and $b_k$ are trainable weight matrices and bias vectors, and $\sigma(z) = \max(\alpha z, z)$ is the \texttt{LeakyReLU} activation function with slope parameter $\alpha>0$. Let $m_0 = 2d+1$ denote the input dimension, corresponding to the concatenated variable $(\bx,\bv)$. Let $m_1,\dots,m_{L-1}$ denote the hidden layer widths and $m_L=d$ the output dimension. Then the network parameters satisfy
\[
W_k \in \mathbb{R}^{m_k \times m_{k-1}}\,, 
\qquad 
b_k \in \mathbb{R}^{m_k}\,,
\qquad 
k=1,\dots,L\,.
\]
The parameter vector $\theta$ consists of all trainable weights and biases,
\[
\theta = \{W_1,b_1,\dots,W_L,b_L\}\,,
\]
which may be identified with a vector in $\mathbb{R}^{d_{\theta}}$, where $d_\theta = \sum_{k=1}^{L} (m_k m_{k-1} + m_k)$ denotes the total number of parameters in the network.
Then \eqref{JKO0} rewrites as:
\begin{equation} \label{JKO1}
    \begin{dcases}
    \min_{\theta}  
    \frac{1}{2\Delta t}\int_0^1 \int_{\RR^d}\!\int_{\Omega_{\bx}}
    |\bu_\theta(\tau, \bx(\tau; \bx^n),\bv(\tau; \bv^n))|^2 f^n(\bx^n, \bv^n) \rd \bx^n \rd \bv^n \rd \tau \\
    \hspace{1cm}   +\int_{\RR^d}\!\int_{\Omega_{\bx}} [\frac{|\bv|^2}{2} + \phi(\bx) + \log f](1,\bx(1;\bx^n,\bv^n), \bv(1;\bx^n,\bv^n))  f^n(\bx^n, \bv^n)  \rd \bx^n \rd \bv^n\,, 
    \\
    \text{s.t.} 
    ~~\frac{\rd}{\rd \tau} \bx(\tau) = \Delta t  \bv(\tau)\,, \qquad \bx(0) = \bx^n\,,
    \\ \qquad \frac{\rd}{\rd \tau} \bv(\tau) = - \Delta t \nabla_\bx \phi(\bx(\tau)) +  \bu_\theta(\tau, \bx(\tau),\bv(\tau))\,,  \qquad \bv(0) = \bv^n\,,\\
    \qquad \frac{\rd}{\rd \tau}\log |\text{det} \nabla_{\bx,\bv}\bT (\tau)|=  \nabla_\bv\cdot \bu_\theta( \tau, \bx(\tau),\bv(\tau))\,,\quad 
    \log |\text{det} \nabla_{\bx,\bv}\bT(0)| = 0\,,
    \\
    \qquad f(\tau, \bx(\tau), \bv(\tau)) = \frac{f^n(\bx^n, \bv^n)}{|\det \nabla_{\bx,\bv} \bT(\tau)|}\,.
    \end{dcases}
\end{equation}

\subsection{Particle method}
We now discretize \eqref{JKO1} using particles in space and symplectic integrator in time. Given i.i.d. $N$ samples $\{(\bx_p^n, \bv_p^n)\}_{p=1}^{N_p} \sim f^n(\bx,\bv)$, we solve the following optimization problem:
\begin{equation} \label{JKO2}
    \begin{dcases}
    \min_{\theta}  
    \frac{\Delta t }{2}\frac{1}{N_p}\sum_{p=1}^{N_p} |\bu_\theta(\bx_p^n, \bv_p^n)|^2 
    + \frac{\epsilon}{N_p}\sum_{p=1}^{N_p} \left[ \half |\bv_p^{n+1}|^2 + \phi(\bx_p^{n+1}) + \log f^{n+1}(\bx_p^{n+1}, \bv_p^{n+1}) \right]\,,
    \\
    \text{s.t.}~~~
    \bx_p^\np  = \bx_p^n + \bv_p^\np \Delta t\,,
    \\ \qquad \bv_p^\np = \bv_p^n - \nabla_\bx \phi(\bx_p^n)\Delta t + \bu_\theta(\bx_p^n, \bv_p^n) \Delta t \,,
    \\
    \qquad \log |\text{det} \nabla_{\bx,\bv}\bT(\bx_p^{n+1},\bv_p^{n+1}) |=\nabla_\bv\cdot \bu_\theta( \bx_p^n, \bv_p^n)  \Delta t\,,
    \\
    \qquad f^{n+1}(\bx_p^{n+1}, \bv_p^{n+1}) = \frac{f^n(\bx_p^n, \bv_p^n)}{|\det \nabla_{\bx,\bv} \bT(\bx_p^{n+1}, \bv_p^{n+1})|}\,.
    \end{dcases}
\end{equation}

It is worth noting that the update in the constraint can be replaced by a more general form:
\begin{align} \label{symplect00}
     \left(\begin{array}{c} \bx_p^{n+1}\\ \bv_p^{n+1} \end{array}\right) = \bS (\bx_p^n, \bv_p^n) + \left(\begin{array}{c} 0\\  \bu_\theta(\bx_p^n, \bv_p^n) \Delta t\end{array}\right) \,,
\end{align}
where
\begin{align} \label{SS}
   \bS: (\bx_p^n, \bv_p^n) \mapsto ( \bx_p^{*},  \bv_p^*)\,, 
\end{align}
is an {\it explicit} symplectic mapping. This generalization does not violate the update formula for the log determinant or the density that appeared in the third and fourth equations in the constraint above. 

Below we list two simple choices for $\bS$. For a more comprehensive discussion of symplectic integrators, we refer to \cite{hairer2006geometric}: 
\begin{itemize}
    \item[1)] Sympletic Euler method: 
\begin{equation} \label{symEuler2}
      \begin{dcases}
 \bx_p^*  = \bx_p^n + \bv_p^n \Delta t\,,
\\ 
 \bv_p^*  = \bv_p^n -\nabla_\bx \phi(\bx_p^*) \Delta t\,;
      \end{dcases}
\end{equation}
\item[2)] Stormer–Verlet method:
\begin{equation} \label{Verlet}
     \begin{dcases}
    \bv_p^{n+1/2} = \bv_p^n -\frac{\Delta t}{2} \nabla_\bx \phi(\bx_p^n) \,,
    \\  \bx_p^* = \bx_p^n + \Delta t \bv_p^{n+1/2}\,,
    \\  \bv_p^* = \bv_p^{n+1/2} - \frac{\Delta t}{2} \nabla_x \phi( \bx_p^*)\,.
     \end{dcases}
\end{equation}
    \end{itemize}
    
We verify the optimality condition for \eqref{JKO2} to ensure that the constrained optimization problem indeed yields the correct optimizer $\bu_\theta$. 
Denote 
\begin{align}  \label{0417}
J[\bu_\theta]&:= 
 \frac{1}{N_p}\sum_{p=1}^{N_p}  \left[\frac{ \Delta t }{2} |  \bu_\theta(\bx_p^{n}, \bv_p^{n})|^2 
    + \frac{\epsilon}{2} |\bv_p^{n+1}|^2 + \epsilon\phi(\bx_p^{n+1}) +\epsilon \log f^{n+1}(\bx_p^{n+1}, \bv_p^{n+1}) \right]\,.
\end{align}
By substituting the constraint into the objective function in \eqref{0417}, we have 
\begin{align*}
    \frac{\partial}{\partial \bu_\theta(\bx_p^n,\bv_p^n)}J &= \frac{1}{N_p} \left[\Delta t   \bu_\theta(\bx_p^{n}, \bv_p^{n}) + \epsilon \bv_p^{n+1} \Delta t + \epsilon \nabla_\bx \phi(\bx_p^{n+1}) \Delta t^2  \right.
    \\ & \left.\qquad +  \epsilon \nabla_\bx \log f^{n+1}(\bx_p^{n+1}, \bv_p^{n+1}) \Delta t^2  + 
    \epsilon\nabla_\bv \log f^{n+1}(\bx_p^{n+1}, \bv_p^{n+1}) \Delta t \right]\,.
\end{align*}
Setting $\frac{\partial }{\partial \bu_\theta}J = 0$, we have 
\begin{align}\label{eq:drift-error}
    \bu_\theta(\bx_p^{n}, \bv_p^{n}) = - \epsilon\bv_p^{n+1} - \epsilon \nabla_\bv \log f^{n+1}(\bx_p^{n+1}, \bv_p^{n+1}) + \mathcal {O} (\Delta t)\,, 
\end{align}
which is consistent with the expected drift form from the dissipative part of  \eqref{VFP0}.


We also note that Proposition~\ref{prop: energy dissipation} can be extended to the Lagrangian formulation. In particular, we have the following proposition. 
\begin{proposition}
 Let $\bx(\tau; \bx^n)$ and $\bv(\tau; \bv^n)$ be the solution to  \eqref{JKO1}, define
    \begin{align*}
    \mathcal{E}(\tau) & = \int_{\RR^d}\!\int_{\Omega_{\bx}} \big[\half |\bv(\tau; \bv^n)|^2   + \phi(\bx(\tau; \bx^n)) + \log f(\bx(\tau; \bx^n),\bv(\tau; \bv^n) ) \big] f^n(\bx^n, \bv^n)\rd \bx^n \rd \bv^n\,.
    \end{align*}
    Then we have $ \mathcal{E}(1) \leq \mathcal{E}(0)$.
\end{proposition}
\begin{proof}
Starting from \eqref{JKO1}, we first observe that $\mathcal{E}(1) \leq  \mathcal{\tilde E}(1)$, where $\mathcal{\tilde E}$ is obtained by setting $\bu_\theta$ to zero. More precisely, define 
 \begin{align*}
    \mathcal{\tilde E}(\tau) & = \int_{\RR^d}\!\int_{\Omega_{\bx}} \big[\half | \tilde\bv(\tau; \bv^n)|^2   + \phi(\tilde \bx(\tau; \bx^n)) + \log f(\tilde\bx(\tau; \bx^n), \tilde\bv(\tau; \bv^n) ) \big] f^n(\bx^n, \bv^n)\rd \bx^n \rd \bv^n\,,
    \end{align*}
where 
\begin{equation} \label{0928}
   \begin{cases}{}
   & \frac{\rd}{\rd\tau}{\tilde\bx}(\tau) = \Delta t  \tilde \bv(\tau)\,, \qquad \tilde \bx(0) = \bx^n\,,
    \\& \frac{\rd}{\rd \tau}{\tilde \bv}(\tau)  = - \Delta t \nabla_\bx \phi( \tilde\bx(\tau))\,,   \qquad \tilde\bv(0) = \bv^n\,,
    \\ & \frac{\rd}{\rd \tau}\log |\text{det} \nabla_{\bx,\bv}\bT (\tau)|  =0\,,\quad 
    \log |\text{det} \nabla_{\bx,\bv}\bT(0)| = 0\,,
    \\ & f(\tau, \tilde\bx(\tau)\,, \tilde\bv(\tau))  = \frac{f^n(\bx^n, \bv^n)}{|\det \nabla_{\bx,\bv} \bT(\tau)|}\,.
    \end{cases}
\end{equation}
It is obvious that
\[
\int_{\RR^d}\!\int_{\Omega_{\bx}} [\half |\tilde\bv(\tau)|^2 + \phi(\tilde\bx(\tau)) ] f^n(\bx^n, \bv^n) \rd \bx^n \rd \bv^n\,,
\]
is preserved along the dynamics. We also claim that 
\[
\int \log f(\tilde \bx(\tau), \tilde \bv(\tau) ) f^n(\bx^n, \bv^n)\rd \bx^n \rd \bv^n\,,
\]
is not changing, and this is guaranteed by the fact that $|\text{det} \nabla_{\bx,\bv}T(\tau)| \equiv 1$ along the dynamics \eqref{0928}. Then the result concludes because $\mathcal{\tilde E}(1) = \mathcal{\tilde E}(0) = \mathcal{E}(0)$.
\end{proof}

\begin{remark}
We emphasize that free energy dissipation at the semi-discrete level does not automatically carry over to the fully discrete scheme, since the symplectic update \eqref{SS} is not exactly energy-conserving. Nevertheless, the symplectic discretization preserves energy up to a small error and exhibits stable long-time behavior. 
\end{remark}

\begin{algorithm}[ht!]
\caption{Deep Kinetic JKO scheme based on \eqref{JKO2}}
\label{alg:jko2_time_marching}
\begin{algorithmic}[1]

\Require Terminal time $T>0$, time step $\Delta t>0$, number of particles $N_p$, collision strength $\epsilon>0$,
initial particles $\{(\bx_p^0,\bv_p^0,f^0)\}_{p=1}^{N_p}$, number of inner optimization steps $K$, a parametrized function $\bu_\theta$ with parameters $\theta$.

\Ensure Approximate particle state at time $T$.

\State Set $M=\lceil T/\Delta t\rceil$.

\For{$n=0,1,\dots,M-1$}

    \State Compute the potential $\phi^n$ and the gradient vector field $-\nabla_\bx \phi^n$ from $\{\bx_p^n\}_{p=1}^{N_p}$.

    \State Initialize network parameters $\theta \leftarrow \theta^{(0)}$.

    \For{$k=1,\dots,K$}
        \State Perform one optimization step to approximately solve \eqref{JKO2} and update $\theta$.
    \EndFor

    \For{$p=1,\dots,N$}

        \State Update velocity:
        \[
        \bv_p^{n+1}
        =
        \bv_p^n
        -
        \nabla_\bx \phi^n(\bx_p^n)\Delta t
        +
        \bu_\theta(\bx_p^n,\bv_p^n)\Delta t.
        \]

        \State Update position:
        \[
        \bx_p^{n+1}
        =
        \bx_p^n
        +
        \bv_p^{n+1}\Delta t.
        \]

        \State Update density:
        \[
        f^{n+1}(\bx_p^{n+1},\bv_p^{n+1})
        =
        \frac{f^{n}(\bx_p^{n},\bv_p^{n})}
        {\left|\det \nabla_{\bx,\bv}\bT(\bx_p^{n+1},\bv_p^{n+1})\right|}.
        \]

    \EndFor

\EndFor

\State \Return $\{(\bx_p^M,\bv_p^M,f^M)\}_{p=1}^{N_p}$.

\end{algorithmic}
\end{algorithm}
\begin{remark}\label{rem:logdet_autodiff}
In practice, the Jacobian determinant appearing in \eqref{JKO2} is approximated through the divergence of the neural control field,
\[
\log\lvert\det\nabla_{\bx,\bv}\bT(\bx,\bv)\rvert
\;\approx\;
\Delta t\,\nabla_{\bv}\cdot \bu_\theta(\bx, \bv)\,,
\]
as specified by the constraint in \eqref{JKO2}. The divergence $\nabla_{\bv}\cdot \bu_\theta$ is evaluated by automatic differentiation within the neural network framework. More precisely, each partial derivative $\partial u_{\theta,i}/\partial v_i$ is computed by backpropagation, and the divergence is obtained by summation over the velocity components. This approach avoids the explicit construction of the full Jacobian matrix and enables an efficient and scalable evaluation of the entropy production term for large particle ensembles and high-dimensional velocity spaces.
\end{remark}

\subsection{Discussion on other formulations}
\subsubsection{Operator splitting}
It is worth noting that \eqref{JKO2} with the update \eqref{symEuler2} or \eqref{Verlet} is equivalent to an alternative operator-splitting discretization. In particular, the problem can be splitted into the following substeps:
\begin{itemize}
    \item[Step 1:]  Vlasov component:
    \begin{equation} \label{split-1}
     \begin{dcases}
    \left(\begin{array}{c} \bx_p^*\\ \bv_p^* \end{array}\right) = \bS (\bx_p^n, \bv_p^n) \,, 
    \\ f^*(\bx_p^*, \bv^*_p) = f^n(\bx_p^n, \bv_p^n)\,.
     \end{dcases}
    \end{equation}

    \item[Step 2:]  Fokker-Planck component:
    \begin{equation} \label{split-2}
    \begin{dcases}
    \min_{\theta}  
    \frac{\Delta t }{2}\frac{1}{N}\sum_{p=1}^N |\bu_\theta(\bx_p^*, \bv_p^n)|^2 
    + \frac{\epsilon}{N}\sum_{p=1}^N  \half |\bv_p^{n+1} |^2+  \log f^{n+1}(\bx_p^{*}, \bv_p^{n+1})\,,
    \\
    \text{s.t.} 
    ~~~\bx_p^{n+1} = \bx_p^*\,,
   \\ \qquad \bv_p^{n+1} = \bv_p^* + \bu_\theta (\bx_p^*, \bv_p^{n})  \Delta t\,,
    \\
    \qquad \log |\text{det} \nabla_{\bx,\bv}\bT(\bx_p^{*},\bv_p^{n+1}) |=\nabla_\bv\cdot \bu_\theta( \bx_p^*, \bv_p^n)  \Delta t\,,
    \\
    \qquad f^{n+1}(\bx_p^{n+1}, \bv_p^{n+1}) = \frac{f^*(\bx_p^*, \bv_p^*)}{|\det \nabla_{\bx,\bv} \bT(\bx_p^{*}, \bv_p^{n+1})|}\,.
    \end{dcases}
\end{equation}
\end{itemize}
While the split formulation may appear straightforward, we note that the non-split formulation has its own merits, as it provides a natural extension from gradient flows to non-gradient-flow systems and may offer new insights into the analytical structure and stability properties of the equation.

\subsubsection{Velocity matching} \label{sec:vm}
If one views \eqref{VFP0} as a generic transport equation in the combined phase space $(\bx,\bv)$, then the velocity matching approach of \cite{li2023self, shen2023entropy} applies naturally. More precisely, one seeks to match the velocity field $\bu$ with $\Amat \nabla_{\bx,\bv} \frac{\delta \energy}{\delta f}$ by solving the following variational problem:
\begin{equation} \label{VFP-var}
\begin{dcases}
& (f,\bu)=\arg\min_{ f,\bu}~ \int_0^T\int_{\Omega_{\bx,\bv}} f |\bu + \Amat \nabla_{\bx,\bv} \frac{\delta \energy}{\delta f}|^2 \rd \bx \rd  \bv \rd t \,, 
\\ & \textrm{s.t.} \quad  \partial_t f + \nabla_{\bx, \bv} \cdot (f \bu) = 0\,, \quad f(0,\bx, \bv ) = f_0 (\bx, \bv )\,,
\end{dcases}
\end{equation}
where $\Amat = \Dmat + \Jmat$. 
We parameterize the velocity field as $\bu(t,\bx,\bv)=\bu_\theta(t,\bx,\bv)$ and denote by $f_\theta$ the corresponding density induced by the continuity equation. 
Following the fixed-point iteration strategy in \cite{li2023self}, problem \eqref{VFP-var} can be solved iteratively via
\begin{equation} 
\begin{dcases}
& \theta^{k+1} \in \arg\min_{ \theta }~ \int_0^T\int_{\Omega_{\bx,\bv}} f_{\theta^k} |\bu_\theta + \Amat \nabla_{\bx,\bv} \frac{\delta \energy}{\delta f_{\theta^k}}|^2 \rd \bx \rd  \bv \rd t\,, \label{self0}
\\ & \textrm{s.t.} \quad  \partial_t f_{\theta^k} + \nabla_{\bx, \bv} \cdot (f_{\theta^k} \bu_{\theta^k}) = 0\,, \quad f(0,\bx, \bv ) = f_0 (\bx, \bv )\,.
\end{dcases}
\end{equation}
Substituting the explicit form of $\energy$ from \eqref{energy0} into \eqref{self0}, the objective function can be simplified as follows (for notational convenience, we omit the subscript $\theta^k$ in $f$):
\begin{align} \label{VFP-var2}
    F :=  & \int_0^T \int_{\Omega_{\bx,\bv}} [|\bu^x_\theta  - (\bv + \nabla_\bv \log f)|^2 + \nonumber
    \\ 
    & \hspace{2cm} |\bu^v_\theta + (\nabla_\bx \phi + \nabla_\bx \log f) - (\bv+ \nabla_\bv \log f)|^2] f \rd \bx \rd \bv \rd t \nonumber
    \\  = & \int_0^T \int_{\Omega_{\bx,\bv}} 
    [|\bu_\theta^x|^2 - 2 \bu_\theta^x \cdot \bv + |\bu_\theta^v|^2  + 2 \bu_\theta^v \cdot \nabla_\bx \phi - 2 \bu_\theta^v \cdot \bv \nonumber 
    \\ & \hspace{2cm} + 2 \nabla_\bv\cdot \bu^x_\theta - 2 \nabla_\bx \cdot \bu^v_\theta + 2 \nabla_\bv \cdot \bu_\theta^v 
    ) ] f \rd \bx \rd \bv \rd t  + \text{constant}\,,
\end{align}
where $\bu = (\bu^x, \bu^v)$. The integral \eqref{VFP-var2}
can be evaluated directly using samples from $f$, thereby completely avoiding explicit density estimation. This simplification is enabled by the score matching principle \cite{hyvarinen2005estimation}. 

Compared to our formulation \eqref{JKO0}, the formulation \eqref{VFP-var} differs in two main aspects.
(i) It adopts an end-to-end training strategy, rather than the sequential training used in \eqref{JKO0}. This approach has both advantages and drawbacks: on the one hand, it avoids repeated retraining over time by learning a time-dependent velocity field in a single stage; on the other hand, the resulting training problem can be significantly more challenging.
(ii) It does not explicitly preserve the conservative–dissipative structure of the equation; consequently, these structural properties are not inherently enforced in the formulation.

\subsubsection{Score based transport modeling}\label{ssec: score-matching}
Another approach that has been studied extensively is score-based transport modeling, which can be viewed as a simplified variant of \eqref{VFP-var}. Instead of learning the entire velocity field, this method focuses on learning only the component that cannot be represented directly by particles, namely the score function. This idea was first introduced in \cite{boffi2023probability} and subsequently extended to the mean-field Fokker–Planck equation in \cite{lu2024score}, as well as to the Landau equation in \cite{huang2025score, ilin2025transport}. In particular, when considering the end-to-end setting, one arrives at the following formulation, which parallels \eqref{VFP-var}:
\begin{equation} \label{score}
\begin{dcases}
& (f,\bu)=\arg\min_{ f,\bu}~ \int_0^T\int_{\Omega_{\bx,\bv}} f |\bu -\nabla_\bv \log f|^2 \rd \bx \rd  \bv \rd t \,,
\\ & \textrm{s.t.} \quad   \partial_t f + \bv \cdot \nabla_\bx f - \nabla_\bx \phi \cdot \nabla_\bv f -\nabla_\bv \cdot(\bv f)-  \nabla_\bv \cdot (\bu f) = 0\,,\quad f(0,\bx, \bv ) = f_0 (\bx, \bv )\,.
\end{dcases}
\end{equation}
Comparing \eqref{score} with \eqref{VFP-var}, the objective function in \eqref{score} involves fewer terms and is therefore easier to optimize. 
It is worth noting that score-based dynamics were introduced by \cite{shen2022selfconsistency}, from the self-consistency of the Fokker-Planck equation. 
The numerical approximation of score dynamics is an independent challenge in its own right. For example, \cite{li2023self} proposes an iterative method with a biased gradient estimator to avoid the estimation of the score function directly. 
In contrast, our method embeds both the conservative Hamiltonian and the dissipative gradient-flow mechanisms directly within a constrained variational framework. This formulation extends the structure-preserving properties of deep JKO schemes from Wasserstein gradient flows to kinetic equations. 
A brief numerical comparison will be presented in Section~\ref{numerical:linear}.

\subsection{Extension to the nonlinear system}\label{sec:pic_jko}
We extend the generalized dynamical JKO scheme \eqref{JKO2} to the nonlinear setting. To handle the Poisson equation, we employ a particle-in-cell method \cite{sonnendrucker2013numerical, chen2011energy}. Here, the subscript $p$ denotes particles, while the subscript $h$ denotes grid points. In particular, $(\bx_p^n, \bv_p^n)$ represent the position and velocity of particle $p$ at time $t^n$, and $\bx_h$ denotes the $h$-th grid point, which remains fixed in time. The scheme takes the following form: 
\begin{equation} \label{JKO2-VPFP}
    \begin{dcases}
    \min_{\theta}  
    \frac{\Delta t }{2}\frac{1}{N}\sum_{p=1}^N |u_\theta(\bx_p^n, \bv_p^n)|^2 
    + \frac{\epsilon}{N}\sum_{p=1}^N \left[ \half |\bv_p^{n+1}|^2 + T_0\log f^{n+1}(\bx_p^{n+1}   , \bv_p^{n+1})  \right] + \sum_{h} |\bE_h^\np|^2 (\Delta x)^d\,,
    \\
    \text{s.t.} 
    ~~\bx_p^{n+1} = \bx_p^n + \bv_p^{n}\Delta t\,,
    \\ \qquad \bv_p^{n+1} = \bv_p^n + \bE_p^{n+1}\Delta t + \bu_\theta (\bx_p^n, \bv_p^{n})  \Delta t\,,
    \\ 
    \qquad \log |\text{det} \nabla_{\bx,\bv}\bT(\bx_p^{n+1},\bv_p^{n+1}) |=\nabla_\bv\cdot \bu_\theta( \bx_p^n, \bv_p^n) \Delta t\,,
    \\
    \qquad f^{n+1}(\bx_p^{n+1}, \bv_p^{n+1}) = \frac{f^n(\bx^n, \bv^n)}{|\det \nabla_{\bx,\bv} \bT(\bx_p^{n+1}, \bv_p^{n+1})|}\,.
    \end{dcases}
\end{equation}
where the field term $\bE$ is computed as follows: 
\begin{align*}
    &\bE_p^{n+1} = \sum_h S(\bx_p^{n+1 }- \bx_h) \bE_h^{n+1}\,,
    \\
    &\bE_h^{n+1} = -\nabla_\bx \phi_h^{n+1}\,, ~~ -\Delta_\bx \phi_h^{n+1} = \rho_h^{n+1} -1\,,  ~~
    \rho_h^{ n+1} = \frac{1}{\Delta x} \sum_p \omega_p S(\bx_h- \bx_p^{n+1}),\, ~~ \omega_p = \frac{1}{N}.
\end{align*}
Here, the index $h$ represents the index of a point $x_h$ which is a center point of each grid cell, $S(\bx)$ is a local basis function, such that $\frac{1}{|\text{cell}_h|}\int_{\text{cell}_h} S(\bx) \rd \bx = 1$. A common choice is the tent function
\[
S(z) = \max\{0, 1- \frac{|z|}{\Delta x}\}\,.
\]

\begin{remark}
We note that here we employ an explicit symplectic integrator for the Hamiltonian part, which is not exactly energy-conserving. Nevertheless, one can follow the nice approach proposed in \cite{ricketson2025explicit} to enforce exact energy conservation. This adaptation is straightforward in our setting, and we therefore do not pursue it further here.
\end{remark}

\begin{algorithm}[ht!]
\caption{Kinetic JKO scheme based on \eqref{JKO2-VPFP}}
\label{alg:vafp_jko2_time_marching}
\begin{algorithmic}[1]

\Require Terminal time $T>0$, time step $\Delta t>0$, number of particles $N$, collision strength $\epsilon>0$,
initial particles $\{(\bx_p^0,\bv_p^0,f^0)\}_{p=1}^N$, number of inner optimization steps $K$, a parametrized function $\bu_\theta$ with parameters $\theta$.

\Ensure Approximate particle state at time $T$.

\State Set $M=\lceil T/\Delta t\rceil$.

\For{$n=0,1,\dots,M-1$}

    \State \textbf{Position update:}
    \For{$p=1,\dots,N$}
        \State $\bx_p^{n+1} = \bx_p^n + \bv_p^{n}\Delta t$\,;
    \EndFor

    \State \textbf{Field update:}
    \State Compute grid density $\rho_h^{n+1} = \frac{1}{\Delta x} \sum_p \frac{1}{N} S(\bx_h- \bx_p^{n+1})$ from new positions.   \State Solve the discrete Poisson equation $-\Delta_\bx \phi_h^{n+1} = \rho_h^{n+1} - 1$.
    \State Compute $\bE_h^{n+1} = -\nabla_\bx \phi_h^{n+1}$ and interpolate to particles: $\bE_p^{n+1} = \sum_h S(\bx_p^{n+1}- \bx_h) \bE_h^{n+1}$.

    \State \textbf{Neural network optimization:}
    \State Initialize network parameters $\theta \leftarrow \theta^{(0)}$.
    \For{$k=1,\dots,K$}
        \State Perform one optimization step to minimize the loss in \eqref{JKO2-VPFP} and update $\theta$.
    \EndFor

    \State \textbf{Velocity and density update:}
    \For{$p=1,\dots,N$}
        \State $\bx_p^{n+1} = \bx_p^n + \bv_p^{n}\Delta t$\,;
        \State $\bv_p^{n+1} = \bv_p^n + \bE_p^{n+1}\Delta t + \bu_\theta(\bx_p^n,\bv_p^n)\Delta t$\,;
        \State $f^{n+1}(\bx_p^{n+1},\bv_p^{n+1}) = f^{n}(\bx_p^{n},\bv_p^{n}) \exp\bigl(-\nabla_\bv\cdot \bu_\theta(\bx_p^n,\bv_p^n) \Delta t\bigr)$\,;
    \EndFor

\EndFor

\State \Return $\{(\bx_p^M,\bv_p^M,f^M)\}_{p=1}^N$\,.

\end{algorithmic}
\end{algorithm}

\section{Numerical examples}\label{sec4}

In this section, we present numerical results to demonstrate the effectiveness of our proposed method. We first consider several linear kinetic Fokker-Planck equations with known analytical solutions to validate the accuracy of our approach. We then apply our method to nonlinear kinetic Fokker-Planck equations, particularly the Vlasov-Poisson-Fokker-Planck system, demonstrating its capability to handle complex interactions and long-time behaviors.

\subsection{Linear kinetic Fokker-Planck equations} \label{numerical:linear}
Here we present two examples of linear kinetic Fokker-Planck equations: one with an explicit solution over time, and another with an explicit solution only at equilibrium. Both serve as initial testbeds for our proposed method. We present results in both 1D position–1D velocity and 3D position–3D velocity settings.

\bigskip

\paragraph{Example 1: Gaussian cases with explicit solutions}

We begin with a simple test case where the solution to the kinetic Fokker-Planck equation remains Gaussian over time. Specifically, we consider the state variable $\bx = (x, v)$, where $x, v \in \mathbb{R}^d$. Assume there exists a positive definite matrix $\mathsf{K} \in \mathbb{R}^{2d \times 2d}$ and a vector $\tilde{\mu} \in \mathbb{R}^{2d}$ such that the potential function $\phi(x)$ satisfies
\[
\frac{v^2}{2} + \phi(x) = \frac{1}{2}(\bx - \tilde{\mu})^\top \mathsf{K}^{-1} (\bx - \tilde{\mu})\,.
\]
Suppose the initial distribution is $f(0, \bx) \sim \mathcal{N}(\mu_0, \mathsf{C}_0)$. Then the solution to the kinetic equation \eqref{VFP0} remains Gaussian:
    \[
    f(t, \bx) \sim \mathcal{N}(\mu(t), \mathsf{C}(t))\,,
    \]
    where $\mu(t)$ and $\mathsf{C}(t)$ evolve according to
    \begin{subequations} \label{mu_C}
    \begin{align}
        \dot \mu(t) &= -(\mathsf{D} + \mathsf{J})\mathsf{K}^{-1}(\mu(t) - \tilde{\mu})\,, \\
        \dot{\mathsf{C}}(t) &= 2\mathsf{D} - 2\,\mathrm{Sym}\big((\mathsf{D} + \mathsf{J})\mathsf{K}^{-1} \mathsf{C}(t)\big)\,,
    \end{align}
    \end{subequations}
    with initial conditions $\mu(0) = \mu_0$, $\mathsf{C}(0) = \mathsf{C}_0$, and
    \[
        \mathrm{Sym}(A) = \frac{1}{2}(A + A^\top)\,.
    \]

In the 1D case ($d=1$), we choose
\[
\mathsf{K} = \begin{pmatrix} \frac{1}{2} & 0 \\ 0 & 1 \end{pmatrix}\,, \quad \tilde{\mu} = \begin{pmatrix} 0 \\ 0 \end{pmatrix}\,,
\]
which corresponds to the potential $\phi(x) = x^2$. The initial condition is set as $\mu_0 = (0, 1)^\top$ and
\[
\mathsf{C}_0 = \begin{pmatrix} 2 & 1 \\ 1 & 3 \end{pmatrix}\,.
\]
We then simulate the dynamics of \eqref{VFP0}.

\bigskip

Since the exact solution remains Gaussian, the optimal control field can be well approximated by an affine function. To exploit this structure, we parameterize the neural network model for $u_\theta(x, v)$ as a single-layer affine map:
\[
\bu_\theta(x, v) = W_\theta 
\begin{pmatrix}
x \\ v
\end{pmatrix} + b_\theta\,,
\]
where $W_\theta \in \mathbb{R}^{d \times 2d}$ and $b_\theta \in \mathbb{R}^d$.
This parameterization is sufficient to capture the true control dynamics in this setting while keeping the model simple and efficient to train.

The control field $\bu_\theta$ is trained using the proposed variational formulation given in \eqref{JKO2} and \Cref{alg:jko2_time_marching}. To assess the accuracy of this approach, we compare it against a score-matching method, which directly estimates the score function $\nabla_{\bv} \log f$ from samples. To establish a ground truth for this comparison, we compute the exact solution by solving the ODE system \eqref{mu_C} for the mean and covariance of the Gaussian distribution, yielding $f_{\mathrm{exact}}$. We then evaluate the drift error at each time step $n$, defined as the squared difference between the learned field $\bu_\theta$ and the exact analytical field:
\begin{equation}\label{eq:drift_error}
\mathrm{Error}^2(n)
=
\frac{1}{N_p}
\sum_{p=1}^{N_p}
\left|
\bu_\theta^n(\bx_p^{n}, \bv_p^{n})
+
\bv_p^{n+1}
+
\nabla_\bv \log f_{\mathrm{exact}}^{n+1}(\bx_p^{n+1}, \bv_p^{n+1})
\right|^2\,.
\end{equation}

Figure~\ref{fig:exp1-over-time} demonstrates the temporal convergence of the proposed method detailed in \Cref{alg:jko2_time_marching}. It plots the time-averaged drift error, computed from $t=0$ to the terminal time $T=8$, as a function of the time step size $\Delta t$. For both the 1D and 3D simulations, we use $N_p = 10{,}000$ particles, where $(\bx_p^n, \bv_p^n)$ denotes the state of the $p$-th particle at time step $n$. The reference exact density, $f_{\mathrm{exact}}$, is generated using a highly refined time step of $\Delta t_{\mathrm{ref}} = 10^{-4}$ to ensure high accuracy. As indicated by the dashed reference lines, the proposed method exhibits a linear, first-order ($\mathcal{O}(\Delta t)$) convergence rate.

Figure~\ref{fig:exp1-error-comparison} compares the performance of our proposed method against the score-matching approach over time. The left panel tracks the drift error (computed via \eqref{eq:drift_error}) at each individual time step to highlight the relative accuracy of the two methods. The right panel illustrates the energy decay exclusively for our proposed algorithm. The discrete energy at time step $n$ is calculated as:
\begin{equation}\label{eq:energy}
    \mathrm{Energy}(n) = \frac{1}{N_p} \sum_{p=1}^{N_p} \left( \frac{1}{2} |\bv_p^{n+1}|^2 + \phi(\bx_p^{n+1}) + \log f^{n+1}(\bx_p^{n+1}, \bv_p^{n+1}) \right).
\end{equation}

These experiments were conducted over the time interval $[0, 10]$ for the 1D case and $[0, 20]$ for the 3D case, both utilizing a discrete time step of $\Delta t = 0.1$. For the neural network architectures, both configurations used fully connected networks with \texttt{Tanh} activations between layers. Specifically, the 1D model consisted of 2 hidden layers with 16 units each, while the 3D model used 2 hidden layers with 32 units each.

In the 1D case, our algorithm exhibits stable and consistently low error throughout the time interval, whereas the score-matching method shows noticeable oscillations, suggesting some instability in the drift approximation. In the 3D case, the score-based model performs well during the early phase, yielding smaller errors up to $t \approx 6$, but its accuracy appears to plateau thereafter. In contrast, our algorithm continues to improve over time and eventually attains lower errors for $t > 6$.




\begin{figure}
    \centering
    \begin{subfigure}[t]{0.48\linewidth}
        \centering
        \includegraphics[width=0.99\textwidth]{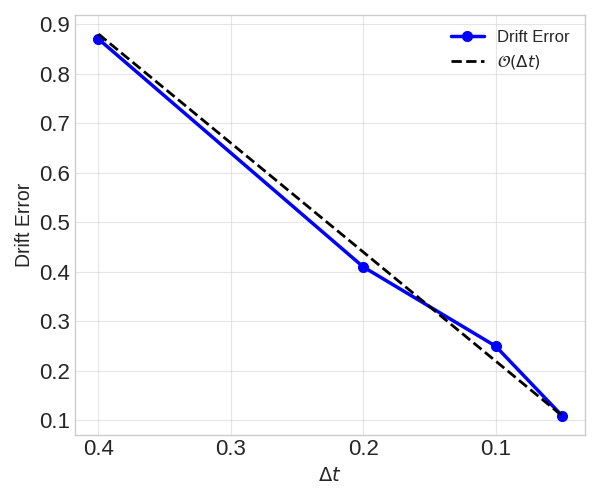}
        \caption{1D}
    \end{subfigure}
    \begin{subfigure}[t]{0.48\linewidth}
        \centering
        \includegraphics[width=0.99\textwidth]{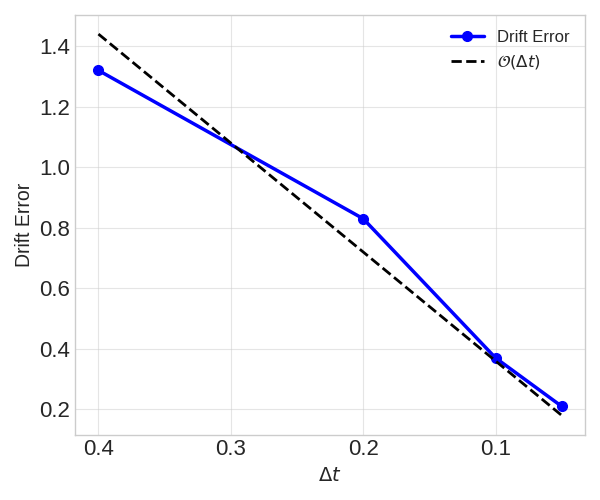}
        \caption{3D}
    \end{subfigure}
    \caption{
        Drift error vs. time step size $\Delta t$ for Example~1 in the 1D (left) and 3D (right) cases.
        The plots demonstrate linear convergence behavior as the time step decreases.
        The dashed lines indicate a reference $\mathcal{O}(\Delta t)$ convergence rate.
    }
    \label{fig:exp1-over-time}
\end{figure}

\begin{figure}[htbp]
    \centering
    \begin{subfigure}[t]{0.48\linewidth}
        \centering
        \includegraphics[width=\linewidth]{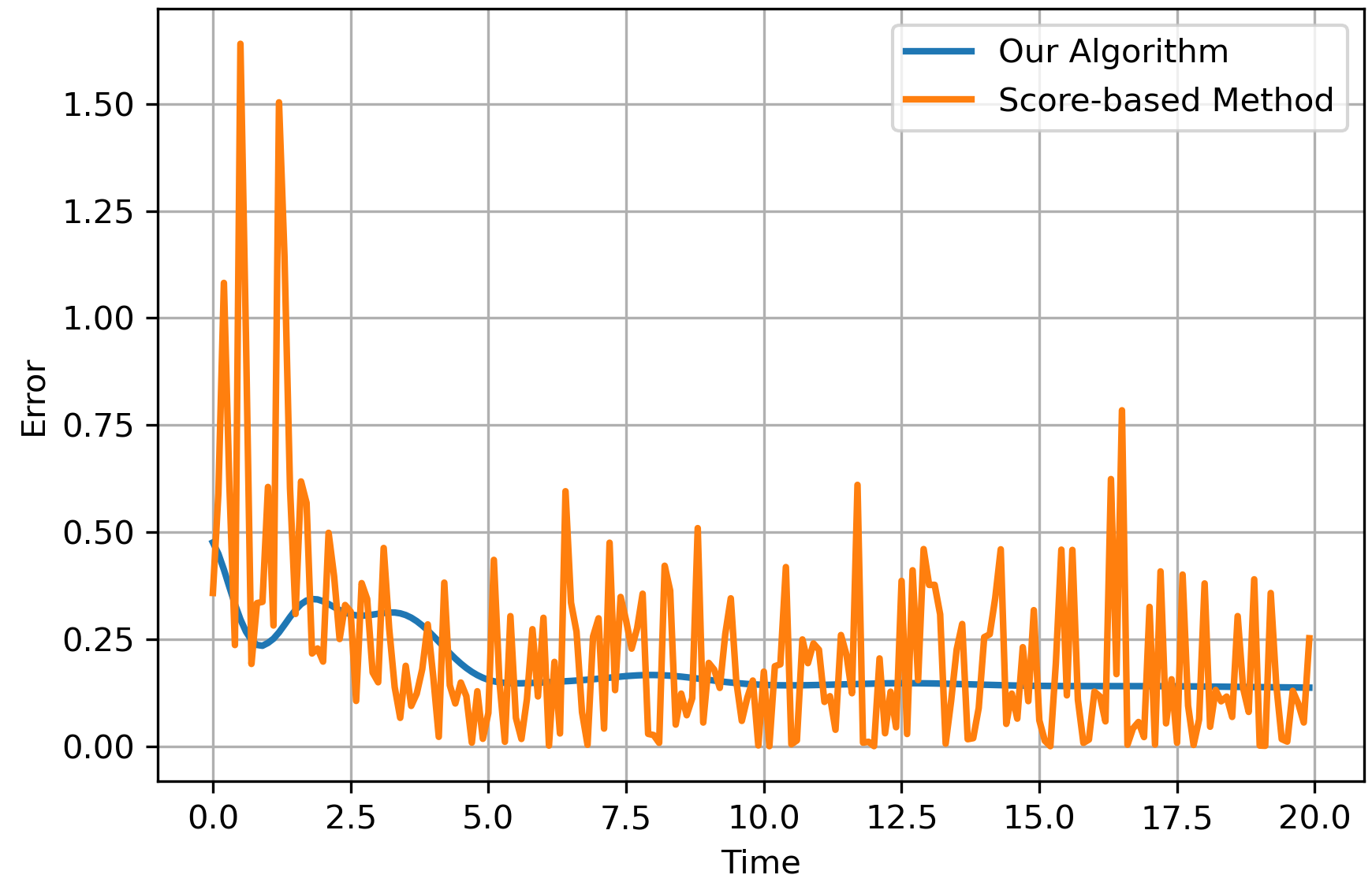}
        \caption{1D: error plot}
        \label{fig:exp1-error}
    \end{subfigure}
    \hfill
    \begin{subfigure}[t]{0.48\linewidth}
        \centering
        \includegraphics[width=\linewidth]{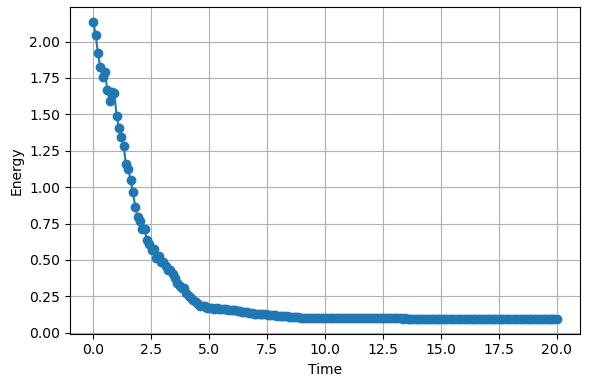}
        \caption{1D: energy plot}
        \label{fig:exp1-3d-error}
    \end{subfigure}

    \begin{subfigure}[t]{0.48\linewidth}
        \centering
        \includegraphics[width=\linewidth]{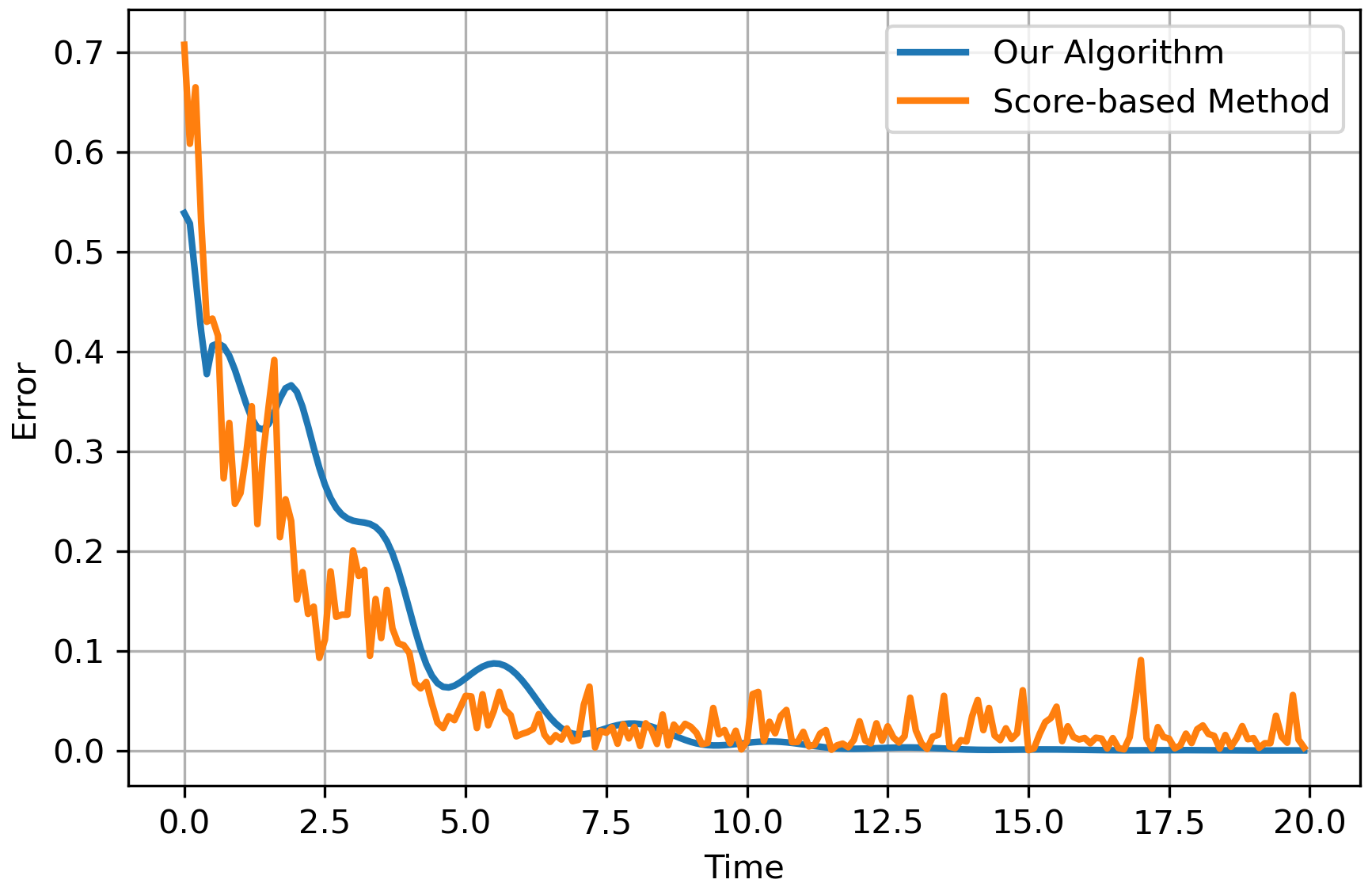}
        \caption{3D: error plot}
        \label{fig:exp1-error}
    \end{subfigure}
    \hfill
    \begin{subfigure}[t]{0.48\linewidth}
        \centering
        \includegraphics[width=\linewidth]{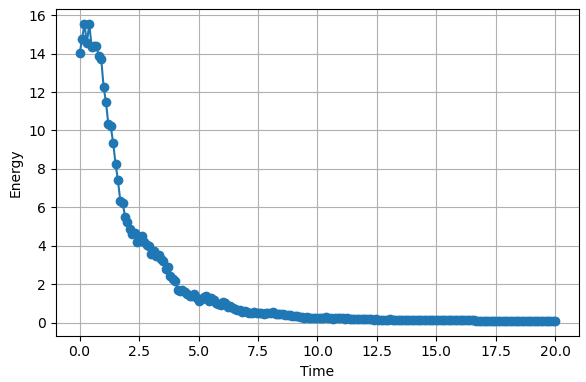}
        \caption{3D: energy plot}
        \label{fig:exp1-3d-error}
    \end{subfigure}
    \caption{Drift error plots using \eqref{eq:drift-error} compare the errors between our algorithm and the score-based model for Example 1 in the 1D case (left) and the 3D case (right). The 1D experiment was conducted over the time interval from 0 to 10 with a time step of 0.1, while the 3D experiment was run from 0 to 20 with the same time step. The same parameters, including the neural network architecture, were used for both experiments and both methods.
}
    \label{fig:exp1-error-comparison}
\end{figure}


\bigskip

\paragraph{Example 2. Convergence to the stationary distribution} 
In the second experiment, we investigate the convergence of the learned solution toward the stationary equilibrium. We first consider $1D$ in $x$ and $1D$ in $v$ case 
\begin{equation} \label{02191}
\partial_t f + v \,\partial_x f - \partial_x \phi \,\partial_v f = \partial_v \cdot \big(v f + \partial_v f\big)\,,
\end{equation}
with initial condition 
\begin{align*}
\rho^{0}(x) &= \mathcal{N}(0,1)\,, \\
f^{0}(x,v) &= \frac{\rho^{0}(x)}{2\sqrt{2\pi}}\left(\exp\!\left(-\tfrac{|v+1.5|^2}{2}\right) + \exp\!\left(-\tfrac{|v-1.5|^2}{2}\right)\right)\,,
\end{align*}
and the  potential function
\[
\phi(x) = \tfrac{1}{2} x^2 + \cos(2 \pi x)\,.
\]
The stationary distribution corresponding to \eqref{02191} is known explicitly as  
\begin{align} \label{equ11}
f_{\infty}(x,v) = C \exp\!\left(-\tfrac{v^2}{2} - \phi(x)\right)\,,
\end{align}
where $C$ is the normalization constant ensuring $\int f_\infty(x,v)\rd\bx\rd\bv = 1$. As shown in \cite{desvillettes2001trend}, solutions to \eqref{02191} converge exponentially fast to \eqref{equ11} in relative entropy.  

To approximate this behavior numerically, we represent the velocity field $\bu_\theta$ with a neural network consisting of two hidden layers with 16 neurons and \texttt{Tanh} activation. The time step is chosen to be $\Delta t = 0.1$, advanced up to final time $T=10$, yielding 100 outer iterations. 

Figure~\ref{fig:example2-quadratic-no-bdry} visualizes the evolution of the density produced by the algorithm from $t=0$ to $t=2$. 
The plots show convergence toward the expected stationary solution: the stationary distribution is drawn as contours, while the computed solution appears as a density map in which lighter regions indicate higher particle density. The close overlap between high-density regions and the contours demonstrates accurate convergence.

Figure~\ref{fig:example2-quadratic-marginals-KL-no-bdry} presents two complementary diagnostics. 
Panel (A) displays the $x$- and $v$-marginals of the computed solution at $t=10$ together with the stationary marginals. 
The close overlap indicates that the numerical method attains high accuracy at late times. 
Panel (B) reports the Kullback--Leibler divergence between the computed distribution and the stationary distribution for $t \in [0,10]$ under the particle JKO dynamics, where an exponential decay is observed. 

\begin{figure}[ht]
    \centering
    \begin{subfigure}{0.32\linewidth}
        \centering
        \includegraphics[width=\linewidth]{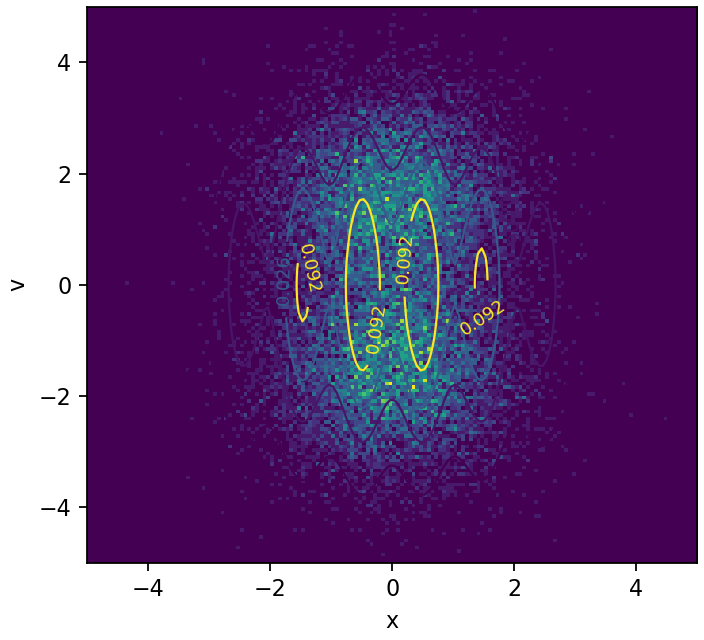}
        \caption{$t=0.00$}
    \end{subfigure}
    \begin{subfigure}{0.32\linewidth}
        \centering
        \includegraphics[width=\linewidth]{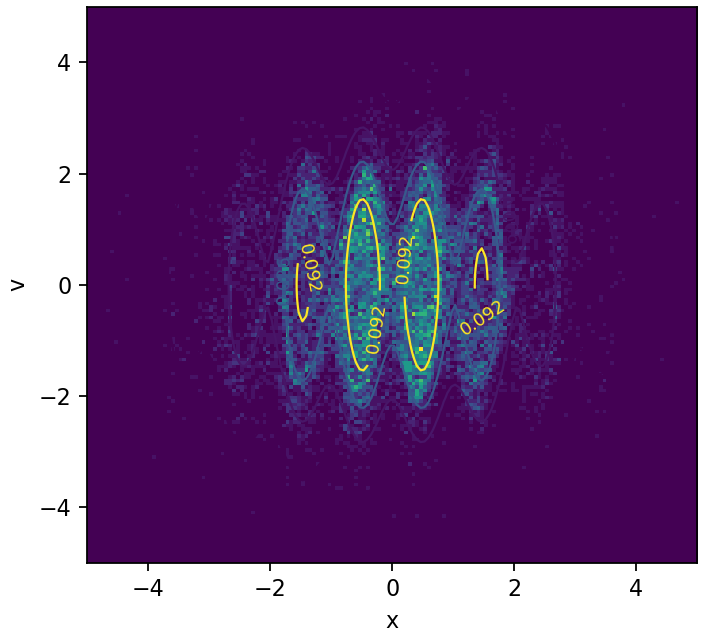}
        \caption{$t=1.00$}
    \end{subfigure}
    \begin{subfigure}{0.32\linewidth}
        \centering
        \includegraphics[width=\linewidth]{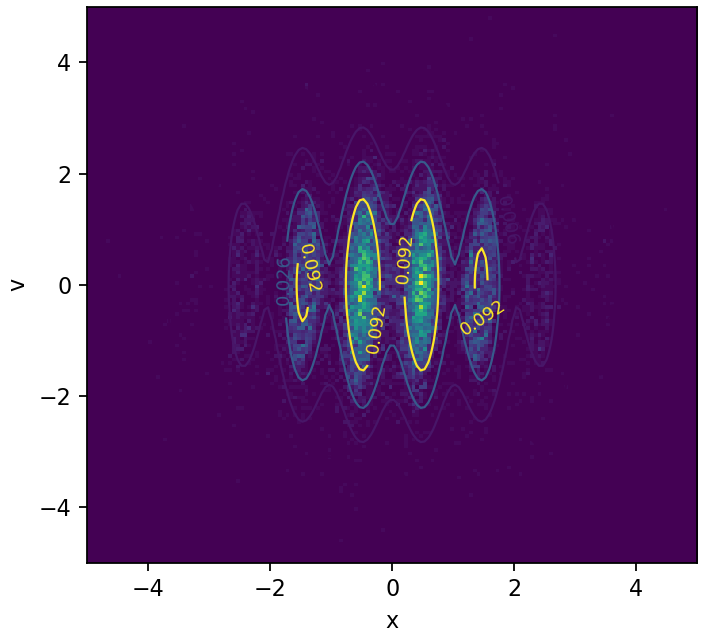}
        \caption{$t=10.00$}
    \end{subfigure}
    \caption{Evolution of the distribution from $t=0$ to $t=2$ from Experiment 2 with step size $0.1$. The results show convergence toward the stationary distribution at $t=2$. Bright regions indicate areas of high particle density, darker regions indicate low density, and the contours 
represent the stationary solution. 
}
    \label{fig:example2-quadratic-no-bdry}
\end{figure}

\begin{figure}[ht]
    \centering
    \begin{subfigure}{0.62\linewidth}
        \centering
        \includegraphics[width=\linewidth]{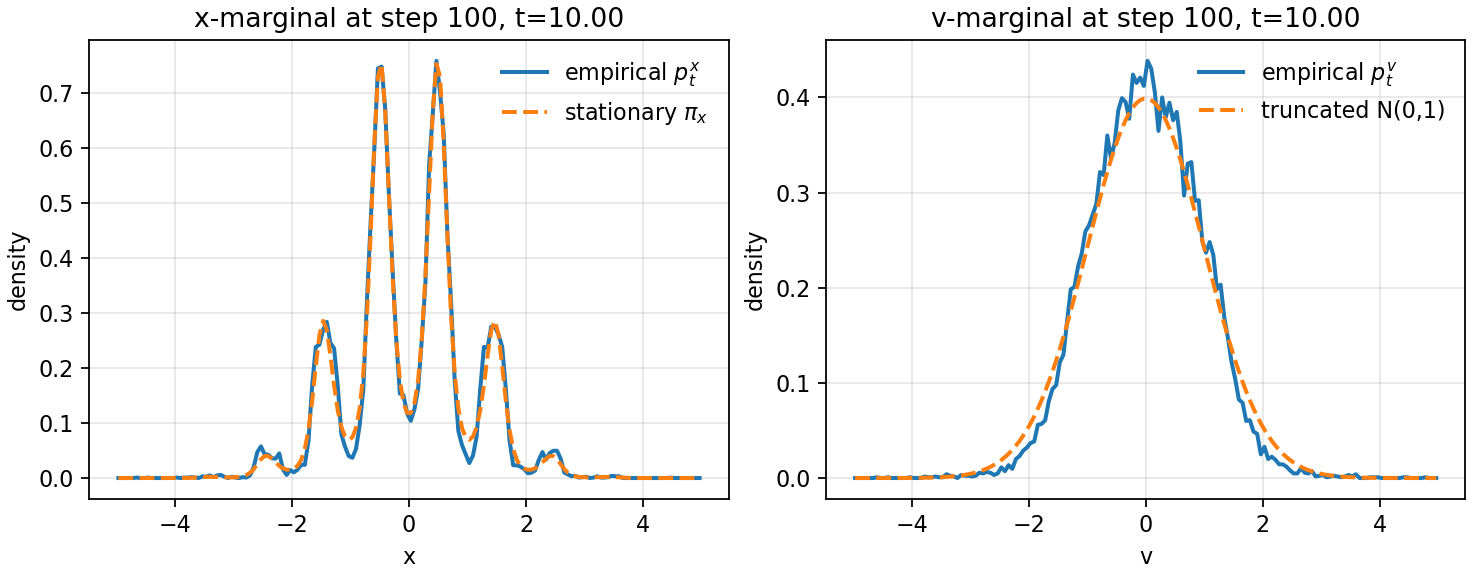}
        \caption{Sine potential}
    \end{subfigure}
    \begin{subfigure}{0.37\linewidth}
        \centering
        \includegraphics[width=\linewidth]{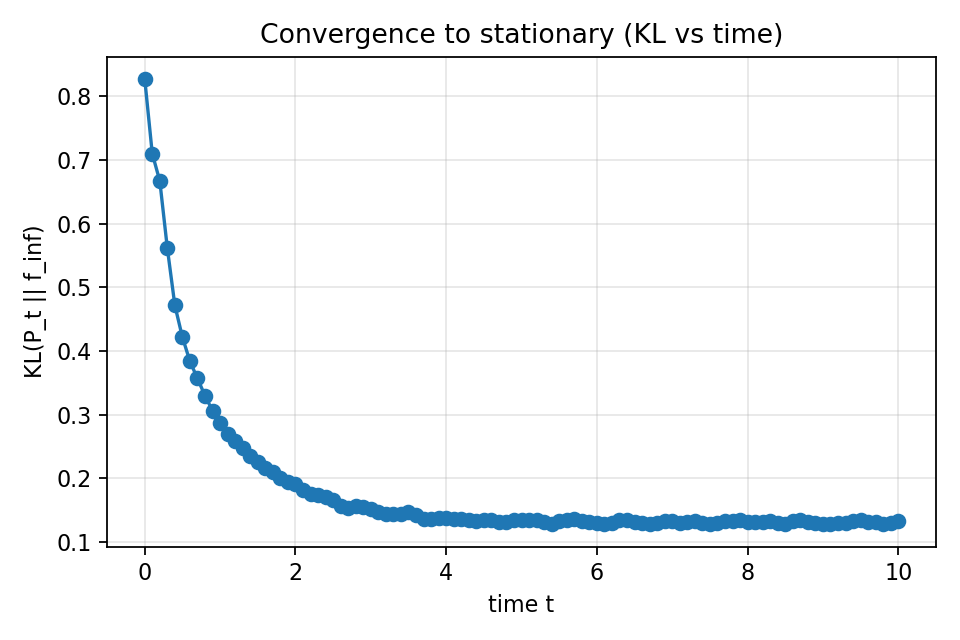}
        \caption{KL divergence over time}
    \end{subfigure}
    \caption{(A) The $x$-marginal and $v$-marginal at time $t=10$. 
    (B) Convergence of the Kullback--Leibler divergence to the stationary distribution for $t \in [0,10]$ under the particle JKO dynamics from Experiment~2.}
    \label{fig:example2-quadratic-marginals-KL-no-bdry}
\end{figure}

\bigskip

\paragraph{Example 3. periodic domain} 
In the third experiment, similar to the second experiment, we investigate the convergence of the learned solution toward the stationary equilibrium of a kinetic Fokker--Planck--type equation but the only change is the boundary condition. Here, we consider the periodic boundary condition on $x$ space,  following the setup in \cite{carrillo2021variational}. We consider the same one-dimensional-in-space and one-dimensional-in-velocity kinetic equation in \eqref{02191} posed on $x \in [0,1]$ with periodic boundary conditions and $v \in \mathbb{R}$. The initial condition is chosen as the product of a periodic function in $x$ and a symmetric mixture of Gaussians in $v$:  
\begin{align*}
\rho^{0}(x) &= 1+ \tfrac{1}{2} \cos(2\pi x)\,,\\ 
f^{0}(x,v) &= \frac{\rho^{0}(x)}{2\sqrt{2\pi}}\left(\exp\!\left(-\tfrac{|v+1.5|^2}{2}\right) + \exp\!\left(-\tfrac{|v-1.5|^2}{2}\right)\right)\,.
\end{align*}
We test the following external potential:  
\begin{equation}\label{eq:potentials-linear}
    \phi_s(x) = \tfrac{1}{5}\sin(2\pi x)\,.
\end{equation}

We approximate the velocity field $\bu_\theta$ using a neural network with two hidden layers of 16 neurons each. To accommodate periodic boundary conditions in this experiment, we enforce periodicity directly in the neural network inputs. Specifically, for a given state $(x,v)$, the network takes $(\sin(2 \pi x), \cos(2 \pi x), v)$ as its input. The evolution of the particle system is advanced iteratively using the proposed algorithm in \Cref{alg:jko2_time_marching}. For each discrete time step $n$, the algorithm computes the state transition to the subsequent step $n+1$ using a fixed step size of $\Delta t = 0.1$. This process is repeated up to a final time of $T = 20$, yielding 200 outer iterations. The density evolution for both choices of the potential function is reported in Figure~\ref{fig:example2-quadratic}. In both cases, the learned solution converges to the stationary state, confirming the expected long-time behavior.

Figure~\ref{fig:example2-quadratic-marginals} summarizes the behavior of the numerical solution through three plots. The first two panels compare the empirical marginals with their stationary targets, namely $f_t^x$ versus $\pi_x(x)\propto e^{-\phi(x)}$ and $f_t^v$ versus $\mathcal N(0,1)$. The overlays show strong agreement in both $x$-marginal and $v$-marginal. The third panel shows the convergence of the Kullback--Leibler divergence between the computed solution and the exact stationary distribution over the time interval $t \in [0,20]$. At each time step, the numerical solution is represented by an ensemble of particles $\{(\bx_p^n,\bv_p^n)\}_{p=1}^{N_p}$. The stationary distribution $f_\infty$ is known in closed form  from \eqref{equ11}.  The Kullback--Leibler divergence
\[
\mathrm{D}_{\mathrm{KL}}\bigl(f^n \,\|\, f_\infty\bigr)
=
\int f^n(\bx,\bv)
\log\!\left(
\frac{f^n(\bx,\bv)}{f_\infty(\bx,\bv)}
\right)
\, \rd\bx \, \rd\bv\,,
\]
is therefore approximated using a Monte Carlo estimator based on the particle ensemble. Specifically, we evaluate the Kullback-Leibler (KL) divergence using a Monte Carlo approximation over the particle ensemble:
\[
\mathrm{D}_{\mathrm{KL}}\bigl(f^n \,\|\, f_\infty\bigr) \;\approx\; \frac{1}{N_p} \sum_{p=1}^{N_p} \left( \log f^n(\bx_p^n,\bv_p^n) - \log f_\infty(\bx_p^n,\bv_p^n) \right)\,.
\]
In the implementation, the exact log-density $\log f^n(\bx_p^n,\bv_p^n)$ is evaluated by dynamically tracking it for each particle along its trajectory. Starting from the analytically known initial distribution, the log-density is updated at each discrete time step using the linearized log-determinant of the pushforward mapping (i.e., $\log f^{n+1} = \log f^n - \Delta t \, \partial_{\bv} \bu_\theta$). For the stationary distribution, the analytical formula from \eqref{equ11} is used. The resulting KL divergence curves decrease substantially over time, demonstrating rapid convergence of the numerical solution to the stationary distribution.

\begin{figure}[ht]
    \centering
    \begin{subfigure}{0.19\linewidth}
        \centering
        \includegraphics[width=\linewidth]{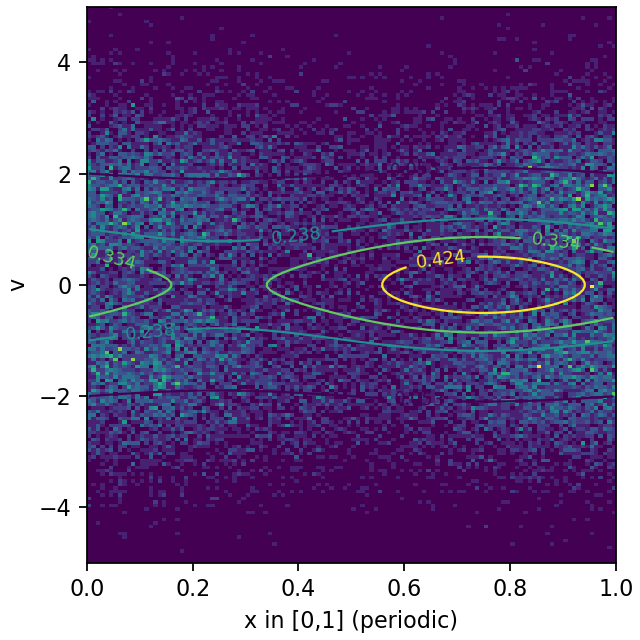}
        \caption{$t=0.00$}
    \end{subfigure}
    \begin{subfigure}{0.19\linewidth}
        \centering
        \includegraphics[width=\linewidth]{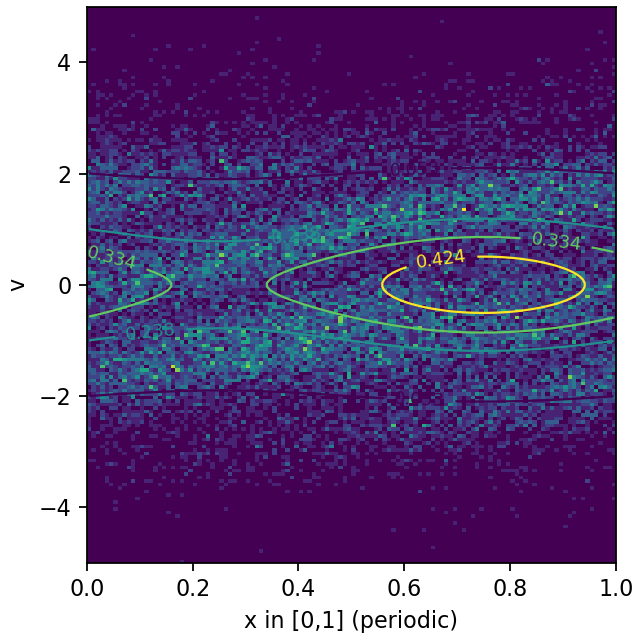}
        \caption{$t=0.50$}
    \end{subfigure}
    \begin{subfigure}{0.19\linewidth}
        \centering
        \includegraphics[width=\linewidth]{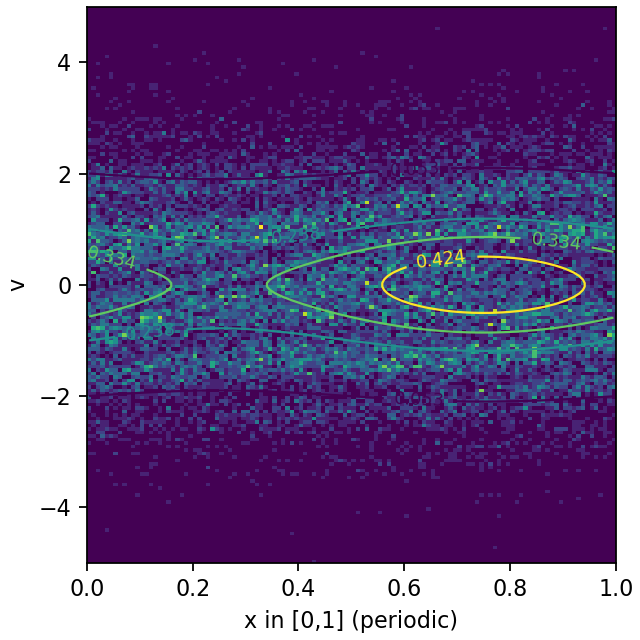}
        \caption{$t=1.00$}
    \end{subfigure}
    \begin{subfigure}{0.19\linewidth}
        \centering
        \includegraphics[width=\linewidth]{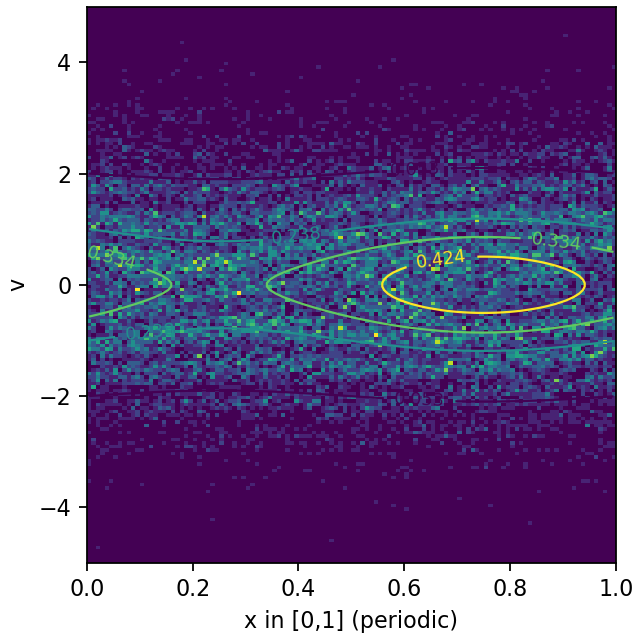}
        \caption{$t=1.50$}
    \end{subfigure}
    \begin{subfigure}{0.19\linewidth}
        \centering
        \includegraphics[width=\linewidth]{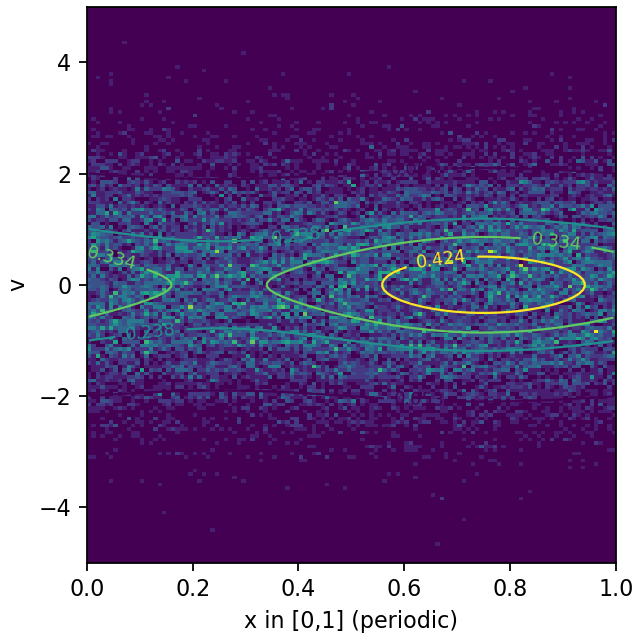}
        \caption{$t=2.00$}
    \end{subfigure}

    \caption{Experiment from Example 3. Evolution of the distribution from $t=0$ to $t=2$ with step size $0.1$, using the sine potential. The results show convergence toward the stationary distribution at $t=2$. 
Bright regions indicate areas of high particle density, darker regions indicate low density, and the contours 
represent the stationary solution.
}
    \label{fig:example2-quadratic}
\end{figure}

\begin{figure}[ht]
    \centering

    \begin{subfigure}{0.60\linewidth}
        \centering
        \includegraphics[width=\linewidth]{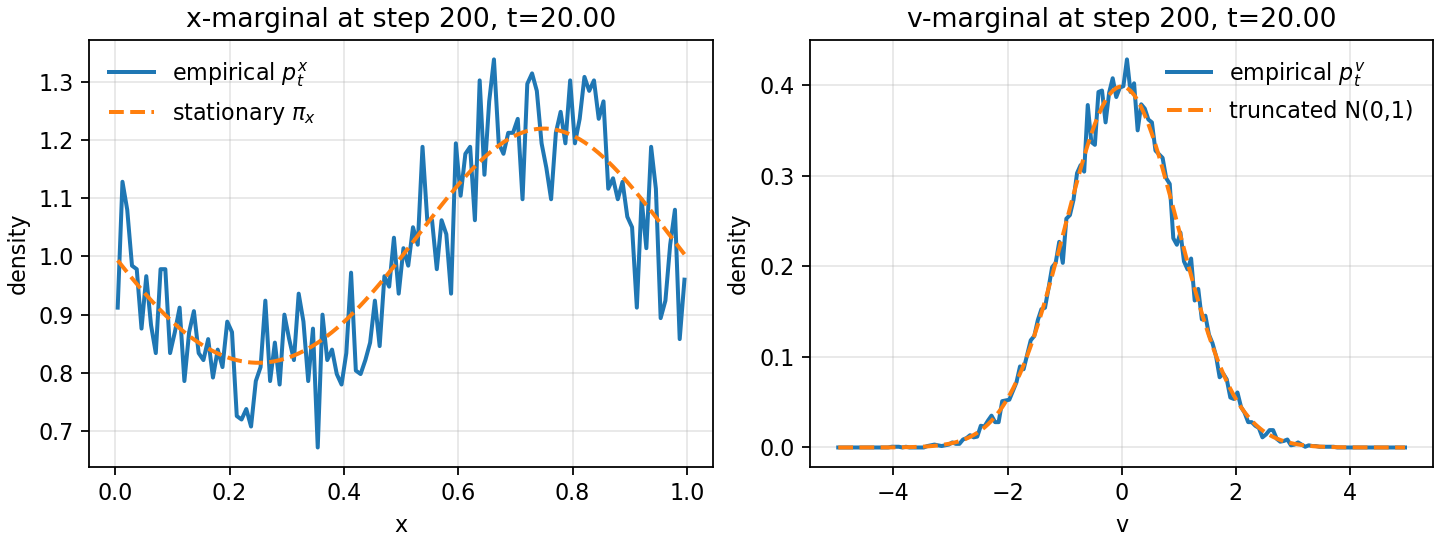}
        \caption{Sine potential}
    \end{subfigure}
    \begin{subfigure}{0.35\linewidth}
        \centering
        \includegraphics[width=\linewidth]{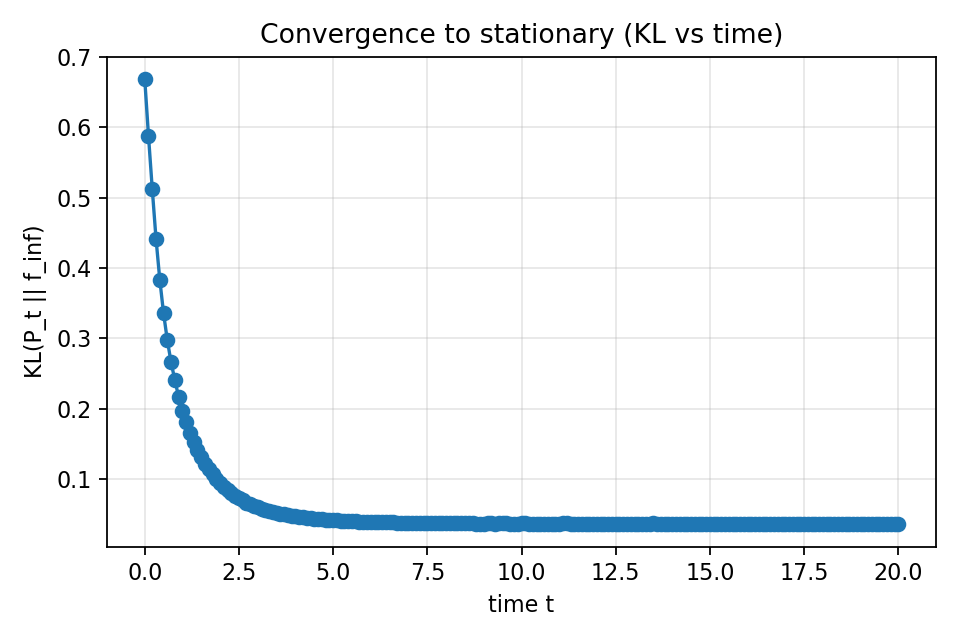}
        \caption{KL from a sine potential}
    \end{subfigure}
    \caption{Experiment from Example 3. The $x$- and $v$-marginal densities at $t = 20$ under the potential defined in \eqref{eq:potentials-linear}.
The right plot shows the convergence of the Kullback–Leibler divergence to the stationary distribution over the time interval $t \in [0,20]$ for the particle JKO dynamics under two external potentials.
}
    \label{fig:example2-quadratic-marginals}
\end{figure}

\bigskip

\paragraph{Example 4: Convergence to the stationary distribution in 3D.}
We extend Example~2, which was posed in one spatial dimension and one velocity dimension, to the three–dimensional setting in both $x$ and $v$ to demonstrate that the algorithm computes an accurate solution that converges toward the expected stationary distribution. For this 3D experiment, the external potential is
\[
  \phi(x) \;=\; \tfrac12 \|x\|_2^2 \;+\; \sum_{i=1}^3 \cos(2\pi x_i)\,,
\]
which combines a quadratic confining term with small coordinate–wise periodic bumps. The associated stationary density $f_\infty$ on $\mathbb{R}^3 \times \mathbb{R}^3$ is
\[
  f_\infty(x,v) \;\propto\; \exp\!\bigl(-\tfrac12 \|v\|_2^2 - \phi(x)\bigr)\,,
\]
so the $v$–marginal is $\mathcal N(0,I_3)$ and each coordinate $x_i$ has a one–dimensional stationary density proportional to $\exp(-\tfrac12 x_i^2 - \cos(2\pi x_i))$. We initialize particles from
\[
\begin{aligned}
    f^0(x,v) &=\frac{1}{2(2\pi)^{3}} \left( e^{-\frac{(v_1+1.5)^2}{2}} + e^{-\frac{(v_1-1.5)^2}{2}} \right) e^{-\frac{v_2^2+v_3^2}{2}}e^{-\frac{\|x\|^2}{2}}.
\end{aligned}
\]
To illustrate the distribution of particles over time, we project the data onto one-dimensional marginals for both the spatial and velocity coordinates. We then construct empirical histograms using 160 bins for both the $x$ and $v$ marginals to approximate the underlying densities. These are overlaid with the analytical stationary distributions to facilitate visual comparison. Across runs, we observe a monotone decrease of $\mathrm{D}_{\mathrm{KL}}(f_t\,\|\,f_\infty)$ 
together with increasing alignment of each marginal with its stationary counterpart, indicating convergence toward $f_\infty$. This behavior is illustrated in Figures~\ref{fig:3d-no-bdry-marginals} and \ref{fig:example3-quadratic-marginals-KL-no-bdry}.

Figure~\ref{fig:example3-quadratic-marginals-KL-no-bdry} further quantifies this trend for Experiment~4 over the time window $t\in[0,5]$ with time step $\Delta t=0.1$. We report the evolution of the KL divergence, which exhibits a steady monotone decay, alongside the corresponding marginal discrepancies. Consistently decreasing trends are observed across all coordinates, indicating that both spatial and velocity components relax toward equilibrium. The agreement between empirical marginals and their stationary targets at selected times provides additional evidence supporting convergence of the full distribution toward $f_\infty$.

\begin{figure}[t]
  \centering
  \begin{subfigure}{0.95\linewidth}
    \centering
    \includegraphics[width=\linewidth]{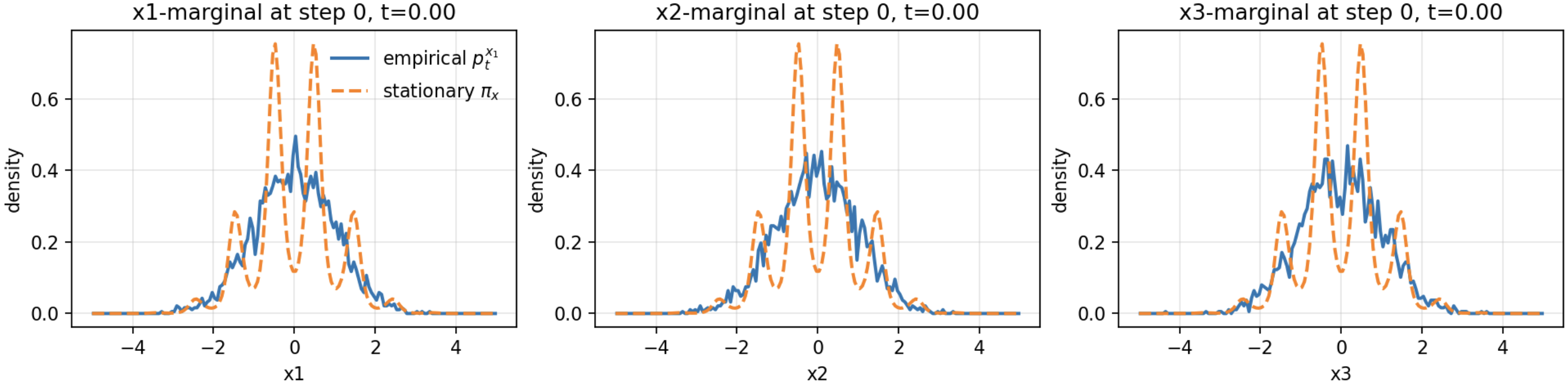}
    \caption{$t=0$}
    \label{fig:3d-no-bdry-t0}
  \end{subfigure}

  \vspace{0.8em}

  \begin{subfigure}{0.95\linewidth}
    \centering
    \includegraphics[width=\linewidth]{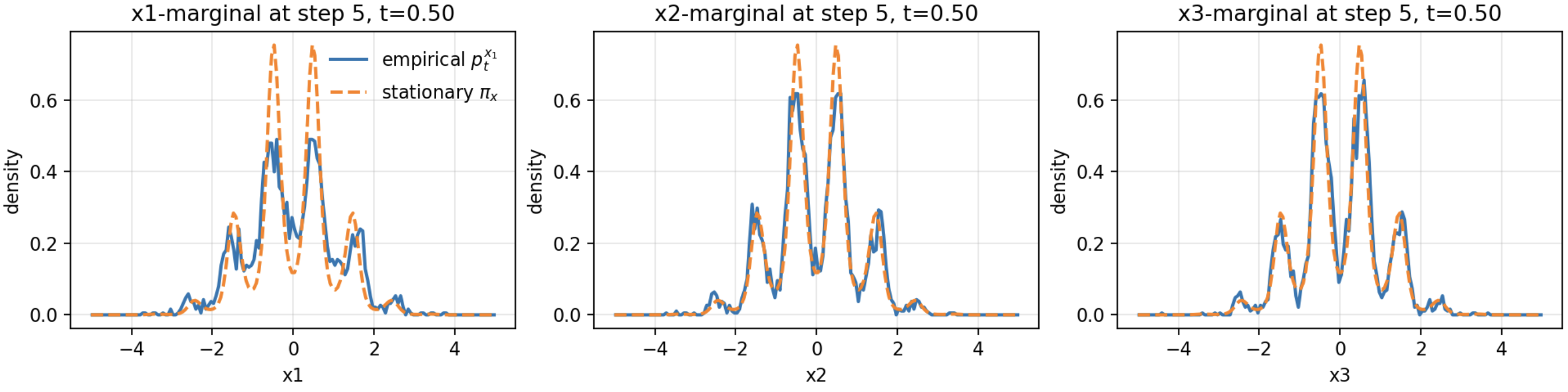}
    \caption{$t=0.5$}
    \label{fig:3d-no-bdry-t05}
  \end{subfigure}

  \vspace{0.8em}

  \begin{subfigure}{0.95\linewidth}
    \centering
    \includegraphics[width=\linewidth]{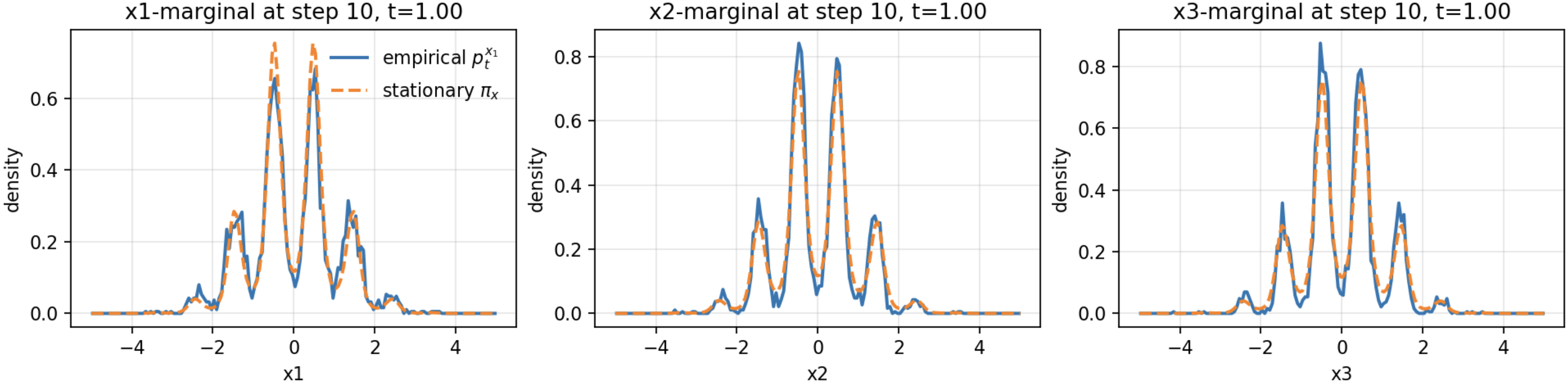}
    \caption{$t=1$}
    \label{fig:3d-no-bdry-t1}
  \end{subfigure}

  \caption{Evolution of coordinate marginals for a 3D position variable $x=(x_1,x_2,x_3)$ with 3D velocity $v$, computed with time step $\Delta t=0.1$. Each panel shows the one–dimensional marginals of $x_1$, $x_2$, and $x_3$ at the indicated time: (\subref{fig:3d-no-bdry-t0}) $t=0$, (\subref{fig:3d-no-bdry-t05}) $t=0.5$, (\subref{fig:3d-no-bdry-t1}) $t=1$.}
  \label{fig:3d-no-bdry-marginals}
\end{figure}

\begin{figure}[ht]
    \centering
    \begin{subfigure}{0.48\linewidth}
        \centering
        \includegraphics[width=\linewidth]{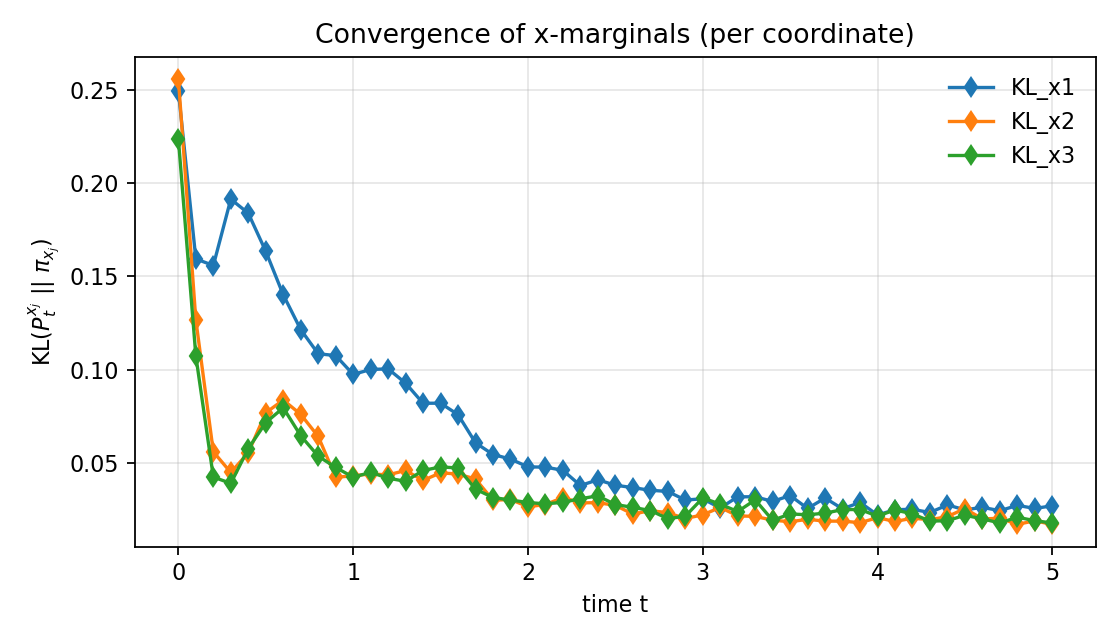}
        \caption{KL divergence over time $x$ marginals}
    \end{subfigure}
    \begin{subfigure}{0.48\linewidth}
        \centering
        \includegraphics[width=\linewidth]{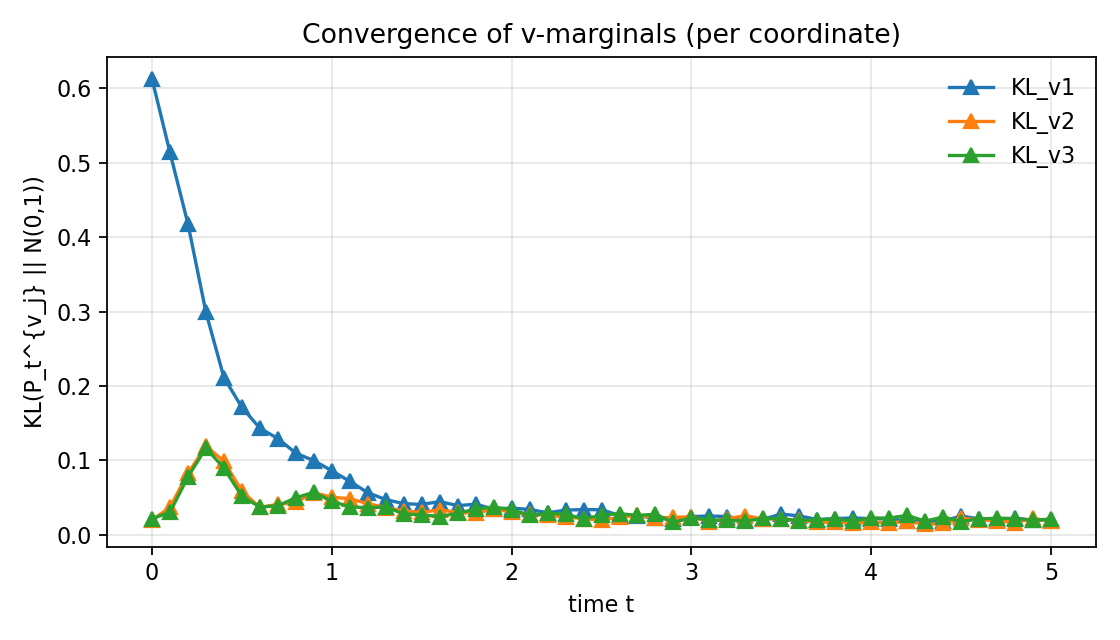}
        \caption{KL divergence over time $v$ marginals}
    \end{subfigure}
    \caption{Convergence of the Kullback--Leibler divergence, $x$ marginals, and $v$ marginals to the stationary distribution for $t \in [0,5]$ under the particle JKO dynamics, shown for Experiment~4.}
    \label{fig:example3-quadratic-marginals-KL-no-bdry}
\end{figure}

\subsection{Vlasov–Poisson–Fokker–Planck systems (PIC–JKO method)}\label{sec:vafp}

We study a 1D in space and 1D in velocity Vlasov–Poisson–Fokker–Planck (VPFP) system on a periodic domain. The dynamics are computed using a particle-in-cell (PIC) discretization for the kinetic density and a spectral Poisson solver for the self-consistent electric field. The evolution follows the proximal formulation described in~\Cref{sec:pic_jko}, where each JKO step is implemented through a neural control field that is trained at every time step.

The spatial domain is $x \in [0,8\pi)$ with periodic boundary conditions, discretized by $N_x = 128$ grid points with spacing $\Delta x = 8\pi / N_x$. Cell centers are given by $x_i = (i + \tfrac{1}{2})\Delta x$ for $i = 0, \dots, N_x-1$. The velocity variable $v \in \mathbb{R}$ remains unbounded in the dynamics. A time step of $\Delta t = 10^{-2}$ is used, and the simulation advances for $n_{\mathrm{steps}} = 200$ iterations, yielding a final time $T = 4$.

The distribution function $f(t,x,v)$ is represented by $N_p = 50,000$ equal-weight particles $(\bx_p, \bv_p, \omega)$. The initial spatial density is
\[
\rho^0(x) = \tfrac{1}{8\pi}(1 + \alpha \cos(kx))\,, \qquad \alpha = 0.1\,, \quad k = 0.2\,,
\]
and the initial phase-space distribution is
\[
f^{0}(x,v)
=
\frac{\rho^0(x)}{2\sqrt{2\pi}\,0.1}
\left(
e^{-\frac{(v-0.3)^2}{2(0.1)^2}}
+
e^{-\frac{(v+0.3)^2}{2(0.1)^2}}
\right)\,,
\]
where the marginal in $v$ is a symmetric mixture of Gaussians centered at
$\pm 0.3$ with standard deviation $0.1$ (variance $0.01$).

To sample particles from the nonuniform spatial profile $\rho^0(x)$, we use inverse transform sampling. The density is first evaluated on the discrete grid $\{x_i\}$ and normalized to form a probability mass function $p_i = \rho^0(x_i) / \sum_j \rho^0(x_j)$. The corresponding cumulative distribution function $P_i = \sum_{j \le i} p_j$ maps uniform random variables $u \sim \mathrm{Uniform}(0,1)$ to grid indices $i$ satisfying $P_{i-1} < u \le P_i$, thereby producing samples consistent with $\rho^0(x)$. Each chosen grid point $x_i$ is then perturbed by a small random displacement within $[-\tfrac{1}{2}\Delta x, \tfrac{1}{2}\Delta x]$ to remove grid bias. The velocity samples are drawn independently from a normal distribution $v_p \sim \mathcal N(0, T_0)$ with $T_0 = 1$.

Each particle contributes to the grid density through a piecewise-linear basis function
\[
S(x) = \max\!\big(0,\, 1 - |x|/\Delta x\big)\,,
\]
which spreads the particle weight over neighboring grid points. The grid density at cell $h$ is computed as
\[
\rho_h = \frac{1}{\Delta x}  \sum_{p=1}^{N}\frac{1}{N}\, S(x_h - x_p)\,.
\]
At every time step, we solve the periodic Poisson equation
\[
-\partial_{xx}\phi = \rho - 1\,, \qquad \bE = -\partial_x \phi\,,
\]
using a spectral method. The discrete Fourier frequencies are
$k_j = 2\pi\,\mathrm{fftfreq}(N_x, \Delta x)$\,,
and the solution is obtained as
\[
\widehat{\phi}(k) = \frac{\widehat{(\rho-1)}(k)}{k^2}\quad (k \neq 0)\,, \qquad
\widehat{\phi}(0) = 0\,, \qquad
\widehat{\bE}(k) = -\,\mathrm{i}k\,\widehat{\phi}(k)\,,
\]
followed by an inverse transform to recover $\phi$ and $\bE$ in physical space. The electric field is then interpolated to particle locations using the same linear basis function to ensure consistency between deposition and interpolation.

Let $t^n = n\Delta t$. One JKO update advances particles according to
\[
\bx_p^{n+1} = \mathrm{wrap}_x(\bx_p^n + \bv_p^n \Delta t)\,, \qquad
\bv_p^{n+1} =\bv_p^n + \bE^{n+1}(\bx_p^{n+1}) \Delta t + u_\theta(\bx_p^n, \bv_p^n) \Delta t\,,
\]
where $\mathrm{wrap}_x(x) = (x \bmod 8\pi)$ enforces periodicity.
The control field $u_\theta:\mathbb{R}^2 \to \mathbb{R}$ is a feedforward neural network with inputs $(\cos(x/2), \sin(x/2), v)$, two hidden layers of width 32 with \texttt{LeakyReLU} activations, and a linear output layer.
The mapping $T(x,v) = (x + v\Delta t,\, v + \bE^{n+1}(x + v\Delta t)\Delta t + u_\theta(x,v)\Delta t)$
defines the discrete transport, and its Jacobian determinant is computed using automatic differentiation.

At each time step and for a fixed parameter $\epsilon>0$, the network parameters $\theta$ are optimized by minimizing
{\[
\mathcal{J}(\theta)
= \frac{\Delta t}{2}\,\sum_{p=1}^{N_p}[u_\theta(\bx_p^n,\bv_p^n)^2]
+ \epsilon\,\sum_{p=1}^{N_p}\!\big[\tfrac{1}{2}(\bv_p^{n+1})^2 + T_0 \log f_p^{n+1}\big]
+ \sum_i (\bE_i^{n+1})^2 \Delta x\,,
\]}
where the expectations are empirical averages over the particles. Each minimization uses 50 Adam iterations with a learning rate of $10^ {-3}$, after which the optimized control is applied to advance the system.

The field energy
$
\mathcal{E}_{\mathrm{field}}(t) = \sum_i \bE_i^2 \Delta x
$
is recorded at every step, and a phase-space histogram $H_{i,j}$ estimating $f(t, x_i, v_j)$ is saved every ten steps. Two parameter regimes are studied, $\epsilon \in \{10, 10^{-2}\}$, with all other settings fixed:
\[
\begin{gathered}
N_x = 128\,,\quad N_p = 50{,}000\,,\quad \Delta t = 10^{-2}\,,\quad n_{\mathrm{steps}} = 200\,,\quad N_{\mathrm{opt}} = 50\,,\quad \text{lr} = 10^{-3}\,.
\end{gathered}
\]

For the weakly collisional case $\epsilon = 5 \times 10^{-3}$, the field energy $\int |\bE|^2\,\rd x$ exhibits mild oscillations paired with a slow overall decay (Figure~\ref{fig:field-energy-two}, left). The initial transient shows a rapid drop followed by a substantial peak near $t \approx 0.8$, driven by the stretching of the particle distribution and weak damping. After $t \approx 1$, the energy fluctuates persistently between $10^{-3}$ and $10^{-2}$. For the strongly collisional case $\epsilon = 10$, the field energy exhibits a highly noisy but rapid overall decay (Figure~\ref{fig:field-energy-two}, right), dropping below $10^{-3}$ by $t \approx 2$ and approaching $10^{-4}$ by $t=4$. While the decay is not strictly steady due to high-frequency fluctuations, the clear downward trajectory demonstrates that the larger regularization parameter enhances velocity diffusion. 

Figure~\ref{fig:phase-evolution-two} further illustrates these contrasting behaviors by tracking the discrete particle positions in phase space at distinct time snapshots. For the weakly collisional case with $\epsilon = 0.005$ (top row), the initial velocity bands roll up to form a large, coherent phase-space vortex by $t = 0.6$. This large-scale structure, characterized by a distinct central empty region, persists even at the final time $t = 4.0$. This persistence shows that the particles are trapped and maintain their organized motion due to the weak damping. Conversely, for the strongly collisional case with $\epsilon = 10$ (bottom row), the initial bands undergo severe stretching and folding by $t = 0.6$, thereby bypassing the formation of a stable vortex. By $t = 4.0$, the particles have scattered significantly and filled the phase-space regions that remained empty in the low-$\epsilon$ scenario. These scatter plots highlight how a stronger collision dictates the behavior by breaking down organized particle structures and rapidly driving the discrete system toward a well-mixed equilibrium.

\begin{figure}[ht!]
  \centering
  \begin{subfigure}[t]{0.49\linewidth}
    \centering
    \includegraphics[width=\linewidth]{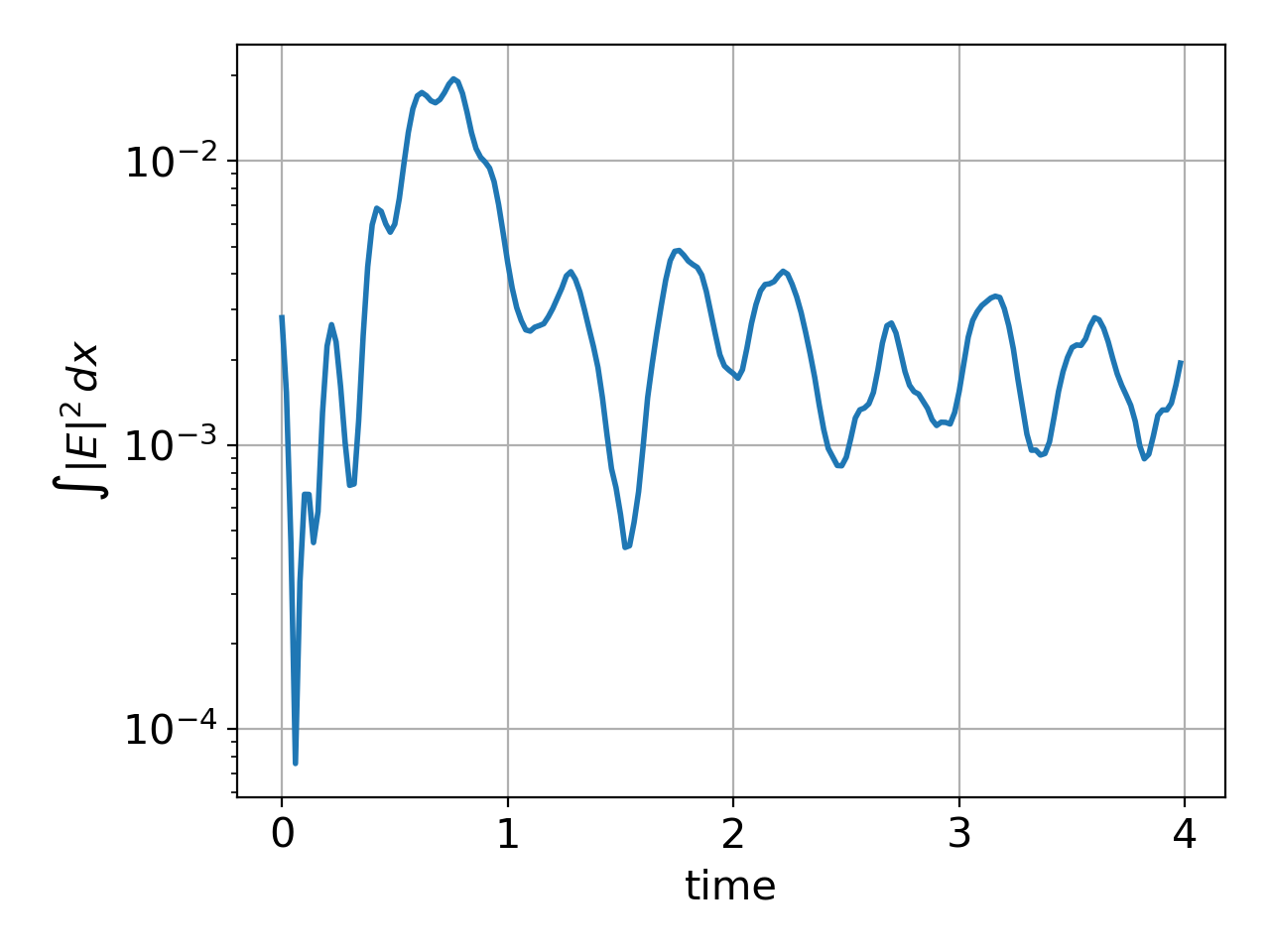}
    \subcaption{$\epsilon=5\times10^{-3}$}
  \end{subfigure}\hfill
  \begin{subfigure}[t]{0.49\linewidth}
    \centering
    \includegraphics[width=\linewidth]{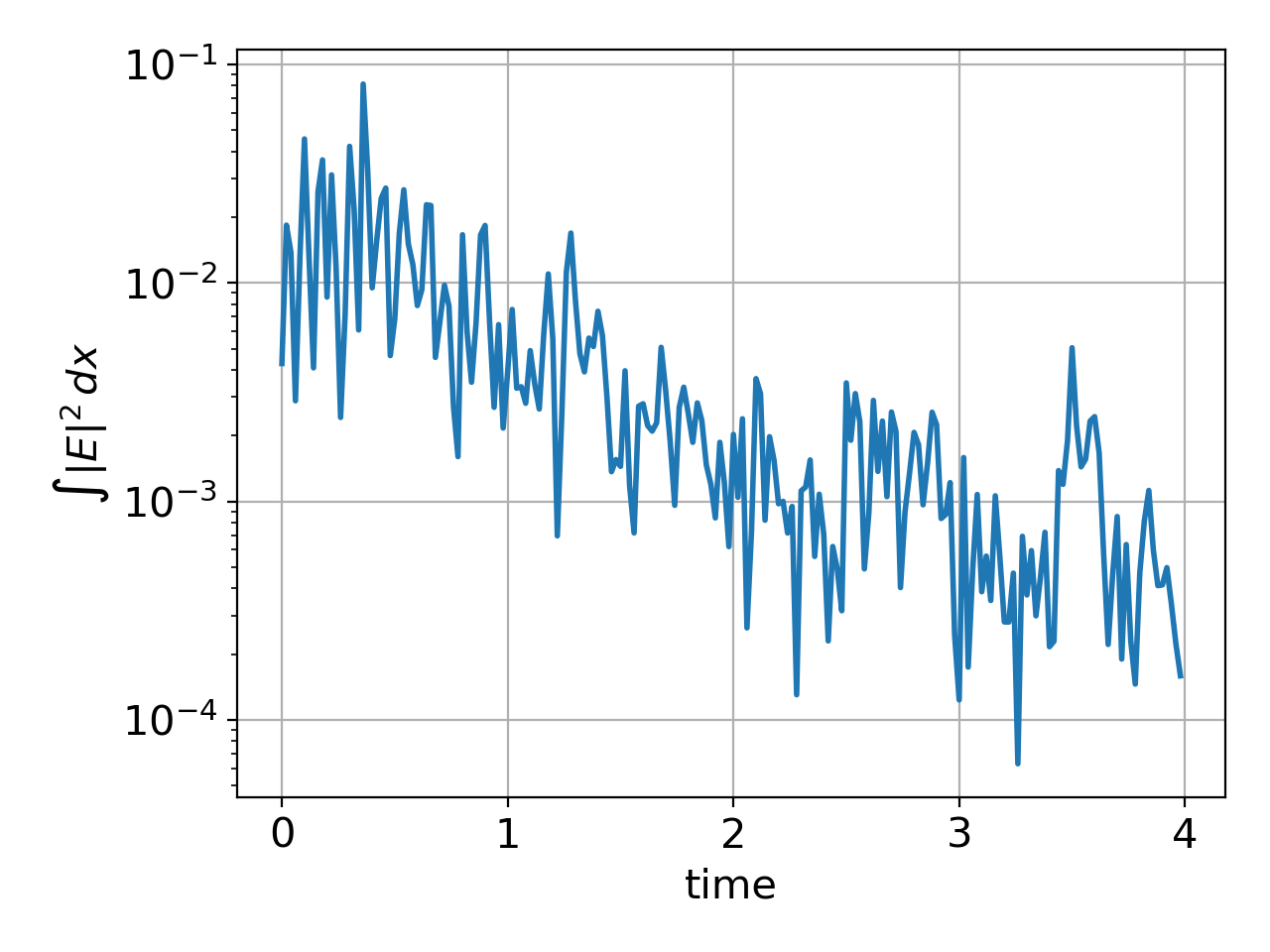}
    \subcaption{$\epsilon=10$}
  \end{subfigure}
  \caption{Experiment from \Cref{sec:vafp}.Comparison of field energy decay $\int |E|^2\,dx$ over time for $\epsilon=5\times10^{-3}$ (left) and $\epsilon=10$ (right).}
  \label{fig:field-energy-two}
\end{figure}

\begin{figure}[ht!]
  \centering

  \begin{subfigure}[t]{0.24\linewidth}
    \centering
    \includegraphics[width=\linewidth]{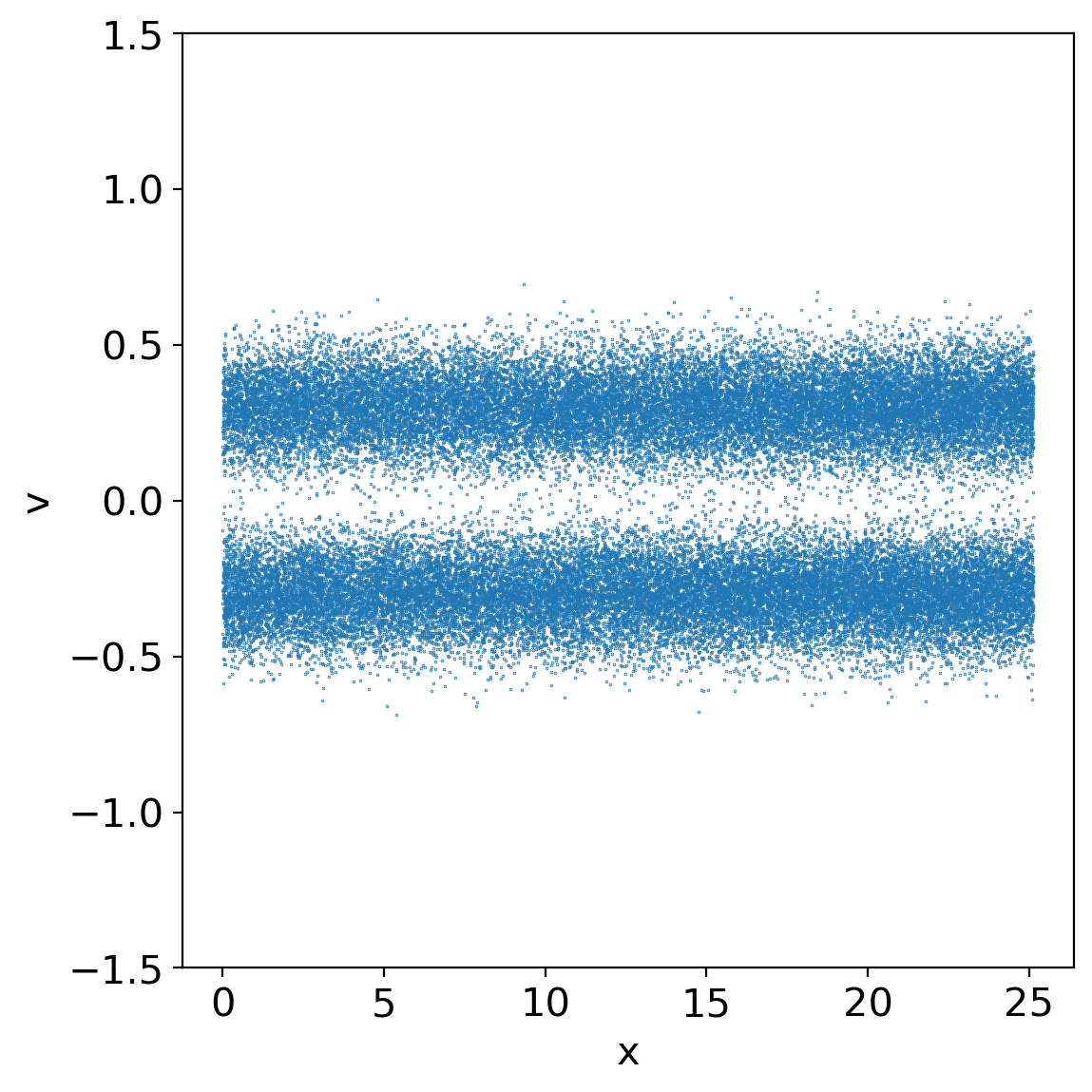}
    \subcaption{$t=0.0$}
  \end{subfigure}\hfill
  \begin{subfigure}[t]{0.24\linewidth}
    \centering
    \includegraphics[width=\linewidth]{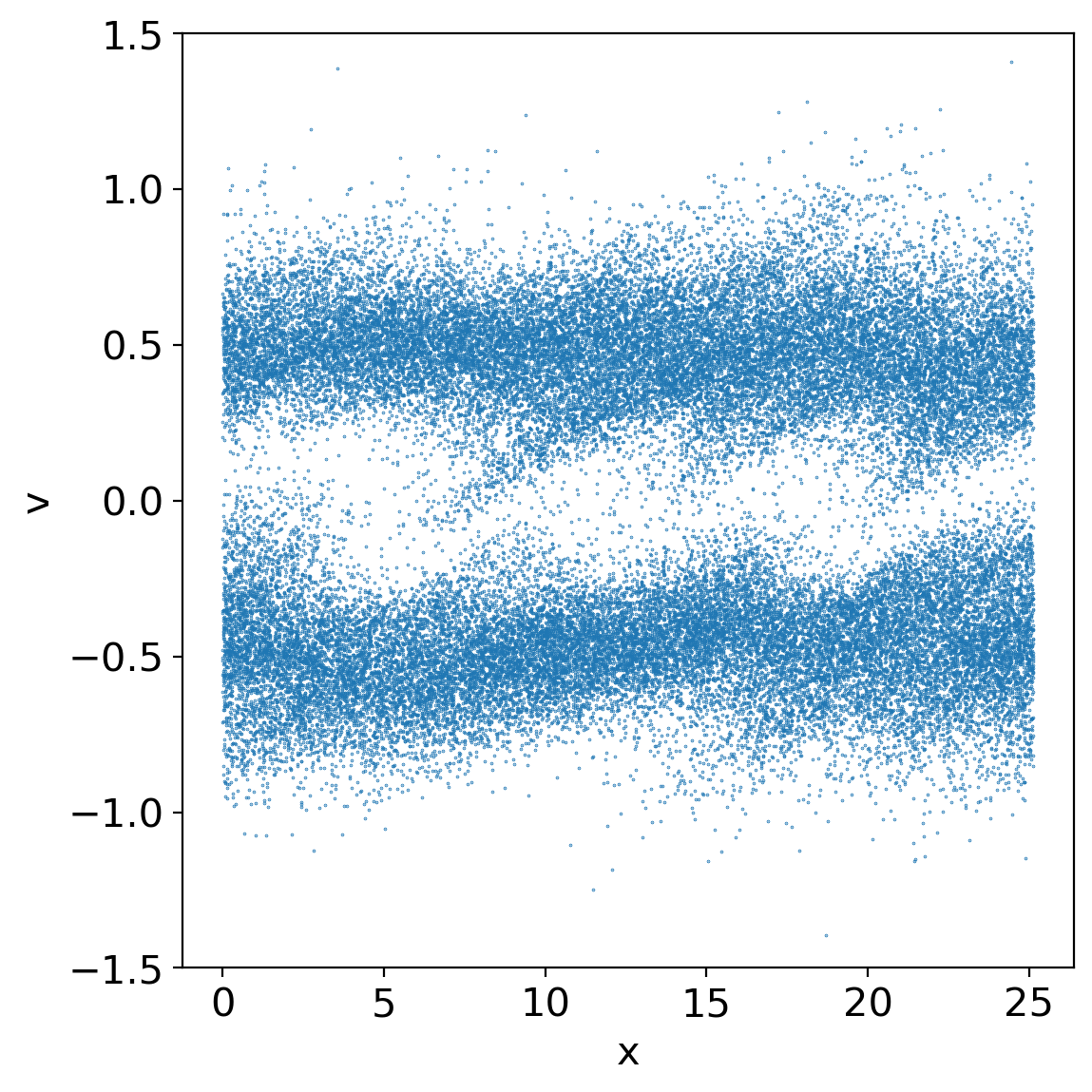}
    \subcaption{$t=0.2$}
  \end{subfigure}\hfill
  \begin{subfigure}[t]{0.24\linewidth}
    \centering
    \includegraphics[width=\linewidth]{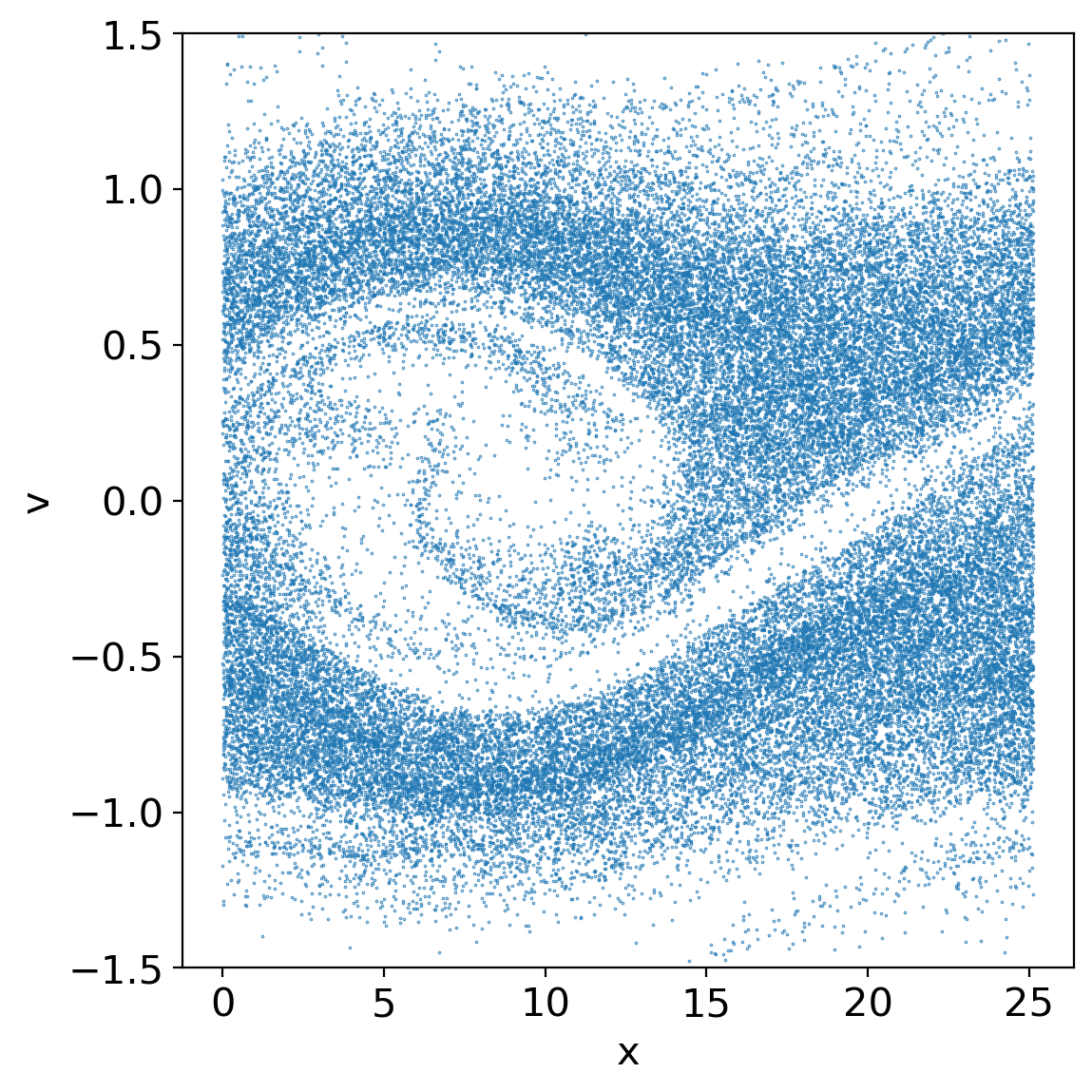}
    \subcaption{$t=0.6$}
  \end{subfigure}\hfill
  \begin{subfigure}[t]{0.24\linewidth}
    \centering
    \includegraphics[width=\linewidth]{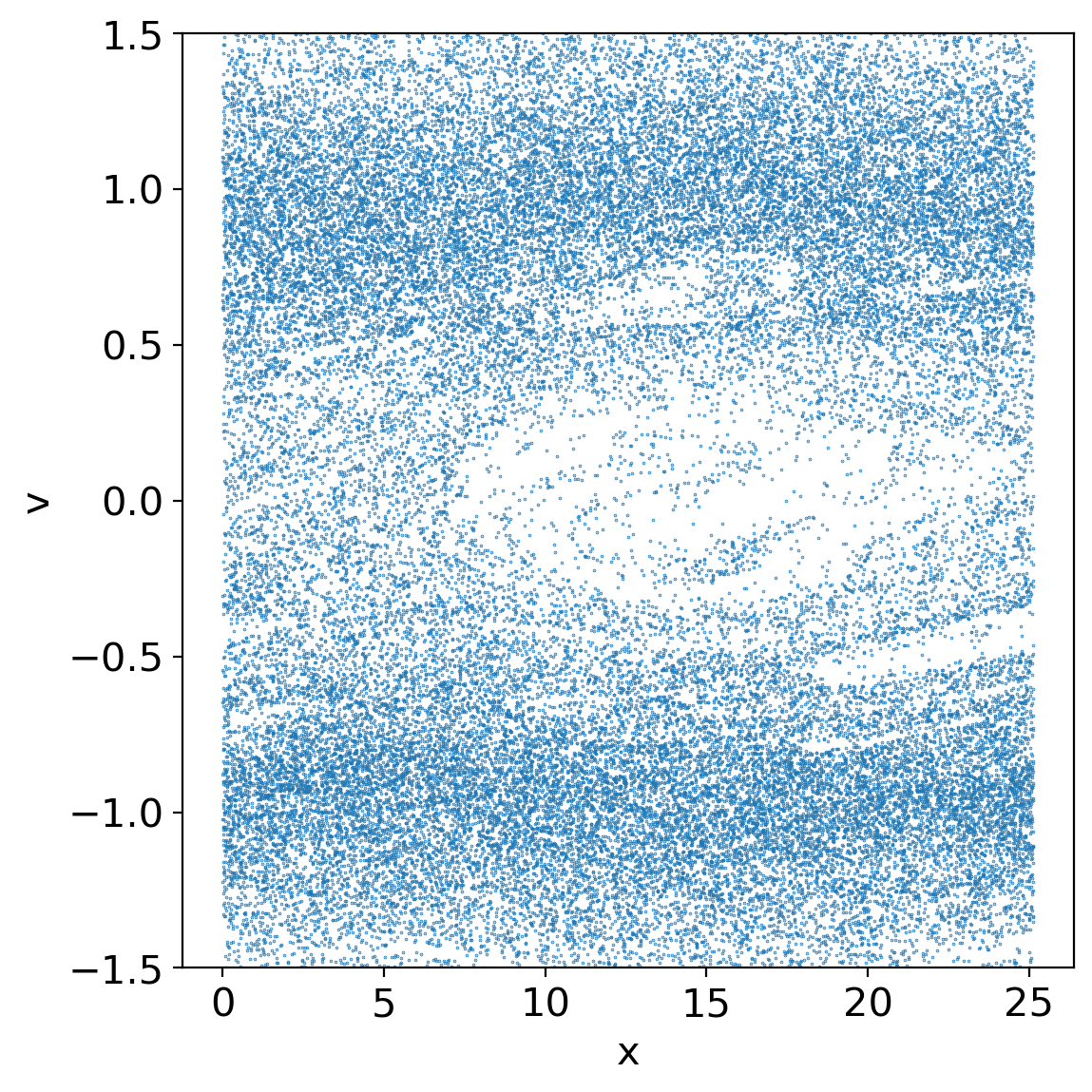}
    \subcaption{$t=4.0$}
  \end{subfigure}\hfill

  \vspace{0.6em}

  \begin{subfigure}[t]{0.24\linewidth}
    \centering
    \includegraphics[width=\linewidth]{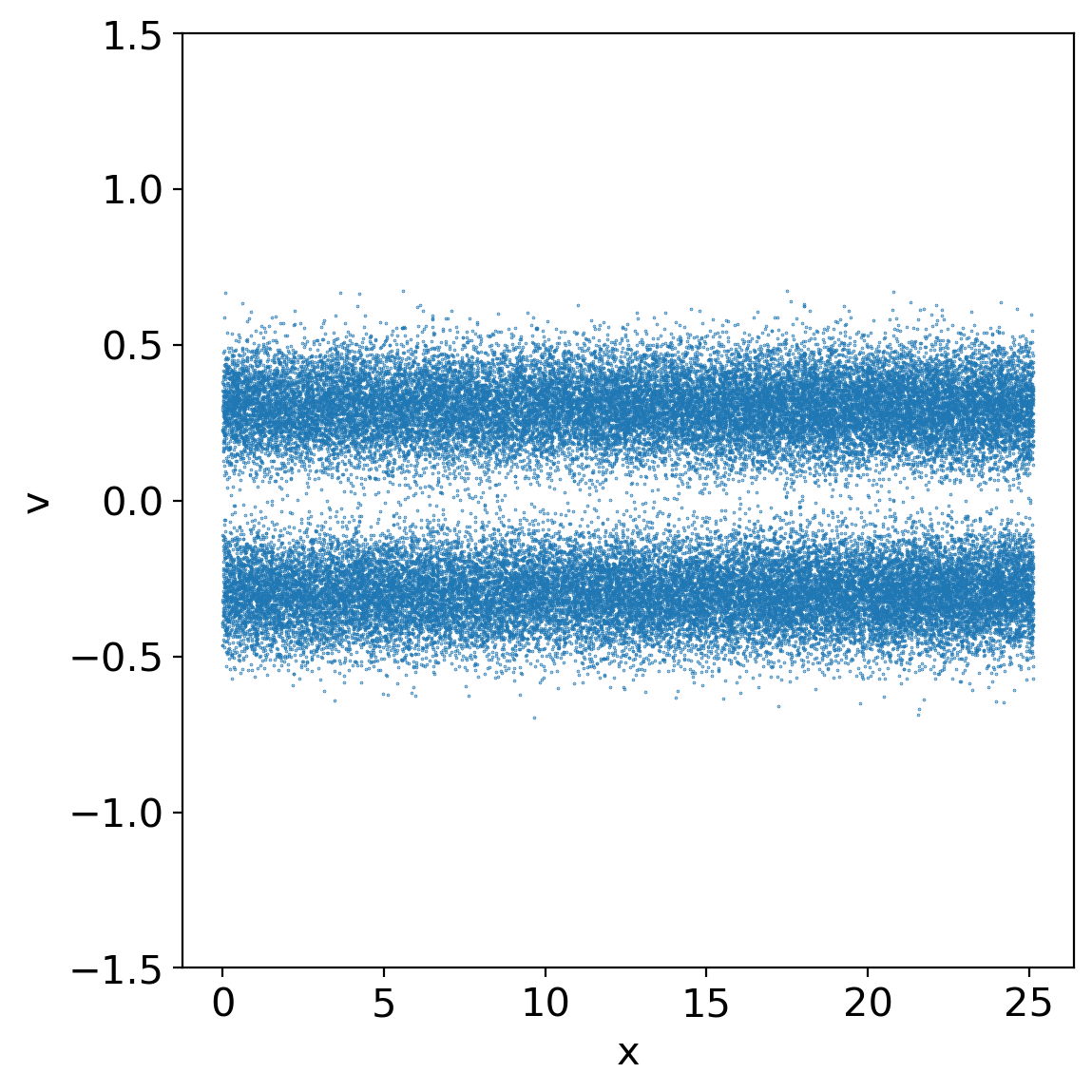}
    \subcaption{$t=0.0$}
  \end{subfigure}\hfill
  \begin{subfigure}[t]{0.24\linewidth}
    \centering
    \includegraphics[width=\linewidth]{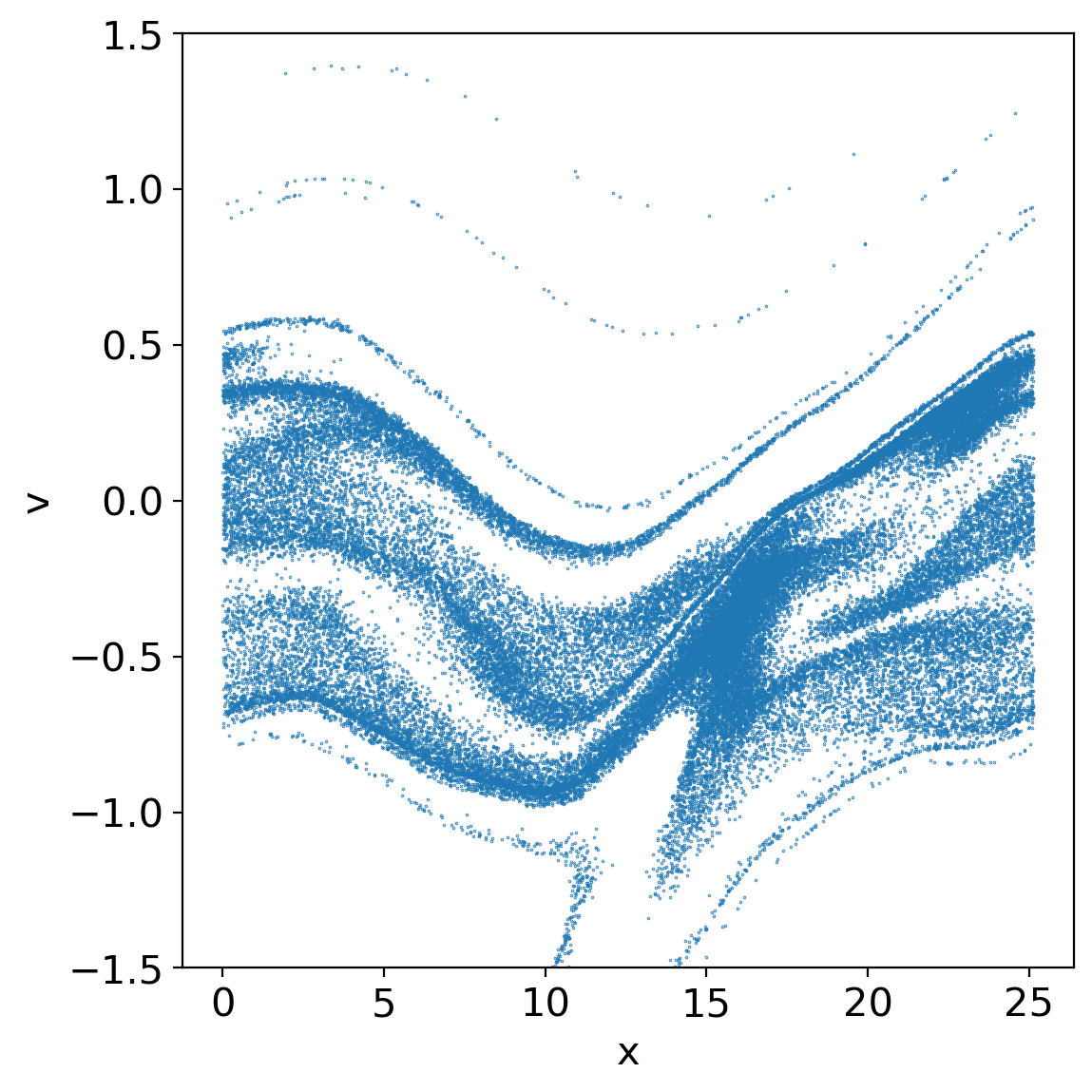}
    \subcaption{$t=0.2$}
  \end{subfigure}\hfill
  \begin{subfigure}[t]{0.24\linewidth}
    \centering
    \includegraphics[width=\linewidth]{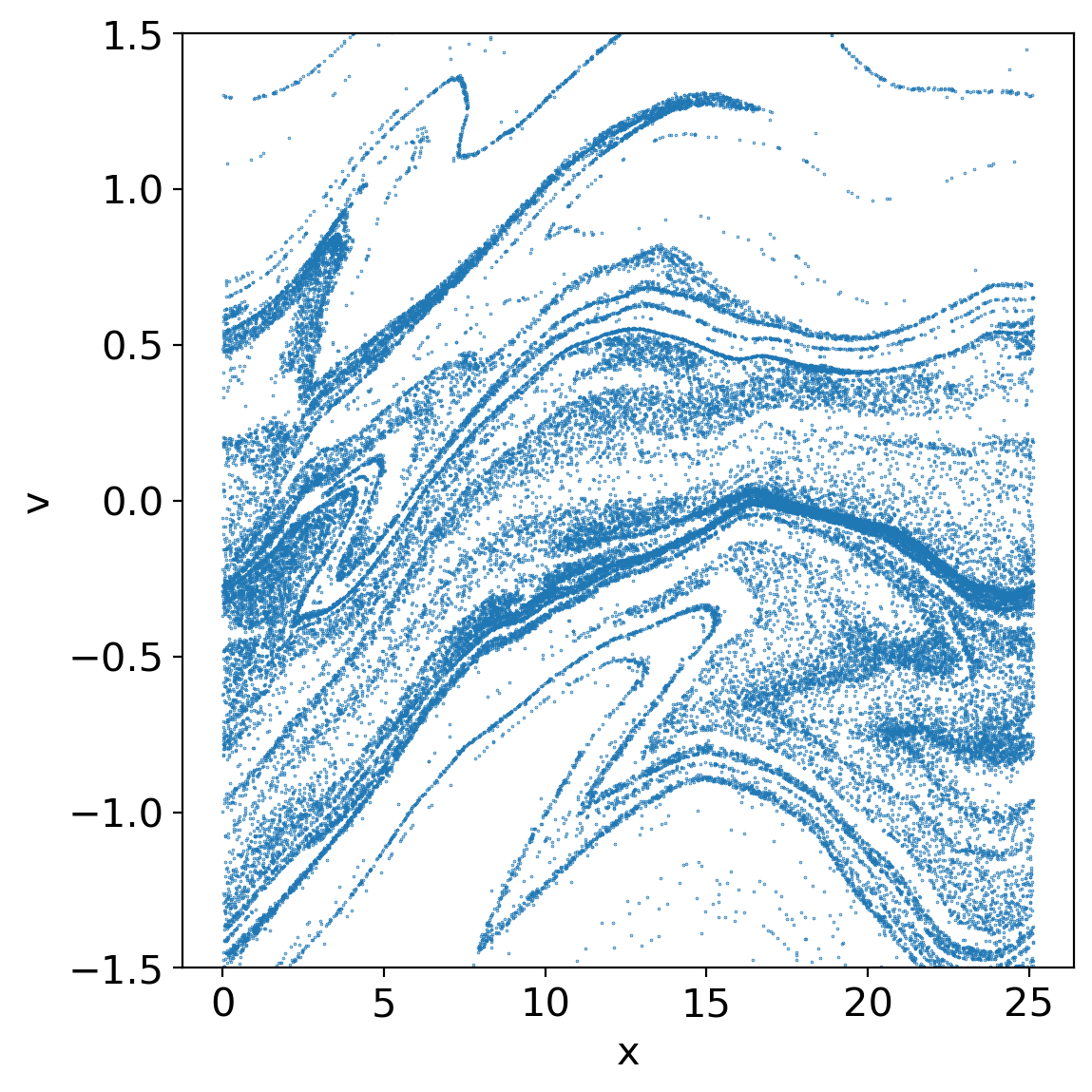}
    \subcaption{$t=0.6$}
  \end{subfigure}\hfill
  \begin{subfigure}[t]{0.24\linewidth}
    \centering
    \includegraphics[width=\linewidth]{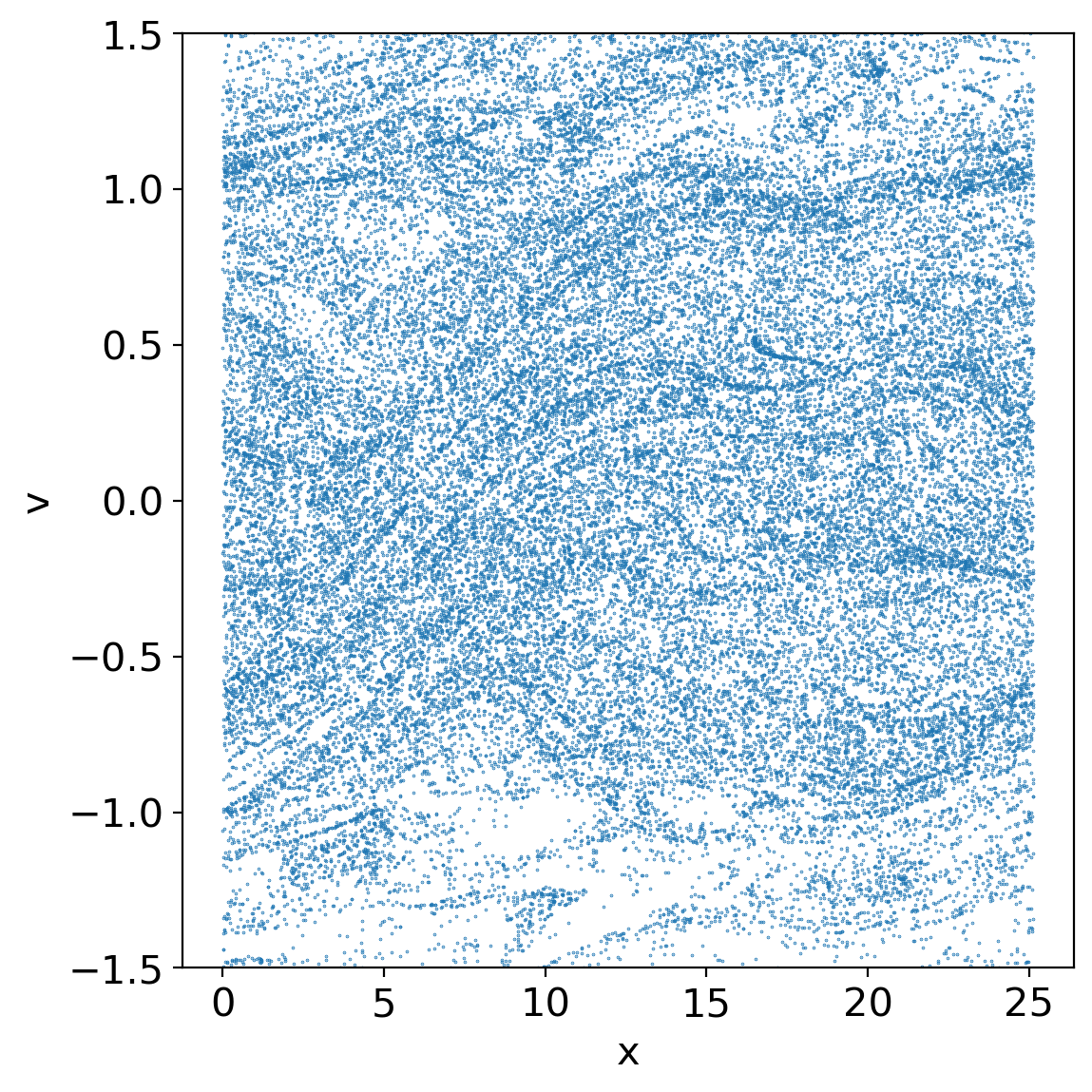}
    \subcaption{$t=4.0$}
  \end{subfigure}\hfill

  \caption{Experiment from \Cref{sec:vafp}. Phase-space density $f(t,x,v)$ for $\epsilon=0.005$ (top row) and $\epsilon=10$ (bottom row) at different times.  }
  \label{fig:phase-evolution-two}
\end{figure}

\bigskip
\paragraph{Extended experiment: $1$D in space and $3$D in velocity.}
To assess the performance of the proposed JKO-based scheme in higher-dimensional phase space, we extend the one-dimensional velocity experiment to a setting with one spatial variable and three velocity components, $x\in[0,8\pi)$ and $v=(v_1,v_2,v_3)\in\mathbb{R}^3$. The spatial domain is periodic, and the electric field is determined by the one-dimensional Poisson equation in $x$, so that the field acts only on the first velocity component $v_1$.

The initial distribution is chosen as
\begin{equation} \label{0321}
    f^{0}(x,v) = \frac{\rho^{0}(x)}{2(2\pi\sigma^2)^{3/2}} \left( e^{-\frac{(v_1-1.5)^2}{2\sigma^2}} + e^{-\frac{(v_1+1.5)^2}{2\sigma^2}} \right) e^{-\frac{v_2^2 + v_3^2}{2\sigma^2}}\,, \qquad \sigma=0.1. 
\end{equation}
The spatial marginal is prescribed by
\begin{equation*}
\rho^{0}(x)
=
\tfrac{1}{8\pi}\bigl(1+0.005\cos(2\pi x)\bigr)\,,
\end{equation*}
so that $\int_{\mathbb{R}^3} f^{0}(x,v)\,\rd \bv=\rho^{0}(x)$.

This configuration represents two counter-propagating beams in the $v_1$ direction with small thermal spread, combined with isotropic Gaussian fluctuations in the transverse velocity components $v_2$ and $v_3$. Compared with the $1D$x--$1D$v experiment, the phase space dimension increases from two to four, providing a nontrivial test of the scalability of the method.

Initial particle positions are sampled according to the discrete approximation of $\rho^{0}(x)$ on the spatial grid, followed by uniform jitter within each cell. Velocities are drawn from the mixture distribution in~\eqref{0321}. The initial particle weights are assigned according to $f^{0}(x,v)$, ensuring consistency with the prescribed density.

The neural control field $\bu_\theta(x,v)\in\mathbb{R}^3$ is approximated by a fully connected feedforward neural network. To respect the periodicity in the spatial variable, the input layer consists of the five features
\[
\bigl(\cos(x/4),\;\sin(x/4),\;v_1,\;v_2,\;v_3\bigr)\,,
\]
which provides a smooth embedding of the spatial coordinate and the three velocity components. The network contains two hidden layers with $64$ neurons each and \texttt{LeakyReLU} activation functions, followed by a linear output layer producing a three-dimensional vector corresponding to the velocity components of the control field.

All remaining parameters, including the time step, grid resolution, and optimization settings, are identical to those used in the $1$D--$1$D experiment, except that the terminal time is set to $T=4$ instead of $T=4$ in this experiment. In particular, we consider both weak and strong regularization regimes to facilitate direct comparison.

This extended experiment demonstrates that the proposed algorithm remains stable and accurate in higher-dimensional velocity spaces, while preserving the qualitative kinetic structures observed in lower dimensions. In particular, the weakly regularized case exhibits sustained oscillations, while the strongly regularized case displays rapid monotone decay, consistent with the behavior observed in the $1$D--$1$D experiment. As in the one-dimensional velocity setting, a coherent vortex structure also emerges in the weakly regularized regime, indicating the onset of nonlinear trapping and beam--beam interaction.

Figure~\ref{fig:phase-evolution-two-3d} illustrates the evolution of the phase-space density projected onto the $(x,v_1)$ plane for both $\epsilon=0.005$ and $\epsilon=10$. For $\epsilon=0.005$, fine filamentary structures, beam interactions, and the formation of a rotating vortex persist over time, whereas for $\epsilon=10$ the distribution rapidly smooths and becomes unimodal. Together, these results confirm that the proposed method successfully captures filamentation, beam interaction, vortex formation, and long-time relaxation in the presence of transverse velocity fluctuations, illustrating its robustness and scalability with increasing dimension.

\begin{figure}[ht!]
  \centering

  \begin{subfigure}[t]{0.24\linewidth}
    \centering
    \includegraphics[width=\linewidth]{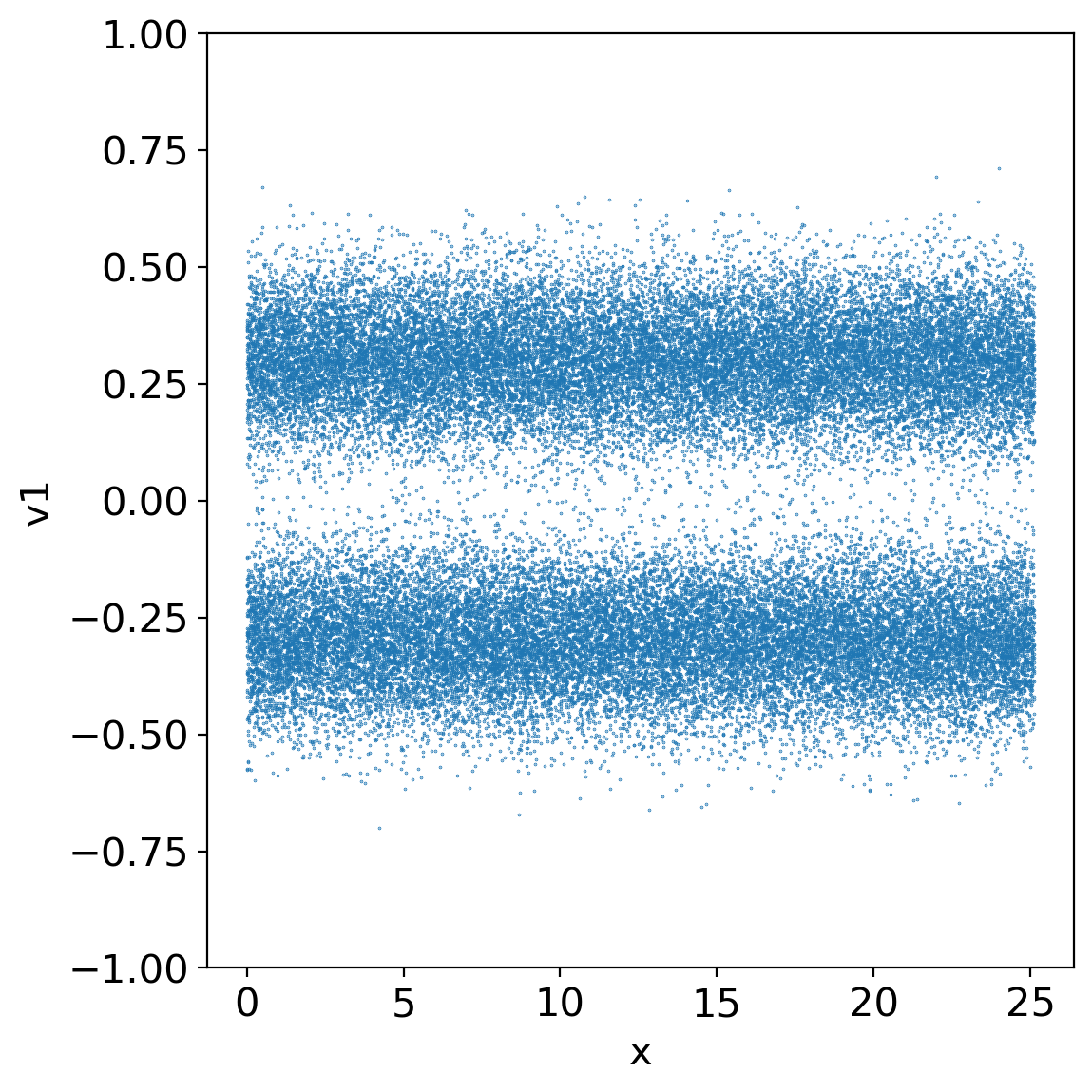}
    \subcaption{$t=0.0$}
  \end{subfigure}\hfill
  \begin{subfigure}[t]{0.24\linewidth}
    \centering
    \includegraphics[width=\linewidth]{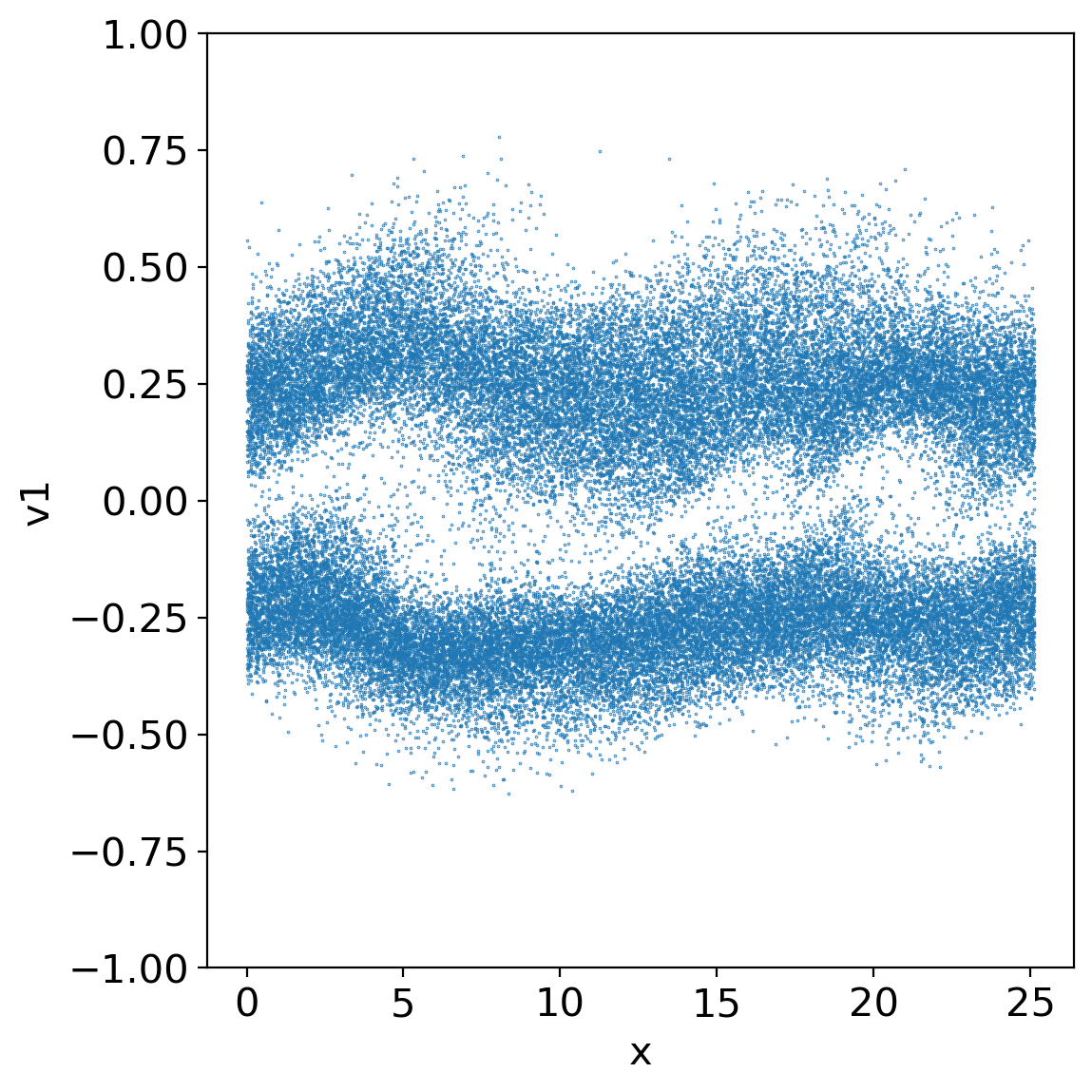}
    \subcaption{$t=0.4$}
  \end{subfigure}\hfill
  \begin{subfigure}[t]{0.24\linewidth}
    \centering
    \includegraphics[width=\linewidth]{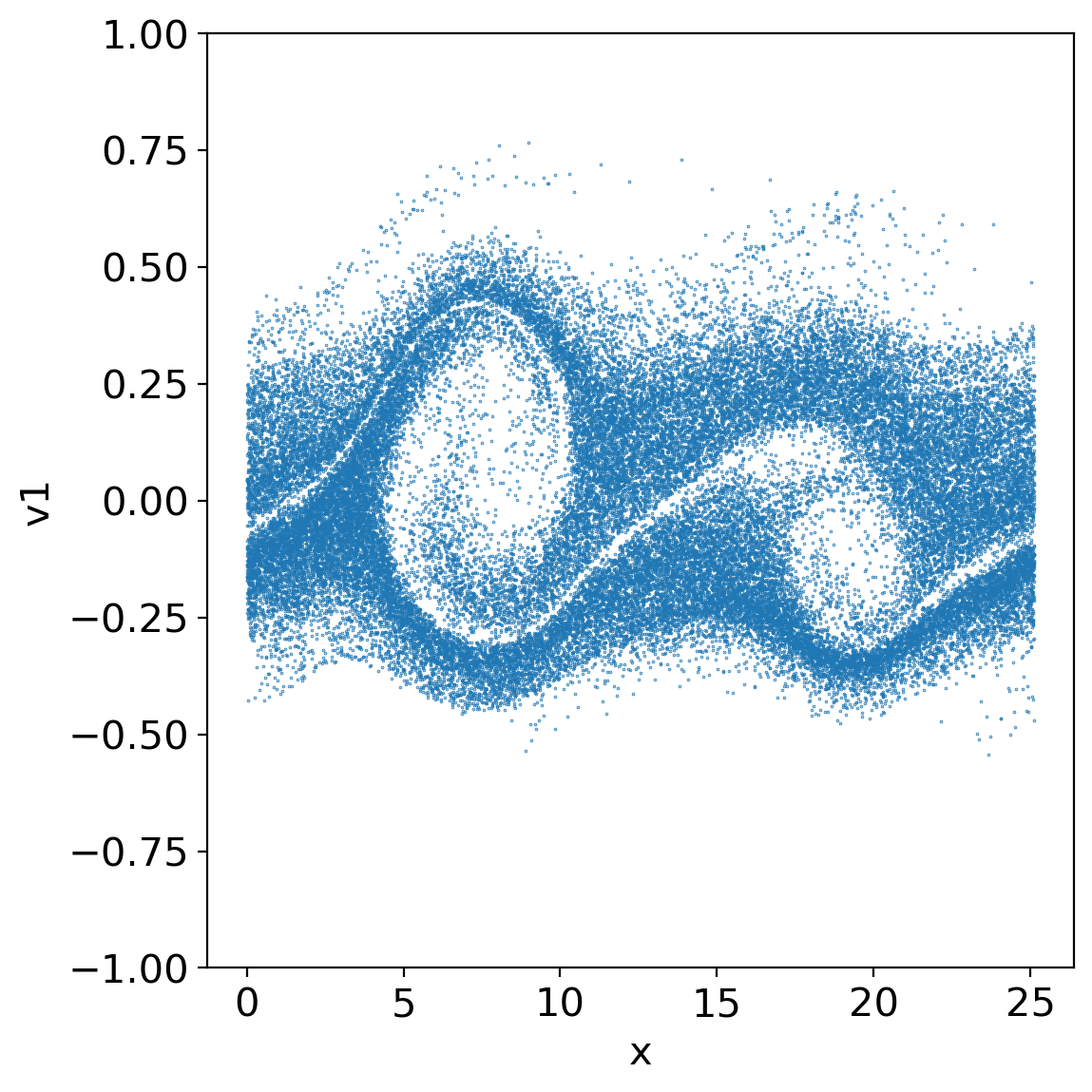}
    \subcaption{$t=1.2$}
  \end{subfigure}\hfill
  \begin{subfigure}[t]{0.24\linewidth}
    \centering
    \includegraphics[width=\linewidth]{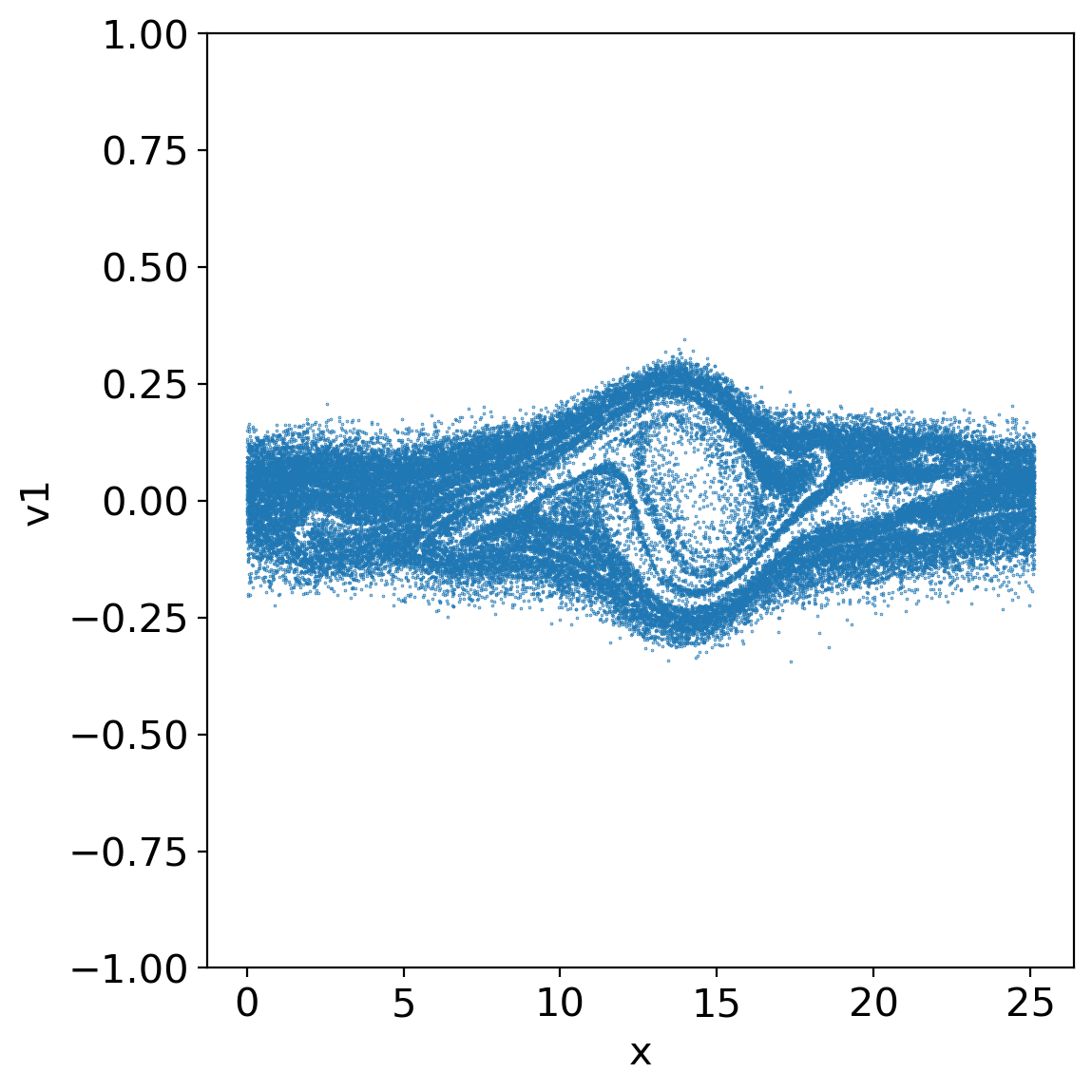}
    \subcaption{$t=4.0$}
  \end{subfigure}\hfill

  \vspace{0.6em}

  \begin{subfigure}[t]{0.24\linewidth}
    \centering
    \includegraphics[width=\linewidth]{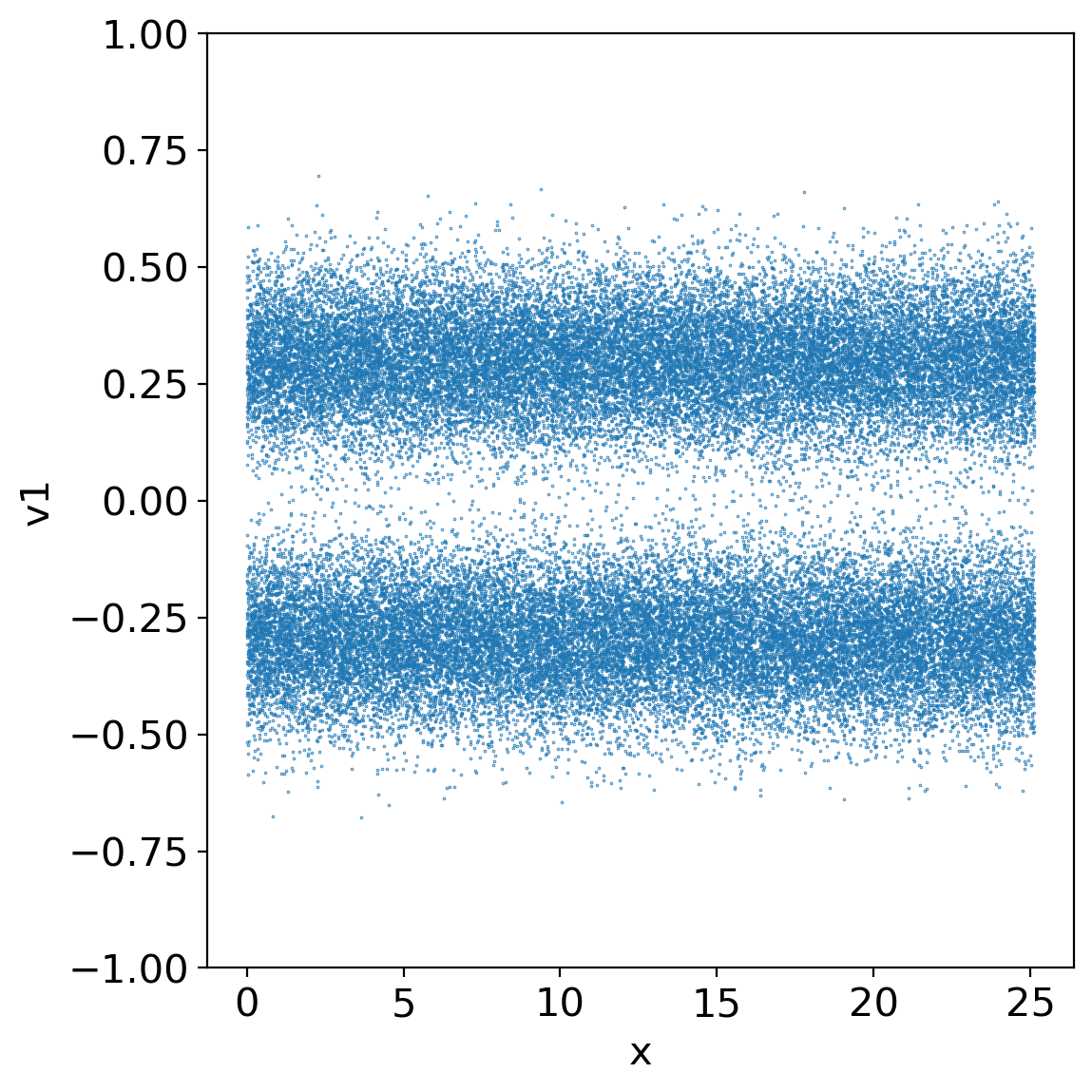}
    \subcaption{$t=0.0$}
  \end{subfigure}\hfill
  \begin{subfigure}[t]{0.24\linewidth}
    \centering
    \includegraphics[width=\linewidth]{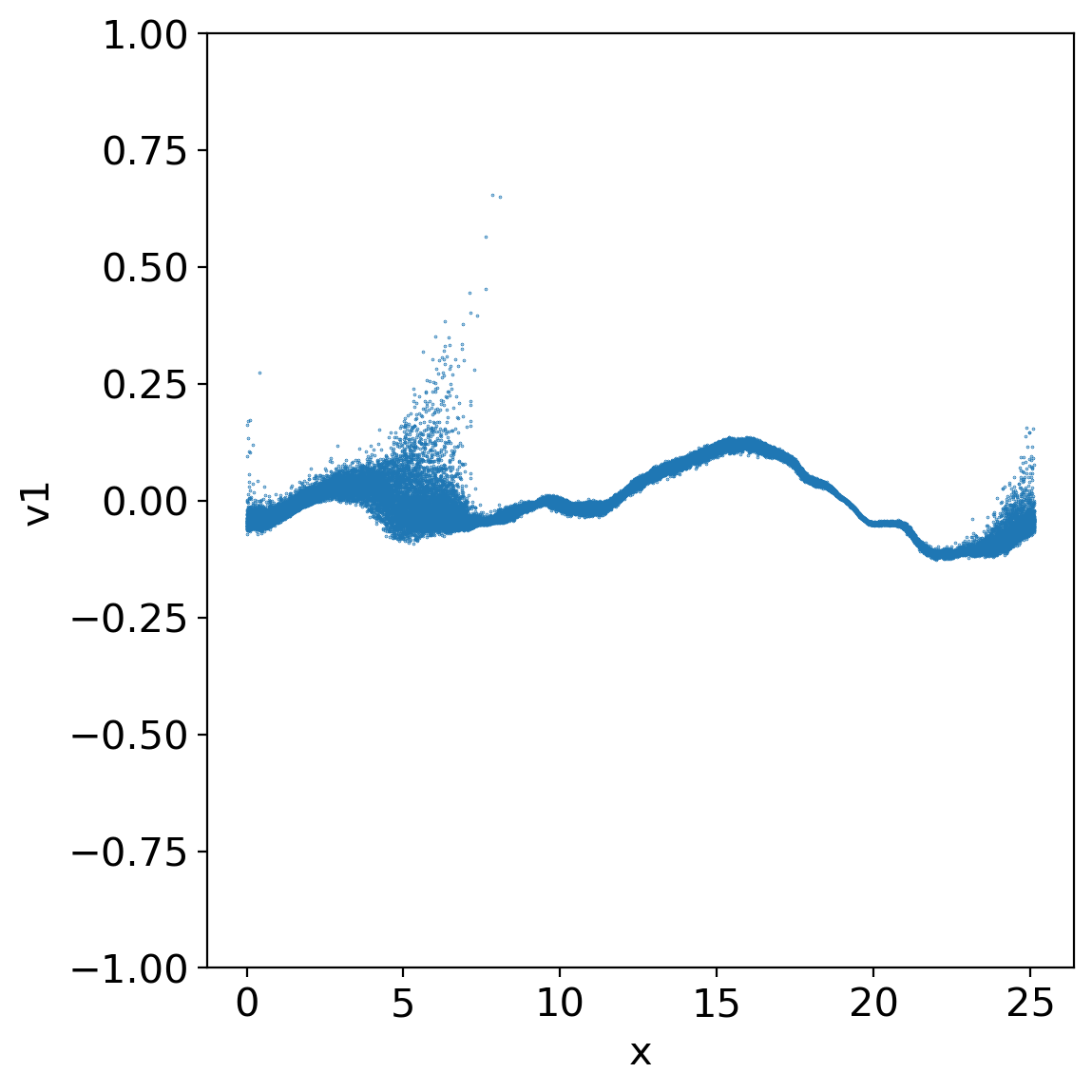}
    \subcaption{$t=0.4$}
  \end{subfigure}\hfill
  \begin{subfigure}[t]{0.24\linewidth}
    \centering
    \includegraphics[width=\linewidth]{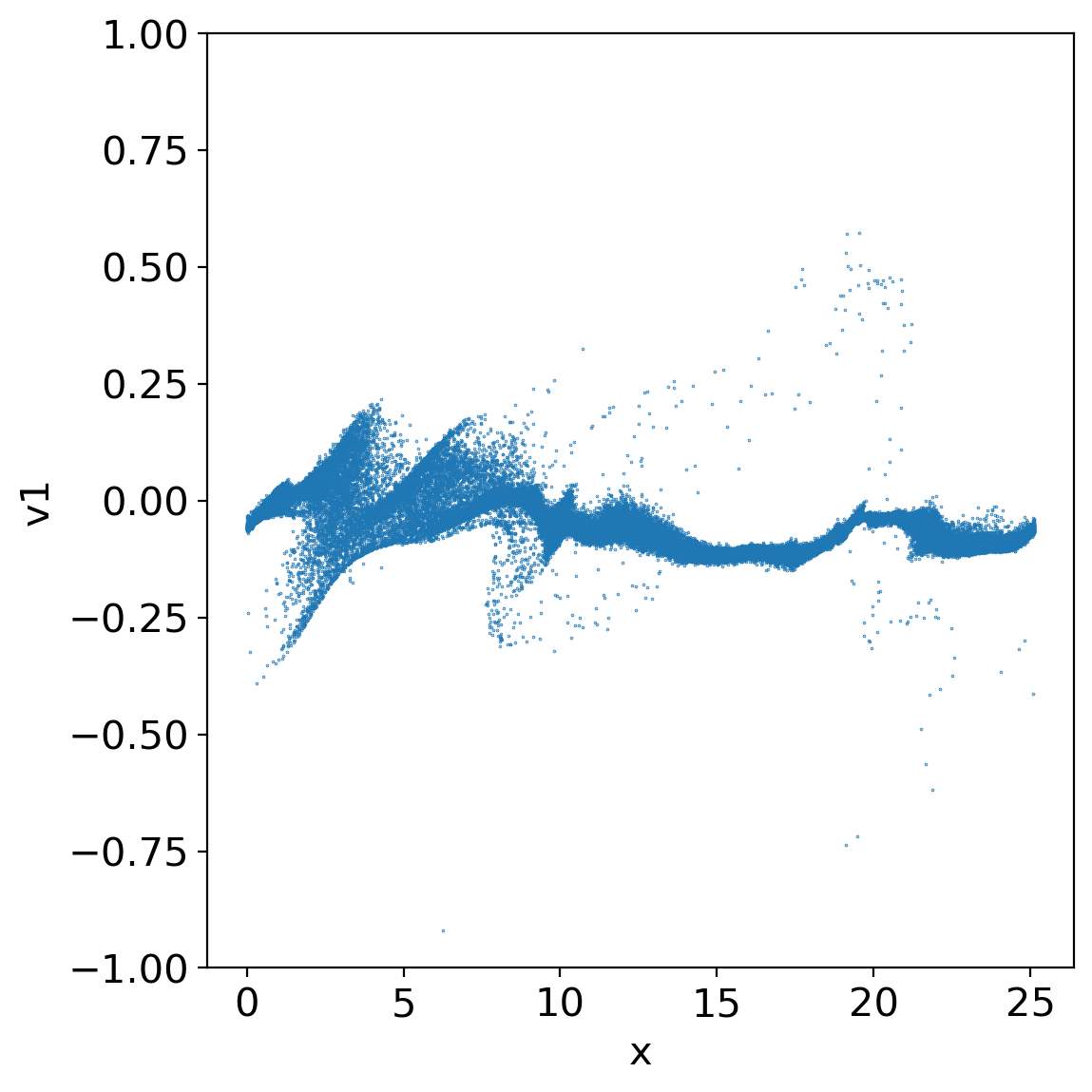}
    \subcaption{$t=1.2$}
  \end{subfigure}\hfill
  \begin{subfigure}[t]{0.24\linewidth}
    \centering
    \includegraphics[width=\linewidth]{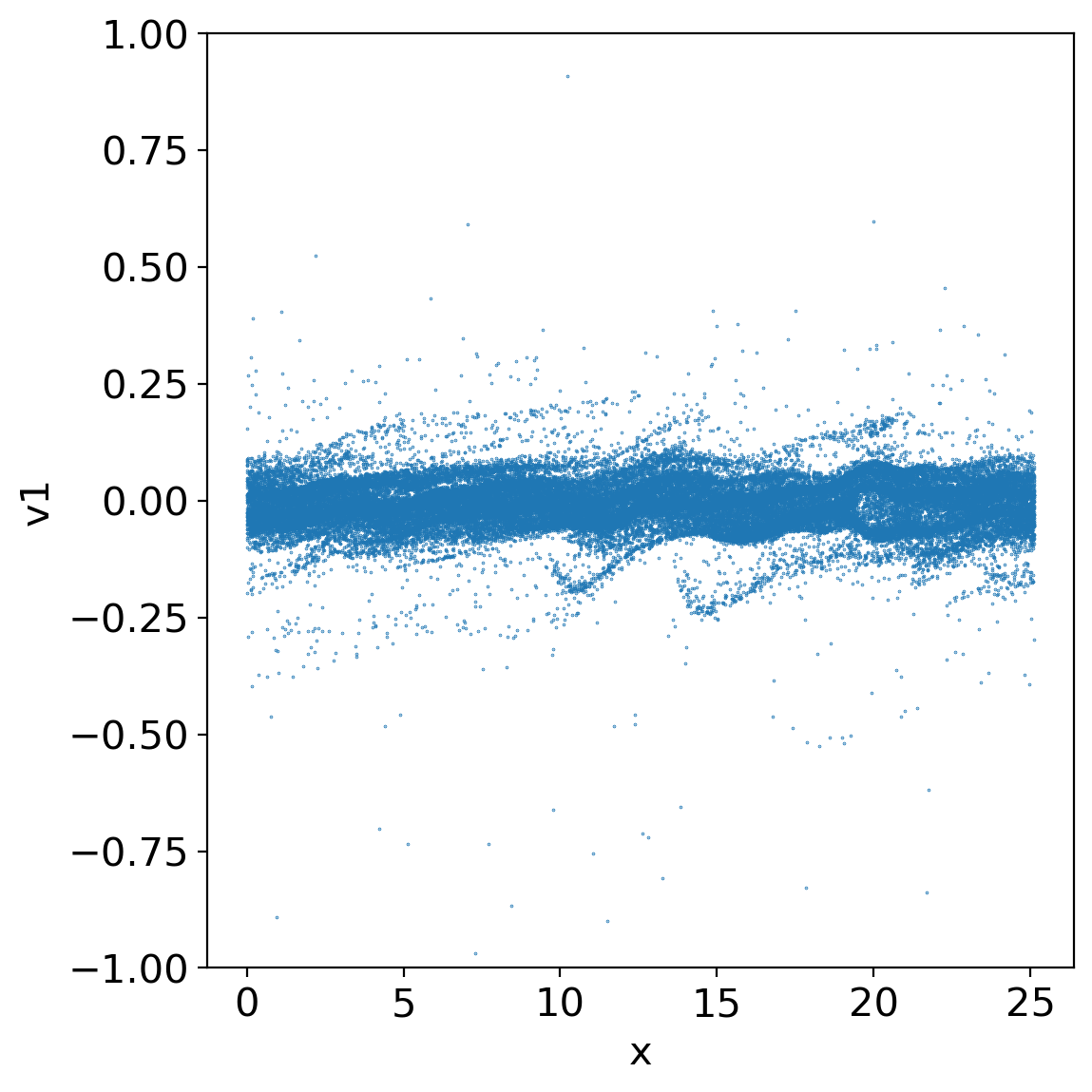}
    \subcaption{$t=4.0$}
  \end{subfigure}\hfill

  \caption{Experiment from \Cref{sec:vafp}. Phase-space density $f(t,x,v_1)$ projected onto the $(x,v_1)$ plane for the $1D$--$3D$ setting, with $\epsilon=0.005$ (top row) and $\epsilon=10$ (bottom row) at different times.
}
  \label{fig:phase-evolution-two-3d}
\end{figure}

\section{Conclusions} \label{sec5}
JKO schemes provide a variational framework for approximating nonlinear gradient flows. In this paper, we extend the classical JKO framework to kinetic equations through a conservative–dissipative decomposition. The resulting kinetic JKO scheme is implemented using particle approximations together with kinetic-oriented neural ODEs. This approach preserves the variational structure of kinetic equations and ensures dissipation of a modified energy functional. Numerical experiments on high-dimensional linear and nonlinear Vlasov–Fokker–Planck systems demonstrate the accuracy and effectiveness of the proposed method.

Future research directions include developing deep neural network–based solvers for more general kinetic equations that exhibit a conservative–dissipative structure. Another important direction is to conduct a detailed numerical analysis of kinetic JKO schemes with neural ODE approximations. Several layers of error analysis warrant careful investigation, including approximation error, optimization error, and sample complexity. It will also be valuable to systematically compare the proposed method with score matching and velocity field matching methods, particularly with respect to their numerical stability and long-time behavior.

\bigskip
\noindent\textbf{Acknowledgment:} 
W. Li is supported by the AFOSR YIP award No. FA9550-23-1-0087, NSF RTG: 2038080, NSF FRG: 2245097, and the McCausland Faculty Fellowship at the University of South Carolina. L. Wang is partially supported by NSF grants DMS-2513336 and Simons Fellowship. L. Wang also thanks Profs. Phillip Morrison and Lukas Einkemmer for fruitful discussions.  W.~Lee acknowledges funding from the National Institute of Standards and Technology under award number 70NANB22H021 and startup funding from The Ohio State University.

\bibliographystyle{siam} 
\bibliography{Var_kinetic_ref.bib}
\end{document}